\documentclass[a4paper, twoside]{scrartcl}
\usepackage[utf8]{inputenc}
\usepackage{amsfonts,amssymb,amsmath,amsthm,amstext,amssymb,mathtools}
\usepackage{subcaption,float,color,xcolor,graphicx,enumitem}
\usepackage{dsfont}
\usepackage[normalem]{ulem}
\usepackage[authoryear,sort,round]{natbib}
\usepackage{fullpage}
\usepackage{hyperref,nicefrac}
\usepackage[noabbrev,capitalise,nosort,nameinlink]{cleveref}

\crefname{assumption}{Assumption}{Assumptions}

\addtokomafont{disposition}{\sffamily}

\usepackage{csquotes}

\usepackage{graphics}
\usepackage{tikz, tikz-cd}
\usetikzlibrary{decorations.pathreplacing,calc}

\DeclareMathOperator*{\dom}{dom}
\DeclareMathOperator*{\domain}{dom}
\DeclareMathOperator*{\tr}{tr}
\DeclareMathOperator*{\trace}{tr}

\DeclareMathOperator*{\argmin}{arg\,min}
\DeclareMathOperator*{\ran}{ran}
\DeclareMathOperator*{\range}{ran}

\usepackage{mathrsfs}
\newcommand*{\Complex}{\mathbb{C}}
\newcommand*{\defeq}{\coloneqq}

\newcommand*{\E}{\mathbb{E}}

\renewcommand*{\epsilon}{\varepsilon}
\newcommand*{\cov}{\operatorname{\mathbb{C}ov}}
\newcommand*{\Naturals}{\mathbb{N}}
\newcommand*{\one}{\mathds{1}}
\newcommand*{\idop}{\operatorname{Id}}
\newcommand*{\prob}{\mathbb{P}}
\newcommand*{\qefed}{\eqqcolon}
\newcommand*{\quark}{\setbox0\hbox{$x$}\hbox to\wd0{\hss$\cdot$\hss}}
\newcommand*{\rd}{\mathrm{d}}

\newcommand*{\Reals}{\mathbb{R}}

\newcommand*{\espace}{\mathop{\textup{eig}}\nolimits}

\newcommand*{\law}{\mathscr{L}}

\newcommand*{\cE}{\mathcal{E}}
\newcommand*{\cH}{\mathcal{H}}
\newcommand*{\cX}{\mathcal{X}}
\newcommand*{\cY}{\mathcal{Y}}

\newcommand*{\Borel}{\mathcal{B}}
\newcommand*{\sigalg}{\mathcal{F}}

\newcommand*{\CXX}{C_{\!X\!X}}
\newcommand*{\CXY}{C_{\!X\!Y}}
\newcommand*{\CYX}{C_{\!Y\!X}}
\newcommand*{\CYY}{C_{\!Y\!Y}}
\newcommand*{\hatCXX}{\widehat{C}_{\!X\!X}}

\newcommand*{\hatCYX}{\widehat{C}_{\!Y\!X}}

\newcommand{\Spectrum}{\sigma}
\newcommand*{\cSpectrum}{\sigma_{\textup{c}}}
\newcommand*{\pSpectrum}{\sigma_{\textup{p}}}
\newcommand*{\rSpectrum}{\sigma_{\textup{r}}}

\renewcommand*{\geq}{\geqslant}

\renewcommand*{\leq}{\leqslant}

\renewcommand*{\mathbf}{\boldsymbol}

\newcommand{\todo}[1]{\bgroup\color{red}#1\egroup}
\newcommand{\MMsays}[1]{\bgroup\color{olive}Mattes:~#1\egroup}
\newcommand{\NMsays}[1]{\bgroup\color{olive}Nicole:~#1\egroup}
\newcommand{\TJSsays}[1]{\bgroup\color{olive}Tim:~#1\egroup}

\theoremstyle{plain}
\newtheorem{theorem}{\sffamily Theorem}[section]
\newtheorem{proposition}[theorem]{\sffamily Proposition}
\newtheorem{lemma}[theorem]{\sffamily Lemma}
\newtheorem{corollary}[theorem]{\sffamily Corollary}
\theoremstyle{definition}
\newtheorem{definition}[theorem]{\sffamily Definition}

\newtheorem{example}[theorem]{\sffamily Example}

\newtheorem{remark}[theorem]{\sffamily Remark}

\newtheorem{assumption}[theorem]{\sffamily Assumption}

\newcommand{\absval}[1]{\lvert #1 \rvert}

\newcommand{\innerprod}[2]{\langle #1 , #2 \rangle}
\newcommand{\norm}[1]{\lVert #1 \rVert}
\newcommand{\set}[2]{\{ #1 \mid #2 \}}
\newcommand{\bigabsval}[1]{\bigl\vert #1 \bigr\vert}

\newcommand{\bignorm}[1]{\bigl\Vert #1 \bigr\Vert}
\newcommand{\bigset}[2]{\bigl\{ #1 \,\big\vert\, #2 \bigr\}}

\newcommand{\Norm}[1]{\left\Vert #1 \right\Vert}
\newcommand{\Set}[2]{\left\{ #1 \,\middle\vert\, #2 \right\}}

\numberwithin{equation}{section}
\numberwithin{figure}{section}
\numberwithin{table}{section}

\usepackage[a4paper, top=88pt, bottom=88pt, left=72pt, right=72pt, headsep=16pt, footskip=28pt]{geometry}
\usepackage{url}
\usepackage{hyperref}
\hypersetup{
	breaklinks=true,
	colorlinks=true,
	linkcolor=blue,
	citecolor=blue,
	urlcolor=blue,
}
\newcommand*{\arXiv}[1]{\bgroup\color{blue}\href{https://arxiv.org/abs/#1}{arXiv:#1}\egroup}
\newcommand*{\doi}[1]{\bgroup\color{blue}\href{https://doi.org/#1}{doi:#1}\egroup}
\newcommand*{\email}[1]{\bgroup\color{blue}\href{mailto:#1}{#1}\egroup}
\renewcommand*{\url}[1]{\bgroup\color{blue}\href{#1}{#1}\egroup}
\usepackage{enumitem, moreenum}
\setlist[enumerate]{nosep}
\setlist[itemize]{nosep}
\usepackage{mleftright} \mleftright

\usepackage{xpatch}
\newcommand{\proofheadfont}{\bfseries\sffamily}
\xpatchcmd{\proof}{\itshape}{\proofheadfont}{}{}

\tikzcdset{scale cd/.style={every label/.append style={scale=#1},
    cells={nodes={scale=#1}}}}

\usepackage[labelfont={sf,bf}]{caption}
\usepackage{scrlayer-scrpage, xhfill}
\automark[section]{section}
\setkomafont{pageheadfoot}{\normalcolor\sffamily}
\setkomafont{pagenumber}{\normalfont\normalsize\sffamily}
\clearpairofpagestyles
\let\oldtitle\title
\renewcommand{\title}[1]{\oldtitle{#1}\newcommand{\theshorttitle}{#1}}
\newcommand{\shorttitle}[1]{\renewcommand{\theshorttitle}{#1}}
\let\oldauthor\author
\renewcommand{\author}[1]{\oldauthor{#1}\newcommand{\theshortauthor}{#1}}
\newcommand{\shortauthor}[1]{\renewcommand{\theshortauthor}{#1}}
\cohead{\xrfill[0.525ex]{0.6pt}~\theshorttitle~\xrfill[0.525ex]{0.6pt}}
\cehead{\xrfill[0.525ex]{0.6pt}~\theshortauthor~\xrfill[0.525ex]{0.6pt}}
\newcommand{\theabstract}[1]{\par\bgroup\noindent\textbf{\textsf{Abstract.}} #1\egroup}
\newcommand{\thekeywords}[1]{\par\smallskip\bgroup\noindent\textbf{\textsf{Keywords.}}\newcommand{\and}{ $\bullet$ } #1\egroup}
\newcommand{\themsc}[1]{\par\smallskip\bgroup\noindent\textbf{\textsf{2020 Mathematics Subject Classification.}}\newcommand{\and}{ $\bullet$ } #1\egroup}

\newcommand*{\affilref}[1]{\ref{affiliation#1}}
\newcommand*{\affiliation}[3]{
	\footnotetext[#1]{\label{affiliation#2}#3}
}

\title{Learning linear operators:\\%
Infinite-dimensional regression as a well-behaved non-compact inverse problem}
\shorttitle{Learning linear operators: Infinite-dimensional regression as an inverse problem}
\author{%
	Mattes~Mollenhauer\textsuperscript{\affilref{MXM}}%
	\and%
	Nicole M\"ucke\textsuperscript{\affilref{TUBS}}%
	\and%
	T.~J.~Sullivan\textsuperscript{\affilref{Warwick}}%
}
\shortauthor{M.~Mollenhauer, N.~M\"ucke, and T.~J.~Sullivan}
\date{\today}

\begin{document}
\maketitle
\affiliation{1}{MXM}{Merantix Momentum, Merantix AI Campus, Max--Urich--Straße 3, 13355 Berlin, Germany\\ (\email{mattes.mollenhauer@merantix-momentum.com})}
\affiliation{2}{TUBS}{Institut f\"ur Mathematische Stochastik, Technische Universit\"at Braunschweig, Universit\"atsplatz 2 (Forumsgeb\"aude), 38106 Braunschweig, Germany (\email{nicole.muecke@tu-bs.de})}
\affiliation{3}{Warwick}{Mathematics Institute and School of Engineering, University of Warwick, Coventry, CV4 7AL, United Kingdom\newline (\email{t.j.sullivan@warwick.ac.uk})}

\begin{abstract}\small
	\theabstract{We consider the problem of learning a linear operator
$\theta$ between two Hilbert spaces from empirical observations,
which we interpret as least squares regression
in infinite dimensions.
We show that this goal can be reformulated as an
inverse problem for $\theta$ with the
feature that its forward operator is generally non-compact
(even if $\theta$ is assumed to be compact or of $p$-Schatten class).
However, we prove that, in terms of spectral properties and regularisation theory, 
this inverse problem is equivalent to the known compact 
inverse problem associated with scalar response regression. 

Our framework allows for the elegant derivation
of dimension-free rates for generic learning algorithms
under H\"older-type source conditions.
The proofs rely on the combination of techniques from kernel regression with
recent results on concentration of measure for
sub-exponential Hilbertian random variables.
The obtained rates hold for
a variety of practically-relevant scenarios in functional regression as well as
nonlinear regression with operator-valued kernels and match
those of classical kernel regression with scalar response.
}
	\thekeywords{statistical learning%
\and%
inverse problems%
\and%
spectral regularisation%
\and%
concentration of measure%
\and%
functional data analysis%
\and%
kernel regression%
\and%
conditional mean embedding%
}
	\themsc{62J05
\and%
65J22
\and%
47A52
\and%
47A68
}
\end{abstract}

\section{Introduction}
\label{sec:introduction}

Let $X$ and $Y$ be Bochner square-integrable random variables taking values in separable Hilbert spaces $\cX$ and $\cY$ respectively.
We aim to solve the following \emph{regression problem}:
\begin{equation}
	\label{eq:regression_problem}
	\tag{RP}
    \text{minimise } \E [ \norm{ Y - \theta X }_{\cY}^{2} ] \equiv \norm{ Y - \theta X }_{L^{2} (\prob; \cY)}^{2} \text{ with respect to $\theta \in L(\cX, \cY)$},
\end{equation}
where $L(\cX, \cY)$ denotes the Banach space of bounded linear operators from $\cX$ into $\cY$.
If at least one of $\cX$ and $\cY$ has infinite dimension, then so too does the search space $L(\cX, \cY)$, and so \eqref{eq:regression_problem} is an infinite-dimensional regression problem.
We are particularly motivated by the case of infinite-dimensional $\cY$, exemplified by relevant applications in \emph{functional linear regression with functional response} \citep{RamsaySilverman2005},
non-parametric regression with \emph{vector-valued kernels} 
\citep{CaponnettoDeVito2007,LiEtAl2024, Meunier2024},
the \emph{conditional mean embedding} (\citealp{ParkMuandet2020}; \citealp{LiEtAl2022})
and inference for \emph{Hilbertian time series} \citep{Bosq2000}.
We will discuss each of these applications in the context of our results in more detail in \Cref{sec:applications}.
We assume that the \emph{joint distribution} $\law(X, Y)$ of $X$ and $Y$ is not
known explicitly and \eqref{eq:regression_problem}
must be solved approximately based on sample pairs drawn from $\law(X, Y)$. We approach this
problem by proposing an inverse problem framework which generalises finite-dimensional
\emph{ridge regression}, \emph{principal component regression}
and various other techniques including
iterative learning algorithms (e.g.\ \emph{gradient descent methods}).

In contrast to finite-dimensional linear regression, the infinite-dimensional analogue does not necessarily admit a minimiser.
Under the assumption of a \emph{linear model}, i.e.\ the existence of a bounded linear operator $\theta_{\star} \colon \cX \to \cY$ such that
\[
    Y = \theta_{\star} X + \epsilon
\]
with an exogeneous $\cY$-valued noise variable $\epsilon$ satisfying $\E[ \epsilon \vert X ] = 0$ --- which may be interpreted as the \emph{well-specified case} in the context of statistical learning theory --- it is known that the regression problem \eqref{eq:regression_problem} can be solved via the \emph{operator factorisation problem}
\begin{equation}
	\label{eq:operator_factorisation}
	\tag{OFP}
    \CYX = \theta \CXX, \quad \quad \theta \in L(\cX, \cY) ,
\end{equation}
where $\CYX \in L(\cX, \cY)$ and $\CXX \in L(\cX, \cX)$ are the \emph{covariance operators} \citep{Baker1973} associated with $X$ and $Y$, which is again related
to an well-known set of range inclusion and operator majorisation
conditions due to \citet{Douglas1966}
and the \emph{Moore--Penrose pseudoinverse} \citep{EnglHankeNeubauer1996}.
The representation of linear conditional expectations
in terms of pseudo\-inverses of covariance operators
has also been investigated recently by \citet{KlebanovEtAl2021}.

The above operator factorisation problem \eqref{eq:operator_factorisation} can be reformulated in terms of a (potentially ill-posed) \emph{linear inverse problem}
\begin{equation}
	\label{eq:inverse_problem}
	\tag{IP}
    A_{\CXX}[\theta] = \CYX, \quad \theta \in L(\cX, \cY)
\end{equation}
based on the (generally non-compact) \emph{forward operator} $A_{\CXX} \colon L(\cX, \cY) \to L(\cX, \cY)$ given by
\begin{equation*}
    A_{\CXX}[\theta] \defeq \theta \CXX .
\end{equation*}
We call the operator $A_{\CXX}$ the \emph{precomposition operator} associated with $\CXX$.

We show that, even in the \emph{misspecified case} (i.e.\ without the
assumption of a linear model), the solution to
the inverse problem \eqref{eq:inverse_problem}
still characterises the minimiser of the linear regression problem \eqref{eq:regression_problem}.
We argue that, due to this fact, our viewpoint
is the natural generalisation of finite-dimensional
least squares regression. We will revisit this perspective later on and
compare \eqref{eq:inverse_problem} to the
well-known inverse problem associated with scalar-valued response regression.

\paragraph{Main results.}
Motivated by the ultimate need for a regularised empirical solution to the inverse problem, we give a characterisation of the spectrum of the precomposition operator.
In particular it turns out that, \emph{although $A_{\CXX}$ is generally non-compact, the spectra of $A_{\CXX}$ and $\CXX$ coincide} (\Cref{thm:spectral_properties_of_A_C}).
As $\CXX$ is self-adjoint and compact, we show that the non-compact inverse problem \eqref{eq:inverse_problem} can be regularised as if it were a compact inverse problem.
Thus, one key message of this article is that,
from a regularisation standpoint, when the random variable $X$ is assumed to be infinite-dimensional,
\begin{displayquote}
    \emph{predicting an infinite-dimensional random variable $Y$ is ``just as hard'' as predicting a finite-dimensional random variable $Y$ based on information provided by $X$.}
\end{displayquote}
We show that convergence proof techniques for functional regression and kernel regression with
finite-dimensional $Y$ can be used --- with surprisingly minor adaptions --- for infinite-dimensional $Y$.
As these techniques and the empirical theory of covariance operators have seen rapid development over recent years,
this opens up a powerful arsenal of tools for the infinite-dimensional learning problem.
We underline this point by proving basic rates for generic learning algorithms
(i.e.\ regularisation strategies) by modifying approaches from the kernel community.
The rates obtained match known rates
for kernel regression under comparable assumptions.

\paragraph{Scope of this paper.}
This paper has two main goals.
\begin{enumerate}[label=(\alph*)]
	\item
	The primary focus of this paper is to investigate the \emph{population formulation}
	of linear regression for Hilbert-valued random variables.
	Although this concept has gained interest in terms
	of \emph{learning linear operators} in machine learning and
	functional data analysis, the mathematical background in the
	general context of inverse problems is novel.
	\item
	We also contribute to
	the population formulation by providing the functional-analytic framework for
	our secondary goal, namely the \emph{statistical analysis of regularised empirical solutions}. We tackle this problem
	by modifying proof techniques developed by the
	statistical learning community for kernel-based regression
	in combination with concentration bounds for sub-exponential
	Hilbertian random variables.
	However, we do not yet provide an entirely exhaustive statistical analysis;
	instead we aim to lay the groundwork for a unified empirical theory,
    with further results such as \emph{fast rates} and \emph{optimal rates}
    under additional regularity assumptions deferred to future work.
\end{enumerate}

\begin{remark}[Precomposition and operator factorisation]
	Problem \eqref{eq:operator_factorisation} is a special case of a general operator factorisation problem:
	given spaces $\cX$ and $\cY$ and operators $B \in L(\cX, \cY)$ and $C \in L(\cX, \cX)$, find an operator $\theta \in L(\cX, \cY)$ such that $B = \theta C$.
	This problem can be approached as an inverse problem \`a la \eqref{eq:inverse_problem}, i.e.\ $A_{C} [\theta] = B$, and it would certainly be interesting to explore the power and limitations of this approach.
	However, it suffices for our present purposes to consider only instances of $B = \theta C$ arising from problem \eqref{eq:regression_problem}, which offers several simplifications:
	the operators $B = \CYX$ and $C = \CXX$ are of trace class, $\CXX$ is self-adjoint and positive, and $\CYX$ is ``almost'' dominated by $\CXX$ in the sense that, for some operator $R$ of at most unit norm, $\CYX = \CYY^{1/2} R \CXX^{1/2}$ \citep[Theorem~1]{Baker1973}.
\end{remark}

\begin{remark}[Spectral theory of precomposition operators]
    Although we investigate the spectrum of the precomposition operator in order to understand the regularisation
    of \eqref{eq:inverse_problem}, our spectral-theoretic results
    apply to general precomposition operators which are not necessarily related to covariance operators.
    As composition operators can be understood as \emph{tensor products of linear operators} acting on tensor products of Hilbert spaces \citep[Section~12.4]{Aubin2000}, these results may be relevant in other disciplines involving a similar functional-analytic background.
\end{remark}

\begin{remark}[Related work]
The theory of learning linear operators on infinite-dimensional 
spaces has gained significant attention over the last years.
However, in contrast to our present work, 
the current literature seems to almost exclusively consider problem settings
and assumptions which are rather specific.
Notable examples include the 
approximation of $L^p$-space operators based on reproducing kernels
(e.g. \citealp{MollenhauerKoltai2020}, \citealp{LiEtAl2022}, 
\citealp{KosticEtAl2022}, \citealp{JinEtAl2022})
and functional regression with functional response
(e.g. \citealp{CrambesMas2013}, \citealp{HoermannKidzinski2015},
\citealp{BenatiaEtAl2017}, \citealp{ImaizumiKato2018}).
As we discuss later, estimators for these settings
can be viewed as special cases of our framework in some instances.
We also point out the more recent work by \citet{deHoopEtAl2023},
who investigate a Bayesian estimator in a Hilbertian linear model
when the eigenfunctions of the ground truth operator are known a priori,
allowing for a detailed statistical investigation.
\end{remark}

The remainder of this paper is structured as follows.
\Cref{sec:notation} gives the problem setting and fixes some
basic notation for the rest of the paper.
\Cref{sec:inverse_problem} formally introduces the precomposition
operator, investigates its mathematical properties and
its establishes its connection to the infinite-dimensional regression problem.
\Cref{sec:regularisation} lays the groundwork for
a detailed statistical analysis of regression with Hilbert--Schmidt
operators by combining spectral regularisation techniques
with concentration bounds for sub-exponential Hilbertian
random variables.
In particular, we derive basic rates under the assumption
of H\"older-type source conditions.
\Cref{sec:applications} discusses the consequences of our results
for specific statistical fields and connects them to the relevant literature.
\Cref{sec:closing}
recapitulates the key points of our work
and their implications in the abstract context of
\emph{learning infinite-dimensional information from data}
by comparing the infinite-dimensional response setting to
standard regression with real-valued responses.

\section{Notation and setup of the problem}
\label{sec:notation}

\paragraph{Random variables of interest.}
Let $X$ and $Y$ be random variables defined on a common probability space $(\Omega, \sigalg, \prob)$ and taking values in separable real Hilbert spaces $\cX$ and $\cY$ respectively.
We require that $X$ and $Y$ each have finite second moment, i.e.\ that $X$ and $Y$ lie in the Bochner spaces $L^{2}(\Omega, \sigalg, \prob; \cX)$ and $L^2(\Omega, \sigalg, \prob; \cY)$ respectively.
To save space, we usually shorten $L^{2}(\Omega, \sigalg, \prob; \cX)$ to $L^{2}(\prob; \cX)$ --- or even to $L^{2}(\prob)$ when $\cX = \Reals$.
See e.g.\ \citet[Section~II.2]{DiestelUhl1977} for a comprehensive treatment of the Bochner integral.
We write $\law(X)$ for the \emph{distribution} or \emph{law} of $X$ on the measurable space $(\cX, \Borel_{\cX})$
(i.e.\ the pushforward measure of $\prob$ under $X$), where $\Borel_{\cX}$ denotes the Borel $\sigma$-algebra of $\cX$.
We write $\sigalg^{X}$ for the $\sigma$-algebra generated by $X$, i.e.\ the coarsest one on $\Omega$ with respect to which $X$ is a measurable function from $(\Omega, \sigalg^{X})$ into $(\cX, \Borel_{\cX})$.

\paragraph{Spaces of operators.}
We write $L(\cX, \cY)$ for the Banach space of bounded linear operators from $\cX$ to $\cY$ and write either $\norm{ A }_{L(\cX, \cY)}$ or $\norm{ A }_{\cX \to \cY}$ for the operator norm of a linear operator $A \colon \cX \to \cY$.
The Banach space of compact linear operators from $\cX$ to $\cY$, also equipped with the operator norm, will be denoted $S_{\infty} (\cX, \cY)$.
For $1 \leq p < \infty$, $S_{p} (\cX, \cY) \subseteq S_{\infty} (\cX, \cY)$ denotes the Banach space of $p$-\emph{Schatten class} operators from $\cX$ to $\cY$, i.e.\ those whose singular value sequence is $p$\textsuperscript{th}-power summable, with the norm being the $\ell^{p}$ norm of the singular value sequence.
For $p=1$ we obtain the \emph{trace class operators};
for $p=2$ we obtain the \emph{Hilbert--Schmidt operators}.
We have $S_{p} (\cX, \cY) \subseteq S_{q} (\cX, \cY)$ for $p \leq q$.
The class $S_{2} (\cX, \cY)$ is a Hilbert space with respect to the inner product
\begin{equation}
	\label{eq:Hilbert--Schmidt_inner_product}
	\innerprod{ A }{ B }_{S_{2}(\cX, \cY)} \defeq \trace ( A^{\ast} B ) .
\end{equation}
Naturally, we shorten $L(\cX, \cX)$ to $L(\cX)$, $S_{p}(\cX, \cX)$ to $S_{p}(\cX)$, etc.

\paragraph{Hilbert tensor products.}
For $y \in \cY$ and $x \in \cX$, $y \otimes x \in L(\cX, \cY)$ is the \emph{rank-one operator}
\begin{equation}
	\label{eq:rank-one_operator}
\cX \ni v \mapsto (y \otimes x)(v) \defeq \innerprod{ x }{ v }_{\cX}\, y \in \cY .
\end{equation}
The Hilbert tensor product $\cY \otimes \cX$ is defined to be the completion of the linear span of all such
rank-one operators with respect to the inner product $\innerprod{ y \otimes x }{ y' \otimes x' }_{\cY \otimes \cX} \defeq \innerprod{ y }{ y' }_{\cY} \innerprod{ x }{ x' }_{\cX}$.
We will freely use the isometric isomorphisms $S_{2} (\cX, \cY) \cong \cY \otimes \cX$ and $L^{2}(\prob; \cX) \cong L^{2}(\prob; \Reals) \otimes \cX$ \citep[Chapter~12]{Aubin2000}.
Note that
\begin{equation}
	\label{eq:norm_of_outer_product}
	\norm{ y \otimes x }_{S_{p} (\cX, \cY)} = \norm{x}_{\cX} \norm{y}_{\cY} \quad \text{for any $x \in \cX$, $y \in \cY$, and $1 \leq p \leq \infty$.}
\end{equation}

\paragraph{Covariance operators.}
The (uncentred) \emph{covariance operators} \citep{Baker1973} of $Y$ with $X$, and of $X$ with itself, are given by
\begin{align}
	\label{eq:C_YX}
    \cov[ Y, X ] \defeq
    \CYX & \defeq \E [ Y \otimes X ] \in S_1(\cX, \cY) \textnormal{ and} \\
	\label{eq:C_XX}
    \cov[ X, X ] \defeq
	\CXX & \defeq \E [ X \otimes X ] \in S_1(\cX) .
\end{align}
Note that $\CYX^{\ast} = \CXY$, and so $\CXX$ is self-adjoint.
The covariance operators are the unique operators satisfying
\[
    \E[ \innerprod{y}{Y}_{\cY} \innerprod{x}{X}_{\cX} ]
    = \innerprod{ y }{ \CYX x}_{\cY} \quad \textnormal{for all }
    x \in \cX, y \in \cY.
\]
Moreover, for all $X_1, X_2 \in L^2(\prob; \cX)$,
\[
    \innerprod{ X_1 }{ X_2}_{L^2(\prob; \cX)} = \trace( \cov[X_1, X_2] ).
\]
For any $\theta \in L(\cX, \cY)$, $ \cov[ Y, \theta X ] = \cov[ Y, X ] \theta^{\ast}$ \citep[Lemma~A.6]{KlebanovEtAl2021}, which also implies that $ \cov[ \theta Y, X ] = \theta \cov[ Y, X ]$.
It can additionally be shown that
$\prob[ X \in \overline{ \range }(\CXX)  ] = 1$
\citep[Theorem~7.2.5]{HsingEubank2015},
and hence that $\prob[ X \in \ker(\CXX) \setminus \{0\} ] = 0$.

\paragraph{Pseudoinverses.}
Given $A \in L(\cX, \cY)$, its \emph{Moore--Penrose pseudoinverse} $A^{\dagger} \colon \domain (A^{\dagger}) \to \cX$ is the unique extension of the bijective operator
\[
	\bigl( A |_{\ker(A)^{\perp}} \bigr)^{-1} \colon \range(A) \to \ker(A)^{\perp}
\]
to a linear operator $A^{\dagger}$ defined on $\domain (A^{\dagger}) \defeq \range (A) \oplus \range(A)^{\perp} \subseteq \cY$ subject to the criterion that $\ker(A^{\dagger}) = \range(A)^{\perp}$.
In general, $\domain(A^{\dagger})$ is a dense but proper subpace of $\cY$ and $A^{\dagger}$ is an unbounded operator;
global definition and boundedness of $A^{\dagger}$ occur precisely when $\range(A)$ is closed in $\cY$.
For $y \in \domain (A^{\dagger})$,
\[
	A^{\dagger} y
	=
	\argmin \Set{ \norm{ x }_{\cX} }{ \vphantom{\big|}x \text{ minimises } \norm{ A x - y }_{\cY} } .
\]
In particular, for $y \in \range A$, $A^{\dagger} y$ is the minimum-norm solution of $A x = y$.
See e.g.\ \citet[Section~2.1]{EnglHankeNeubauer1996} for a comprehensive treatment of the Moore--Penrose pseudoinverse.

\section{Inverse problem: Identifying the optimal regressor}
\label{sec:inverse_problem}

The purpose of this section is to characterise the solution of the infinite-dimensional regression problem \eqref{eq:regression_problem} in terms of an inverse problem involving the \emph{precomposition operator}.

\subsection{Precomposition operator}

Here we define the precomposition operator $A_{C}$ induced by $C \in L(\cX)$ and establish its key properties.
At this stage $C$ is not required to be the (self-adjoint, positive semi-definite, trace-class) covariance operator $\CXX$, although that is the main use case that we have in mind.

\begin{definition}[Precomposition]
    Given $C \in L(\cX)$, the associated \emph{precomposition operator} $A_{C} \colon L(\cX, \cY) \to L(\cX, \cY)$ is defined by
    \[
        A_{C} [ \theta ] \defeq \theta C .
    \]
\end{definition}

\begin{lemma}[Basic properties of the precomposition operator]
    \label{lem:basic_precomposition}
    Let $C \in L(\cX)$ and $1 \leq p \leq \infty$.
    Then $A_{C} \colon L(\cX, \cY) \to L(\cX, \cY)$
    admits the following properties:
	\begin{enumerate}[label=(\alph*)]
        \item \label{lem:basic_precomposition_bounded}
            $A_{C}$ is bounded with $\norm{ A_{C} }_{L(L(\cX, \cY))} = \norm{ C }_{L(\cX)}$;
        \item
            $A_{c B + C} = c A_{B} + A_{C}$
            for all $c \in \Reals$ and $B \in L(\cX)$;
        \item
            \label{lem:basic_precomposition_injective}
            $A_{C}$ is injective if and only if $C^{\ast}$ is injective;
		\item \label{lem:basic_precomposition_Schatten}
        if $C \in S_{p} (\cX)$, then $\range(A_{C}) \subseteq S_{p} (\cX, \cY)$; and
        \item \label{lem:basic_precomposition_invariant}
        the subspaces $S_{p} (\cX, \cY) \subseteq L(\cX, \cY)$ are
        invariant under the action of $A_{C}$, i.e.
        \[
            A_{C}[ S_{p} (\cX, \cY) ] \subseteq S_{p} (\cX, \cY).
        \]
    \end{enumerate}
\end{lemma}

\begin{remark}[Precomposition on $p$-Schatten classes]
	\label{rem:p_schatten}
	The invariance of the $p$-Schatten classes (\Cref{lem:basic_precomposition}\ref{lem:basic_precomposition_invariant}) allows the precomposition operator $A_{C}$ to be considered as an operator on $S_{p} (\cX, \cY)$ instead --- with its domain and range equipped with the corresponding $p$-Schatten norm.
	We indicate to which norm we refer by adding corresponding subscripts to the operator norm if it is relevant in the context.
	In this case,
	$\norm{ A_{C} }_{L(S_{p} (\cX, \cY)) \to S_{p} (\cX, \cY)} = \norm{C}_{L(\cX)}$,
	since
	\[
		\norm{ A_{C}[\theta] }_{S_{p} (\cX, \cY)} \leq \norm{C}_{L(\cX)} \norm{ \theta }_{S_{p} (\cX, \cY)}.
	\]
	for all $\theta \in S_{p} (\cX, \cY)$ and $1 \leq p \leq \infty$, and the proof
	strategy is the same as for
	\Cref{lem:basic_precomposition}\ref{lem:basic_precomposition_bounded}.
\end{remark}

\begin{proof}[Proof of \Cref{lem:basic_precomposition}]
	\begin{enumerate}[label=(\alph*)]
		\item
		The upper bound
		\[
			\norm{ A_{C} }_{L(L(\cX, \cY))} \leq \norm{ C }_{L(\cX)}
		\]
		follows from submultiplicativity of the operator norm.
		To see that equality holds, let $y \in \cY$ have unit norm.
		Then
		\begin{align*}
		\sup_{x \in \cX, \norm{x}_{\cX} = 1}
		\norm{ A_{C}[ y \otimes x] }_{L(\cX, \cY)}
		& = \sup_{x \in \cX, \norm{x}_{\cX} = 1} \norm{ (y \otimes x) C}_{L(\cX, \cY)} \\
		& = \sup_{x \in \cX, \norm{x}_{\cX} = 1} \norm{ y \otimes (C^{\ast}x)}_{L(\cX, \cY)} \\
		& = \norm{C^{\ast}}_{L(\cX)} && \text{(by \eqref{eq:norm_of_outer_product} with $p = \infty$)} \\
		& = \norm{C}_{L(\cX)} .
		\end{align*}

        \item
        Let $\theta \in L(\cX, \cY)$ and $x \in \cX$ be arbitrary.
        Then
        \[
        	A_{c B + C} [\theta] (x) \defeq \theta ( ( c B + C ) (x) ) = \theta ( c B x + C x ) = c (\theta B) x + (\theta C) (x) = c A_{B} [\theta] (x) + A_{C} [\theta] (x) ,
        \]
        i.e.\ $A_{c B + C} [\theta] = c A_{B} [\theta] + A_{C} [\theta]$, i.e.\ $A_{c B + C} = c A_{B} + A_{C}$, as claimed.

        \item Suppose that $C^{\ast}$ is injective.
          We see that $A_{C}[\theta] = \theta C = 0$ is equivalent to
          the fact that we have $\range(C) \subseteq \ker(\theta)$.
          As $\ker(\theta)$ is closed, this implies
          $\overline{\range}(C) \subseteq \ker(\theta)$.
          Using that $\overline{\range}(C) = \ker( C^{\ast} )^\perp = \cX$
          by assumption, we see that $\theta = 0$, i.e.\ $A_{C}$ is injective.

          Conversely, suppose that $A_{C}$ is injective.
          Then, for every $\theta \neq 0$,
          $A_{C}[\theta] = \theta C  \neq 0$. By letting $\theta$
          vary over all rank-one operators from $\cX$ to $\cY$,
          we see that $\range(C)$ is dense in $\cX$, which is equivalent
          to $\overline{\range}(C) = \ker(C^{\ast})^\perp = \cX$
          and finally $\ker(C^{\ast}) = \{0 \}$.

		\item Suppose that $C \in S_{p}(\cX)$ and $\theta \in L(\cX, \cY)$.
        Then $A_{C} [\theta] \defeq C \theta \in S_{p}(\cX, \cY)$ by the well-known fact that the Schatten classes are operator ideals.

        \item The proof of this claim is similar to that of \ref{lem:basic_precomposition_Schatten}.
    \end{enumerate}\vspace{-\baselineskip}
\end{proof}

\begin{remark}[Precomposition as operator tensor product]
    When the precomposition operator $A_C$ acts on
    $S_2(\cX, \cY)$, it can be shown that it is isomorphically equivalent to the
    \emph{operator tensor product} $C \otimes \idop_{\cY}$
    on the tensor product space $ \cX \otimes \cY \cong S_2(\cX, \cY)$.
    Some of the above properties of $A_C$ are proven in a more convenient way
    when using this representation.
\end{remark}

\subsection{Regression solution and inverse problem}
\label{sec:regression_solution}

The solution of the regression problem can be described by an inverse problem with the forward operator $A_{\CXX}$.
Note that, unless we restrict the
true solution to be an element of $S_2(\cX, \cY)$,
this is an inverse problem with a forward operator
acting on the Banach space $L(\cX, \cY)$.

\begin{proposition}[Inverse problem]
	\label{prop:population_solution}
	There exists an operator $\theta_{\star} \in L( \cX, \cY)$
    solving the regression problem \eqref{eq:regression_problem}
	if and only if
	\begin{equation}
		\label{eq:existence}
		\CYX \in \range( A_{\CXX}) .
	\end{equation}
    The set of
    solutions of \eqref{eq:regression_problem} is exactly given by the
    solutions of the
	\emph{inverse problem} \eqref{eq:inverse_problem}.
\end{proposition}

\begin{proof}
    Given some $\theta \in L(\cX, \cY)$, we expand the objective functional of \eqref{eq:regression_problem} as
    \begin{align*}
        \E [ \norm{ Y - \theta X }^2_\cY ]
        & =
        \norm{ Y }^2_{L^2(\prob; \cY)}
        - 2 \innerprod{ Y }{ \theta X}_{L^2(\prob; \cY)}
        + \norm{ \theta X }^2_{ L^2(\prob; \cY) }.
    \end{align*}
    The first term does not depend on $\theta$. It therefore
    suffices to minimise the last two terms, which
    can be reformulated as
    \begin{align*}
        - 2 \innerprod{ Y }{ \theta X}_{L^2(\prob; \cY)}
        + \norm{ \theta X }^2_{ L^2(\prob; \cY) }
        & =
        - 2 \trace( \cov[ Y, \theta X ] )
        + \trace( \cov[ \theta X, \theta X ] ) \\
        &=
        - 2 \trace( \CYX \theta^{\ast} )
        + \trace( \theta \CXX \theta^{\ast} ) \\
        &=
        \trace(- 2 \CYX \theta^{\ast} + \theta \CXX \theta^{\ast} )
        \qefed F( \theta ).
    \end{align*}
    The minimum of $F(\theta)$ is attained
    at $\theta_{\star}$ satisfying
    $\tfrac{ \rd }{ \rd \theta } F(\theta_{\star} ) = 0 \in L( L(\cX,\cY), \Reals) = L(\cX, \cY)'$.
    Computing the derivative of $F$ at a point $\theta \in L(\cX, \cY)$
    in a direction $\delta \in L(\cX, \cY)$ yields
    \[
        \frac{ \rd }{ \rd \theta } F(\theta ) [\delta] =
    \trace( - 2 \CYX\delta^{\ast} + \theta \CXX \delta^{\ast} + \delta^{\ast} \CXX \theta).
    \]
	Since the trace is linear and invariant under
    forming the adjoint, the above term simplifies to
    \begin{align*}
        \frac{ \rd }{ \rd \theta } F(\theta ) [\delta ]
		& =
		\trace( - 2 \CYX\delta^{\ast} + 2 \theta \CXX \delta^{\ast}) \\
		& =
		\trace( - 2 \CYX\delta^{\ast} + 2 A_{\CXX} [\theta] \delta^{\ast}),
	\end{align*}
	which vanishes
	at some $\theta_{\star} \in L(\cX, \cY)$
	in all directions $\delta \in L(\cX, \cY)$ if and only if
	$\CYX = A_{\CXX} [\theta_{\star} ]$, as claimed.
\end{proof}

\begin{corollary}[Set of solutions]
    \label{cor:solution_set}
    Assume that the regression problem \eqref{eq:regression_problem} admits a minimiser $\theta_{\star} \in L(\cX, \cY)$.
    Then the set of operators in $L(\cX, \cY)$ solving
    \eqref{eq:regression_problem} is exactly given by
    \begin{equation*}
        \theta_{\star} + \ker ( A_{\CXX} )
        =
        \theta_{\star} +
        \bigset{ \theta \in L(\cX, \cY) }{
            \ker( \theta )^\perp \subseteq \ker( \CXX )
        }.
    \end{equation*}
\end{corollary}

\begin{proof}
    The fact that $\theta_{\star} + \ker ( A_{\CXX} )$ is the set of solutions of \eqref{eq:regression_problem}
    is the immediate consequence of \Cref{prop:population_solution}.
    Furthermore, it is clear that $\ker( A_{\CXX})$ consists precisely of the operators $\theta \in L(\cX, \cY)$ which
    satisfy $ \theta \CXX = 0$, which are exactly those operators that satisfy the condition
    $\ker(\theta)^\perp \subseteq \range(\CXX)^\perp = \ker(\CXX)$.
\end{proof}

Intuitively, \Cref{cor:solution_set} describes the fact that an optimal regressor $\theta_{\star} \in L(\cX, \cY)$
may be arbitrarily modified in basis directions of the space $\cX$ which do not lie in the support of $\law(X)$, while still leading to a minimiser of \eqref{eq:regression_problem}.
As a consequence, the minimiser $\theta_{\star}$ is unique whenever the distribution of $X$ has support on all basis directions of $\cX$.

\begin{definition}[Non-degenerate distribution]
	We call the distribution $\law(X)$ of $X$ on $(\cX, \Borel_{\cX})$ \emph{non-degenerate} if $\CXX$ is injective.
	Note that this is equivalent to the condition $\E[ \innerprod{x}{X}^2_\cX ] \neq 0$ for all $x \in \cX \setminus \{ 0 \}$.
\end{definition}

\begin{corollary}[Unique solution]
	\label{cor:unique_regression_iff_nondegenerate}
    Assume that a minimiser $\theta_{\star} \in L(\cX, \cY)$ of the regression problem \eqref{eq:regression_problem} exists.
    Then $\theta_{\star}$ is unique if and only if the distribution $\law(X)$ is non-degenerate.
\end{corollary}

\begin{proof}
    This follows directly from
    \Cref{lem:basic_precomposition}\ref{lem:basic_precomposition_injective}.
\end{proof}

\Cref{cor:solution_set,cor:unique_regression_iff_nondegenerate} assume the existence of a minimiser $\theta_{\star}$ for \eqref{eq:regression_problem}, and so we now turn our attention to necessary and sufficient conditions for this hypothesis to hold.

\subsection{\texorpdfstring{Existence of the minimiser $\theta_{\star} \in L(\cX, \cY)$}{Existence of the minimiser}}
\label{sec:inverse_problem_existence}

In contrast to finite-dimensional regression problems, the
existence of the regression solution $\theta_{\star} \in L(\cX, \cY)$ is
not clear, and so this issue deserves special attention in the Hilbertian case.
We now give alternative characterisations
of the necessary and sufficient condition \eqref{eq:existence}
for the existence of $\theta_{\star}$
and investigate in which practical scenarios these conditions are satisfied.

If $\law(X)$ is non-degenerate, then ``most'' (in the sense of topological density) covariance operators $\CYX$ lie in the range of $A_{\CXX}$:

\begin{remark}[Range of $A_{\CXX}$ is dense]
	\label{rmk:range_ACXX_dense}
    Let us briefly consider the operator
    $A_{\CXX}$ acting on the Hilbert space $S_2(\cX, \cY)$ as described
    in \cref{rem:p_schatten}.
    We assume that $\CXX$ is injective.
    Using the fact that $A_{\CXX}$ is self-adjoint
    with respect to $\innerprod{\quark}{\quark}_{S_2(\cX, \cY)}$
    (proved in
    \Cref{thm:spectral_properties_of_A_C}\ref{thm:spectral_properties_of_A_C_adjoint})
    together with
    the equivalence $\CXX$ injective $\iff A_{\CXX}$ injective
    from \Cref{lem:basic_precomposition}\ref{lem:basic_precomposition_injective},
    we see that
	\[
        \overline{ \range (A_{\CXX}) }^{\norm{ \quark }_{S_{2}(\cX, \cY)}}
        = \ker (A_{\CXX})^{\perp} = S_{2}(\cX, \cY) ,
	\]
    i.e.\ $\range (A_{\CXX})$ is dense in $S_{2}(\cX, \cY)$
    with respect to the Hilbert--Schmidt norm.
    It follows that $\range (A_{\CXX})$ is dense in $S_{\infty}(\cX, \cY)$
    with respect to the operator norm.
	In particular, since $\CYX$ is trace class and hence compact,
    it lies in the operator-norm closure of $\range (A_{\CXX})$
    provided $\law(X)$ is non-degenerate.
\end{remark}

Before investigating the condition for the existence
of $\theta_{\star}$ given in \eqref{eq:existence} in more
detail, we justify our approach with a classical example:
If the true relationship between
$Y$ and $X$ is linear,
then our theory naturally
gives the ``correct'' solution to the regression problem.

\begin{example}[Linear model]
    \label{ex:linear_model}
    Assume that the
    relation between $Y$ and $X$ is given in
    terms of the \emph{linear model}
    \begin{equation}
        \label{eq:linear_model}
        \tag{LM}
        Y = \theta_{\star} X + \varepsilon
    \end{equation}
    for some $\theta_{\star} \in L(\cX, \cY)$ and
    $\varepsilon \in L^2(\prob; \cY)$ such that
    the \emph{exogeneity condition}
    $\E[ \varepsilon \vert X ] = 0 $ holds.
    In this case, $\CYX \in \range( A_{\CXX} )$,
    since \eqref{eq:linear_model} yields
    \begin{align*}
         Y \otimes X & = \theta_{\star} X \otimes X + \varepsilon \otimes X \\
         \implies \quad \CYX  & =  \theta_{\star} \CXX + C_{\varepsilon X}.
    \end{align*}
    Since
    $C_{\varepsilon X} = \E[ \varepsilon \otimes X ]
    = \E[ \, \E[ \varepsilon \vert X] \otimes X ] = 0$,
    this implies $\CYX = A_{\CXX}[\theta_{\star} ]$.
    This confirms \cref{prop:population_solution}
    in the setting that the linear regression problem is
    \emph{well-specified} in the context of statistical
    learning theory.
    Note that the linear model above is equivalent to
    assuming the validity of
    the \emph{linear conditional expectation} (\citealp{KlebanovEtAl2021})
    representation
    \begin{equation}
        \label{eq:LCE}
        \tag{LCE}
        \E[ Y \vert X] = \theta_{\star} X \quad \prob\text{-a.s.}
    \end{equation}
    This equivalence is easy to see, since \eqref{eq:linear_model} and the exogeneity condition together yield \eqref{eq:LCE}.
    For the converse implication, we consider
    the orthogonal projection operator
    $\Pi \colon L^2(\prob; \cY) \to L^2(\prob; \cY)$
    onto the closed
    subspace $L^2(\Omega, \sigalg^{X}, \prob; \cY) \subseteq L^2(\prob; \cY)$
    of $\sigalg^{X}$-measurable functions
    and $\smash{ \Pi^\perp \defeq \idop_{L^2(\prob; \cY)} - \Pi} $.
    Using this notation, the characterisation of the conditional expectation
    as an $L^2$-projection yields that $\Pi Y = \E[ Y \vert X ]$.\footnote{%
        See e.g.\ \citet[Chapter 8]{Kallenberg2021} for
    the real-valued conditional expectation.
    The vector-valued case follows analogously in the
    context of the Bochner-$L^2$ space.}
    Performing the orthogonal decomposition
    $Y = \Pi Y + \Pi^{\perp} Y$
    and defining $\epsilon \defeq \Pi^\perp Y$ yield \eqref{eq:linear_model}.
    Note that $\epsilon$ satisfies the exogeneity assumption, since $\E[ \epsilon | X] = \Pi\epsilon  = \Pi (\Pi^\perp Y) = 0$.
\end{example}

Moreover, the existence of $\theta_{\star}$ can be characterised alternatively
in terms of a well-known range inclusion and majorisation condition
due to \citet{Douglas1966}.
Note that \eqref{eq:existence} can equivalently be expressed as
\begin{equation}
\label{eq:factorisation_problem}
\CYX = \theta_{\star} \CXX, \quad \theta_{\star} \in L(\cX, \cY),
\end{equation} which can be interpreted a
factorisation of $\CYX$ into a known operator $\CXX$ and some bounded
operator $\theta_{\star}$. This fact links \eqref{eq:regression_problem}
directly to Douglas' work.
We list important equivalent conditions
for the existence of this factorisation here for completeness
and give a closed-form solution of the regression problem.

\begin{proposition}[Existence of $\theta_{\star} \in L(\cX, \cY)$ and closed form]
    \label{prop:existence}
    The following conditions are equivalent
    to the existence of an operator $\theta_{\star} \in L(\cX, \cY)$
    which solves \eqref{eq:regression_problem}:
	\begin{enumerate}[label=(\alph*)]
    \item \label{prop:existence_douglas1}
            We have $\CYX = A_{\CXX}[\theta_{\star}] = \theta_{\star} \CXX$
            for some $\theta_{\star} \in L(\cX, \cY)$.
        \item We have $\range(\CXY) \subseteq \range(\CXX)$.
        \item \label{prop:existence_douglas3}
            The operator $\beta \CXX^2 - \CXY \CYX$
            is positive for some $\beta \in \Reals$.
        \item \label{prop:existence_bounded}
        The operator $\CXX^{\dagger} \CXY \colon \cY \to \cX$ is well-defined and bounded.
        \item \label{prop:existence_covariance_bound}
        For all $y \in \cY$,
        \begin{equation}
            \label{eq:covariance_bound}
            \sup_{x \in \cX }
            \frac{ \bigabsval{ \E[ \innerprod{x}{X}_\cX \innerprod{y}{Y}_\cY ]}}%
            { \norm{ \CXX x }_{\cX}  } < \infty.
        \end{equation}
    \end{enumerate}%
    If the above conditions are
    satisfied, then
    $\theta_{\star} \defeq ( \CXX^{\dagger} \CXY )^\ast$
    is a closed-form solution of \eqref{eq:regression_problem}.
    Moreover, $\theta_{\star} = ( \CXX^{\dagger} \CXY )^\ast$
    is the unique solution of \eqref{eq:regression_problem}
    satisfying
    \begin{equation}
        \label{eq:solution_uniqueness_criteria}
        \ker( \theta_{\star} )^\perp \subseteq \ker(\CXX)^\perp
        \text{ or equivalently }
        \theta_\star\vert_{\ker(\CXX)} = 0
    \end{equation}
    and its operator norm is given by
    \begin{equation}
        \label{eq:solution_norm}
        \norm{ \theta_{\star} }_{\cX \to \cY}
        =
        \sup_{x \in \cX} \frac{ \norm{ \CYX x }_\cY }{ \norm{ \CXX x }_\cX}.
    \end{equation}
\end{proposition}

\begin{proof}
    The equivalence of \ref{prop:existence_douglas1}--\ref{prop:existence_douglas3} is due to \citet[Theorem~1]{Douglas1966}, and the equivalence of \ref{prop:existence_bounded} follows from \citet[Theorem~A.1]{KlebanovEtAl2021}.
    The proof of the equivalence of \ref{prop:existence_covariance_bound}
    relies on a lesser known equivalent characterisation of
    Douglas' range inclusion theorem due to \citet[Lemma~1]{Shmulyan1967},
    which asserts that $x_0 \in \cX$ belongs to $\range(\CXX)$ if and only if
    \[
        \sup_{x \in \cX}
        \frac{ \absval{ \innerprod{x}{x_0}_\cX } }%
        { \norm{ \CXX x }_\cX } < \infty ;
    \]
    see also \citet[Corollary~2]{Fillmore1971}.
    Therefore, $\range(\CXY) \subseteq \range(\CXX)$
    if and only if the condition
    \[
        \sup_{x \in \cX}
        \frac{ \absval{ \innerprod{x}{ \CXY y}_\cX } }%
        { \norm{ \CXX x }_\cX } < \infty,
    \]
    is satisfied for all $y \in \cY$, proving the claim.
    The given uniqueness criteria in \eqref{eq:solution_uniqueness_criteria},
    various additional properties of the unique solution
    $\theta_{\star} = ( \CXX^{\dagger} \CXY )^\ast$
    as well as the role of the pseudoinverse
    are well known in the context of operator
    factorisation (\citealp{Fillmore1971, Shmulyan1967, AriasEtAl2008}).
    In order to obtain the norm of $\theta_{\star}$ given in
    \eqref{eq:solution_norm}, we apply
    \citet[Lemma~2c]{Shmulyan1967} to the adjoint of $\theta_{\star}$.
\end{proof}

\begin{remark}[Existence of $\theta_{\star} \in L(\cX, \cY)$]
\label{rem:existence}
	A few remarks about \cref{prop:existence} are in order.
    \begin{enumerate}[label=(\alph*)]
        \item
        We emphasise that, in general,
        $\smash{( \CXX^{\dagger} \CXY )^\ast
        \neq
        \CYX \CXX^\dagger}$.
        In particular, note that $\smash{\CXX^\dagger}$ as well
        as $\smash{\CYX \CXX^\dagger}$ are not globally defined
        whenever $\cX$ is infinite-dimensional.
        Indeed, as discussed in the proof above,
        $\smash{( \CXX^{\dagger} \CXY )^\ast }$
        is globally defined and bounded
        if and only if the equivalent conditions given in \cref{prop:existence}
        hold.
		The closed-form characterisation of the regression solution
        as $\theta_{\star} = ( \CXX^{\dagger} \CXY )^\ast$ can be
        understood as the
        \emph{mathematically informal representation}
		\[
			\theta_{\star} = A_{\CXX^\dagger}[\CYX]
		\]
		in the context of our framework.
        Note that \smash{$A_{\CXX^\dagger}$} is \emph{not}
        rigorously defined as an operator
        acting on $L(\cX, \cY)$.
		However, we argue that our approach is the natural way to
        discuss the solution of this problem, as the precomposition
        operator will later allow us to consider the direct regularisation
        of $\CXX$ in the composition
        $\theta_{\star} = ( \CXX^{\dagger} \CXY )^\ast$
        in a convenient way.
        \item The characterisation of all solutions of \eqref{eq:regression_problem}
        given in \Cref{cor:solution_set}
        together with \eqref{eq:solution_uniqueness_criteria} shows that
        $\theta_\star = (\CXX^\dagger \CXY)^*$ is the unique solution
        with the largest nullspace.
        All other solutions of \eqref{eq:regression_problem}
        perform operations on $\ker(\CXX)$ which
        are irrelevant for the regression.
		\item
		As also noted by \citet{KlebanovEtAl2021},
		it is easy to see that the equivalent conditions given in \cref{prop:existence}
		are always satisfied whenever $\cX$ is finite-dimensional,
		as in this case \smash{$\CXX^\dagger$} is always bounded.
    \end{enumerate}
\end{remark}

\subsection{Constraining the problem: Hilbert--Schmidt regression}

Instead of \eqref{eq:regression_problem}, in which $\theta$ is required merely to be a bounded operator, we may consider the constrained \emph{Hilbert--Schmidt regression problem}
\begin{equation}
    \label{eq:regression_problem_hs}
    \tag{HSRP}
    \min_{\theta \in S_2(\cX, \cY)} \E[ \norm{ Y - \theta X }_\cY^2 ].
\end{equation}
with the solution space $S_2(\cX, \cY)$.
It is straightforward to see that the results from \Cref{sec:regression_solution} remain valid for this constrained problem when
$\smash{A_{\CXX} \colon L(\cX, \cY) \to L(\cX, \cY)}$
is replaced with the modified precomposition operator
$\smash{A_{\CXX} \colon S_2(\cX, \cY) \to S_2(\cX, \cY)}$.
Instead of \eqref{eq:inverse_problem},
we now consider the modified inverse problem
\begin{equation}
    \label{eq:inverse_problem_hs}
    \tag{HSIP}
    A_{\CXX}[\theta] = \CYX, \quad \theta \in S_2(\cX, \cY).
\end{equation}
We emphasise that the ``HS'' in \eqref{eq:inverse_problem_hs} refers only to the fact that the search space consists of Hilbert--Schmidt operators $\theta$;
as we shall see, the forward operator $A_{\CXX}$ is generally not even compact, let alone Hilbert--Schmidt.
The restriction allows us to apply concepts from the rich theory of inverse problems
in Hilbert spaces \citep{EnglHankeNeubauer1996}, which significantly simplifies the regularisation and empirical solution of the regression problem.
The restriction to the solution space $S_2(\cX, \cY)$ in \eqref{eq:inverse_problem_hs} is quite standard in a multitude of practical applications of infinite-dimensional regression
(see \Cref{sec:applications}).

\begin{remark}[Existence of $\theta_{\star} \in S_2(\cX, \cY)$]
        \label{rem:existence_hs}
        The existence of an operator $\theta_{\star} \in S_2(\cX, \cY)$
        which solves
        \eqref{eq:regression_problem_hs}
        is clearly equivalent to
        $\CYX \in \range( A_{\CXX} ) $ with
        $\smash{A_{\CXX} \colon S_2(\cX, \cY) \to S_2(\cX, \cY)}$.
\end{remark}

We now express solutions of \eqref{eq:regression_problem_hs}
in terms of the pseudoinverse operator
\begin{equation*}
    A^{\dagger}_{\CXX} \colon
    S_2(\cX, \cY)
    \supseteq
    \dom(A^{\dagger}_{\CXX} )
    \to S_2(\cX, \cY),
\end{equation*}
where
$
\dom(A^{\dagger}_{\CXX} )
= \ran( A_{\CXX})
\oplus
\ran( A_{\CXX})^\perp.
$
The next result shows that, whenever a solution
of \eqref{eq:regression_problem_hs} exists,
the known closed-form solution
$\smash{\theta_{\star} = ( \CXX^{\dagger} \CXY )^\ast}$
is not only bounded but even Hilbert--Schmidt, and it
is recovered in a natural way
when the inverse problem in $S_2(\cX, \cY)$ is solved via the
pseudoinverse operator.

\begin{proposition}[Pseudoinverse solution]
    \label{prop:pseudoinverse_solution}
    Assume there exists a solution of \eqref{eq:regression_problem_hs},
    i.e.\ assume that $\CYX \in \range(A_{\CXX})$ with
    $\smash{A_{\CXX} \colon S_2(\cX, \cY) \to S_2(\cX, \cY)}$.
    Then
    \begin{equation*}
        \theta_{\star} \defeq
        ( \CXX^\dagger \CXY )^{\ast}
        =
        A_{\CXX}^\dagger[ \CYX ].
    \end{equation*}
    In particular, $\theta_{\star}$ defined in this way
    is the unique solution of \eqref{eq:regression_problem_hs}
    of minimal norm in $S_2(\cX, \cY)$.
\end{proposition}

\begin{proof}
    Assume $\CYX \in \range(A_{\CXX})$ with
    $\smash{A_{\CXX} \colon S_2(\cX, \cY) \to S_2(\cX, \cY)}$.
    By the construction of the pseudoinverse, this implies that the
    operator factorisation
    $\CYX = (A_{\CXX}^\dagger[ \CYX ]) \CXX$ holds and
    hence $A_{\CXX}^\dagger[ \CYX ] \in S_2(\cX, \cY)$ clearly solves
    the Hilbert--Schmidt regression problem
    \eqref{eq:regression_problem_hs} (and therefore
    also the bounded regression problem \eqref{eq:regression_problem})
    due to \cref{prop:population_solution}.
    Furthermore, according to \cref{prop:existence},
    the operator
    $\theta \defeq ( \CXX^\dagger \CXY )^{\ast}$
    is bounded and solves the bounded regression
    problem \eqref{eq:regression_problem}.
    Due to the characterisation of the set of bounded solutions
    in \cref{cor:solution_set},
    both solutions $\theta_{\star}$ and $A^\dagger_{\CXX}[\CYX]$
    can only differ on $\ker(\CXX)$.
    To prove their equivalence, it is thus sufficient to show that
    they coincide on $\ker( \CXX )$.
    By construction of the pseudoinverse, we have
    $ A^\dagger_{\CXX}[\CYX]|_{\ker(\CXX)} = 0$,
    since it is the minimal $S_2(\cX, \cY)$-norm
    solution of the operator factorisation
    $\CYX = \theta \CXX$ for $\theta \in L(\cX, \cY)$.
    Furthermore,
    \begin{equation*}
        \ker( \theta_{\star})
        =
        \range(\theta_{\star}^{\ast})^\perp
        =
        \range( \CXX^{\dagger} \CXY )^\perp
        \supseteq
        \range( \CXX^{\dagger} )^\perp
        =
        \ker(\CXX),
    \end{equation*}
    where the last equality follows again
    from the definition of the pseudoinverse.
    Since this implies $\theta_{\star} |_{\ker(\CXX)} = 0$,
    the claim is proven.
\end{proof}

We now give an equivalent characterisation of the requirement that
$\theta_{\star} = ( \CXX^\dagger \CXY )^{\ast} \in S_2({\cX, \cY})$
in terms of a moment condition
which can be interpreted as the Hilbert--Schmidt analogue of
the condition given in \Cref{prop:existence}\ref{prop:existence_covariance_bound}.

\begin{proposition}[Existence of $\theta_{\star} \in S_2(\cX, \cY)$]
    \label{prop:existence_hs}
    The operator $\theta_{\star} = ( \CXX^\dagger \CXY )^\ast $
    satisfies
    \begin{equation}
        \label{eq:solution_hs_norm}
        \norm{ \theta_{\star} }^2_{S_2(\cX, \cY) }
        =
        \sum_{ i \in I }
        \sup_{x \in \cX}
        \frac{ \absval{ \E[ \innerprod{x}{X}_{\cX} \innerprod{e_i}{Y}_{\cY} ] }^2 }%
        { \norm{ \CXX x }_{\cX}^2 },
    \end{equation}
    where $\{e_i\}_{i \in I}$
    is a complete orthonormal system in $\cY$
    (i.e.\ $\innerprod{ e_{i} }{ e_{j} }_{\cY} = \one [ i = j ]$ and $\overline {\operatorname{span}} \{e_i\}_{i \in I} = \cY$).
    In particular,
    the expression \eqref{eq:solution_hs_norm} is independent
    of the choice of complete orthonormal system and
    the solution $\theta_{\star}$ of \eqref{eq:regression_problem_hs} exists
    if and only if \eqref{eq:solution_hs_norm} is finite.
\end{proposition}

\begin{proof}
    Clearly, a necessary condition for
    $\theta_{\star} = ( \CXX^\dagger \CXY )^{\ast}$ to be
    Hilbert--Schmidt is that $\theta_{\star}$ is bounded.
    According to \Cref{prop:existence}, $\theta_{\star}$ therefore satisfies
    $\ker( \theta_{\star} )^\perp \subseteq \overline{ \range }(\CXX)$
    which is exactly the required condition for
    \citet[Lemma~2b]{Shmulyan1967} to hold. This yields
    \begin{equation*}
        \norm{ \theta_{\star}^\ast y }_{\cX} =
        \sup_{x \in \cX}
        \frac{ \absval{ \innerprod{ \CXY y}{x}_{\cX} } }%
        { \norm{ \CXX x }_{\cX} }
        =
        \sup_{x \in \cX}
        \frac{ \absval{ \E[ \innerprod{x}{X}_{\cX} \innerprod{y}{Y}_{\cY} ] } }%
        { \norm{ \CXX x }_{\cX} }
    \end{equation*}
    for all $y \in \cY$.
    Hence, by definition of the Hilbert--Schmidt norm,
    \begin{equation*}
        \norm{ \theta_{\star} }_{S_2(\cX, \cY) }^2
        =
        \norm{ \theta_{\star}^\ast }_{S_2(\cY, \cX) }^2
        =
        \sum_{ i \in I } \norm{ \theta_{\star}^\ast e_i }^2_{\cX}
        =
        \sum_{ i \in I }
        \sup_{x \in \cX}
        \frac{ \absval{ \E[ \innerprod{x}{X}_{\cX} \innerprod{e_i}{Y}_{\cY} ] }^2 }%
        { \norm{ \CXX x }_{\cX}^2 }
    \end{equation*}
    and clearly this expression is independent
    of the choice of complete orthonormal system.
\end{proof}

If \eqref{eq:regression_problem_hs} has a solution,
then we can characterise the set of all solutions of
\eqref{eq:regression_problem_hs}:

\begin{remark}[Set of Hilbert--Schmidt solutions]
    Whenever
    $\theta_{\star} =( \CXX^\dagger \CXY )^{\ast} \in S_2(\cX, \cY)$,
    the set of all solutions of \eqref{eq:regression_problem_hs}
    is given by
    \begin{equation*}
        \theta_{\star} + \ker( A_{\CXX} )
        =
        \theta_{\star} +
        \{ \theta \in S_2(\cX, \cY) |
            \ker( \theta )^\perp \subseteq \ker( \CXX )
        \}
    \end{equation*}
    with $A_{\CXX} \colon S_2(\cX, \cY) \to S_2(\cX, \cY)$,
    which is a standard characterisation in the theory of inverse problems
    (see \citealp{EnglHankeNeubauer1996}, Theorem~2.5).
\end{remark}

We emphasise that in general,
analogously to the discussion in \cref{rem:existence}, we have
$A^{\dagger}_{\CXX} \neq A_{\CXX^{\dagger}}$,
as the latter does clearly not map to
the space $S_2(\cX, \cY)$
whenever $\CXX^\dagger$ is unbounded.

\begin{remark}[Inverse problem]
	As is standard in linear inverse problems theory, one may distinguish three regimes for the	problem $\CYX = A_{\CXX} [\theta]$ with $\theta \in S_2(\cX, \cY)$:
	\begin{enumerate}[label=(\alph*)]
		\item
		$\CYX \in \range(A_{\CXX})$, in which case $\theta_{\star}$ does exist, as previously discussed;
		
		\item
		\smash{$\CYX \in \overline{\range}(A_{\CXX}) \setminus \range(A_{\CXX})$}, in which case $\CYX$ can be arbitrarily well approximated by a sequence of $\CYX^{(k)} \in \range(A_{\CXX})$, leading to approximate regression solutions $\theta^{(k)} = A_{\CXX}^{\dagger} [ \CYX^{(k)} ]$ which fail to converge in $\cX$ as $k \to \infty$ due to the unboundedness of the pseudoinverse operator $A_{\CXX}^{\dagger}$;
		
		\item
		$\CYX \notin \overline{\range}(A_{\CXX})$,
		which is ruled out as soon as $\CXX$ is injective (cf.\ \Cref{lem:basic_precomposition}\ref{lem:basic_precomposition_injective} and \Cref{rmk:range_ACXX_dense}).
	\end{enumerate}
\end{remark}

\subsection{\texorpdfstring{Spectral properties of $A_{C}$}{Spectral properties of the precomposition operator}}

The characterisation of the minimiser $\theta_{\star}$
in terms of an inverse problem motivates us to investigate the spectral
properties of the operator $A_{C}$ before we address the
regularisation and corresponding empirical solutions.

Recall the following definitions of the point, continuous, and residual spectrum of a linear operator $T \in L(\cX)$:
\begin{align*}
	\pSpectrum (T) & = \set{ \lambda \in \Complex }{ T - \lambda \idop_{\cX} \text{ is not injective} } , \\
\cSpectrum (T) & = \set{ \lambda \in \Complex }{ T - \lambda \idop_{\cX} \text{ is injective, not surjective, and } \overline{\range} (T - \lambda \idop_{\cX}) = \cX }  , \\
\rSpectrum (T) & = \set{ \lambda \in \Complex }{ T - \lambda \idop_{\cX} \text{ is injective, not surjective, and } \overline{\range} (T - \lambda \idop_{\cX}) \neq \cX }  .
\end{align*}
For $\lambda \in \Complex$, we write
\[
	\espace_{\lambda} (T) \defeq \ker (T - \lambda \idop_{\cX}) = \set{ v \in \cX }{ T v = \lambda v },
\]
which is a non-trivial subspace of $\cX$ if and only if $\lambda \in
\pSpectrum(T)$.

\begin{theorem}[Spectral properties of the precomposition operator]
	\label{thm:spectral_properties_of_A_C}
    Let $C \in L(\cX)$ and consider the operator $A_{C} \colon L(\cX, \cY) \to L(\cX, \cY)$.
	\begin{enumerate}[label=(\alph*)]
        \item
	    \label{thm:spectral_properties_of_A_C_spectra}
        The point spectra of $A_{C}$ and $C^{\ast}$ coincide and
        $\espace_{\lambda} (A_{C}) = \cY \otimes \espace_{\lambda} (C^{\ast})$.
        More precisely, for $\lambda \in \Complex$, $v \in \cX$, and $y \in \cY$,
        \[
        	\left.
        		\begin{array}{c}
        			\lambda \in \pSpectrum (C^{\ast}), \\
        			v \in \espace_{\lambda} (C^{\ast}), \\
        			y \in \cY
        		\end{array}
        	\right\}
        	\iff
        	\left\{
        		\begin{array}{c}
        			\lambda \in \pSpectrum (A_{C}), \\
        			y \otimes v \in \espace_{\lambda} (A_{C}) .
        		\end{array}
        	\right.
        \]
    \end{enumerate}
    Moreover, the operator
    $A_{C} \colon S_2(\cX, \cY) \to S_2(\cX, \cY)$
    admits the following properties:
	\begin{enumerate}[label=(\alph*), resume]
        \item
        \label{thm:spectral_properties_of_A_C_adjoint}
        The adjoint of $A_C$
        with respect to $\innerprod{ \quark }{ \quark }_{S_2(\cX, \cY)}$
        is given by $A_C^{\ast} = A_{C^{\ast}}$.
        In particular, if $C$ is self-adjoint with
        respect to $\innerprod{ \quark }{ \quark }_{\cX}$,
        then $A_{C}$ is self-adjoint with respect to
        $\innerprod{ \quark }{ \quark }_{S_2(\cX, \cY)}$.
        \item
		The spectra of $A_{C}$ and $C^{\ast}$ coincide, i.e.
        \begin{align*}
            \pSpectrum (A_{C}) &= \pSpectrum (C^{\ast}),
            \quad \text{(already covered in \ref{thm:spectral_properties_of_A_C_spectra})} \\
        	\cSpectrum (A_{C}) & = \cSpectrum (C^{\ast}) , \\
        	\rSpectrum (A_{C}) & = \rSpectrum (C^{\ast}) .
        \end{align*}

        \item
        \label{thm:spectral_properties_of_A_C_decomposition}
        Let $C$ be compact and self-adjoint with spectral
        decomposition
        \[
            C = \sum_{\lambda \in \pSpectrum(C) } \lambda P_{\espace_\lambda(C)},
        \]
        where $P_{\espace_\lambda(C)} \colon \cX \to \cX$ denotes the orthogonal projection
        operator onto $\espace_\lambda(C)$ and the above
        series expression converges in operator norm.
        Then $A_C$ has the spectral decomposition given by
        \[
            A_C = \sum_{ \lambda \in \pSpectrum(C)}
            \lambda P_{ \cY \otimes \espace_\lambda(C)},
        \]
        where
        $P_{ \cY \otimes \espace_\lambda(C)} \colon S_2(\cX, \cY) \to S_2(\cX, \cY)$
        denotes the orthogonal projection operator onto
        $\cY \otimes \espace_\lambda(C)$ and the above series converges in
        operator norm.
		\item
        \label{thm:spectral_properties_of_A_C_compact}
        Let $C$ be compact and self-adjoint.
        $A_{C}$ is compact if and only if $\cY$ is finite-dimensional.
    \end{enumerate}
\end{theorem}

\begin{proof}
	\begin{enumerate}[label=(\alph*)]
		\item We first characterise the nullspace $\ker(A_C)$.
        Let $\theta \in L(\cX, \cY)$.
        Then
            \begin{align*}
                0 = A_C[\theta] = \theta C
                & \iff
                \range(C) \subseteq \ker(\theta) \\
                & \iff
                \overline{ \range}(C) \subseteq \ker(\theta) &
                \text{(since $\ker(\theta)$ is closed)} \\
                & \iff
                \ker(C^{\ast})^{\perp} \subseteq \ker(\theta),
            \end{align*}
            which is again equivalent to the condition that
            $\theta P^\perp_{\ker(C^{\ast})} = 0$,
            where $P_{\ker(C^{\ast})} \in L(\cX)$ is the orthogonal projection operator
            onto $\ker(C^{\ast})$ and
            $P^\perp_{\ker(C^{\ast})} \defeq \idop_{\cX} - P_{\ker(C^{\ast})}$
            its complement.
            In total, for every $C \in L(\cX)$,
            the nullspace $\ker(A_C)$ consists exactly of the operators
            $\theta \in L(\cX, \cY)$ which satisfy
            \[
                \theta = \theta P_{\ker(C^{\ast})}.
            \]
            Let now $\lambda \in \Complex$.
            Since we can easily verify
            $A_C - \lambda \idop_{L(\cX, \cY)}= A_{C - \lambda \idop_{\cX}}$,
            we can apply the above reasoning and obtain
            \begin{align*}
                \espace_\lambda( A_C)
                &=
                \ker( A_C -\lambda \idop_{L(\cX, \cY)} )
                = \ker(A_{C - \lambda \idop_{\cX}}) \\
                &=
                \{ \theta \in L(\cX, \cY) \mid \theta
                =  \theta P_{\ker(C^{\ast} - \lambda \idop_{\cX})} \} \\
                &=
                \{ \theta \in L(\cX, \cY) \mid \theta
                = \theta P_{ \espace_\lambda(C^{\ast})} \}.
            \end{align*}
            It is straightforward to verify that
            this set is exactly given by
            tensors in $\cY \otimes \espace_\lambda(C^{\ast})$, proving the claim.

		\item Let $\theta_{1}, \theta_{2} \in S_{2}(\cX, \cY)$
        and let $C \in L(\cX)$.
        Then
		\begin{align*}
            \innerprod{ A_C[\theta_{1}] }{ \theta_{2} }_{S_{2} (\cX, \cY)}
            & = \trace ( ( \theta_1 C)^\ast \theta_{2}  ) \\
            & = \trace( C^{\ast} \theta_1^\ast \theta_2 ) \\
            & = \trace( \theta_1^{\ast}  \theta_2 C^\ast ) &
            \text{(cyclic invariance of trace)} \\
            & = \trace( \theta_1^{\ast} A_{C^\ast}[\theta_2] ) \\
            & = \innerprod{ \theta_{1} }{ A_{C^{\ast}}[\theta_{2}] }_{S_{2} (\cX, \cY)}.
        \end{align*}

        \item Let $\lambda \in \Complex$.
        We have already covered the equality of the point spectra
        of $A_C$ and $C^{\ast}$ in
        \ref{thm:spectral_properties_of_A_C_spectra}. We now
        address the continuous and the residual spectrum.
        First, we note that $C^{\ast} - \lambda \idop_{\cX}$ is injective
        if and only if
        $A_{C - \lambda \idop_{\cX}} = A_C - \idop_{S_2(\cX, \cY)}$ is
        injective by
        \Cref{lem:basic_precomposition}\ref{lem:basic_precomposition_injective}.
        Furthermore,
        \begin{align*}
            \overline{\range}( C^{\ast} - \lambda \idop_{\cX} )
            &=
            \ker( C - \lambda \idop_{\cX} )^\perp \text{ and} \\
            \overline{\range}( A_{C} - \lambda \idop_{S_2(\cX, \cY)} )
            &=
            \ker( A^\ast_{{C} - \lambda \idop_{\cX}} )^\perp
            =
            \ker( A_{C^{\ast} - \lambda \idop_{\cX}} )^\perp,
        \end{align*}
        where we use part
        \ref{thm:spectral_properties_of_A_C_adjoint} of this theorem and
        $A_{C} - \lambda \idop_{S_2(\cX, \cY)} = A_{C - \lambda \idop_{\cX}}$
        in the second line.
        We have
        $
        \ker( C - \lambda \idop_{\cX} ) = \{0\}
        \iff
        \ker( A_{C^{\ast} - \lambda \idop_{\cX}} )= \{0\}
        $
        by
        \Cref{lem:basic_precomposition}\ref{lem:basic_precomposition_injective}.
        This implies that
        \[
        	\overline{\range}( C^{\ast} - \lambda \idop_{\cX} ) = \cX
			\iff
			\overline{\range}( A_{C} - \lambda \idop_{S_2(\cX, \cY)} ) = S_2(\cX, \cY) .
		\]
        Therefore, as claimed,
        $\lambda \in \cSpectrum(C^{\ast}) \iff \lambda \in \cSpectrum(A_C)$
        and
        $\lambda \in \rSpectrum(C^{\ast}) \iff \lambda \in \rSpectrum(A_C)$.
        \item Applying \ref{thm:spectral_properties_of_A_C_spectra}
          to the spectral decomposition of $C$ immediately proves the claim.

        \item By the spectral theorem for self-adjoint operators,
          $A_C$ is compact if and only if each of its eigenspaces
          is finite-dimensional and its spectral
          decomposition series converges in operator norm.
          Combined with the fact that
          $\espace_\lambda(A_C) = \cY \otimes \espace_\lambda(C)$ is
          finite-dimensional if and only if $\cY$ is finite-dimensional, this
          proves the claim.
	\end{enumerate}\vspace{-\baselineskip}
\end{proof}

An important consequence of the above discussion is that,
once $C$ is self-adjoint and compact, the self-adjoint
but generally non-compact $A_{C}$ 
can be conveniently manipulated (and, in particular, regularised) 
using the functional calculus \citep[Chapter~VII]{ReedSimon1980}.

\begin{corollary}[Compatibility with the functional calculus]
	\label{cor:functional_calculus}
    Let $C \in L(\cX)$ be compact and self-adjoint with the spectral decomposition
    $C = \sum_{\lambda \in \pSpectrum(C)} \lambda P_{\espace_{\lambda} (C)}$.
    If $g \colon \Reals \to \Reals$ is extended to act on
    self-adjoint Hilbert space operators
    with a discrete spectrum in terms of their spectral decompositions via
    \[
    	g(C) \defeq \sum_{\lambda \in \pSpectrum(C)} g (\lambda) P_{\espace_{\lambda} (C)} ,
    \]
    then $A_{C}$ as an operator on $S_{2} (\cX, \cY)$ satisfies
    \[
    	A_{g(C)} = g(A_{C}) ,
    \]
    with each series converging in operator norm if and only if the other does.
\end{corollary}

\begin{proof}
	By \Cref{thm:spectral_properties_of_A_C}, $A_{C} \colon S_{2} (\cX, \cY) \to S_{2} (\cX, \cY)$ has the spectral decomposition
    \[
    A_{C} = \sum_{\lambda \in \pSpectrum(C)} \lambda P_{\cY \otimes \espace_{\lambda} (C)}.
    \]
    Thus,
	\begin{align*}
		g(A_{C})
		= \sum_{\lambda \in \pSpectrum(C)} g(\lambda) P_{\cY \otimes \espace_{\lambda} (C)}
		= A_{g(C)} ,
	\end{align*}
	which proves the claim.
\end{proof}

\section{Regularisation of Hilbert--Schmidt regression}
\label{sec:regularisation}

Having discussed the general spectral theory
of the precomposition operator $A_C$ for some bounded
operator $C \in L(\cX)$, we now
return to our setting of Hilbert--Schmidt regression
and the Hilbert space inverse problem \eqref{eq:inverse_problem_hs}.
It is noteworthy that the theory of regularisation in
Banach spaces is technically more involved
than the concepts we discuss here \citep[e.g.][]{SchusterEtAl2012}.
In our investigation, we mainly rely on \Cref{thm:spectral_properties_of_A_C}
and \cref{cor:functional_calculus}, as they are sufficient
for a basic understanding of the regularisation of our inverse problem \eqref{eq:inverse_problem_hs}
in $S_2(\cX, \cY)$.

According to \Cref{prop:pseudoinverse_solution}, we may consider
the solution $\theta_{\star} = A_{\CXX}^\dagger[\CYX]$ whenever
we assume that a solution of \eqref{eq:regression_problem_hs} exists.
\Cref{thm:spectral_properties_of_A_C}\ref{thm:spectral_properties_of_A_C_compact} shows that the inverse problem is non-compact whenever $\cY$ is infinite-dimensional.
However, as the covariance operator $\CXX$ is always Hilbert--Schmidt and self-adjoint, the operator $\smash{A_{\CXX}}$ is self-adjoint on $S_2(\cX, \cY)$ and \Cref{thm:spectral_properties_of_A_C}\ref{thm:spectral_properties_of_A_C_decomposition} yields the spectral representation
\begin{equation*}
    A_{\CXX} = \sum_{ \lambda \in \pSpectrum(\CXX)}
    \lambda P_{ \cY \otimes \espace_\lambda(\CXX)}.
\end{equation*}
We now recall the concept of regularisation strategies for inverse problems in
Hilbert spaces.

\begin{definition}
    \label{def:spectral_regularisation_strategy}
	We call a family of functions $g_{\alpha} \colon [0, \infty) \to \Reals$,
    indexed by a \emph{regularisation parameter} $\alpha >0$,
    a \emph{spectral regularisation strategy}
    \citep[Section~4]{EnglHankeNeubauer1996} if it satisfies the following properties for
    all choices of $\alpha$:
	\begin{enumerate}[label=(R\arabic*)]
        \item \label{item:r1}
			$\sup_{\lambda \in [0, \infty)} \absval{ \lambda g_{\alpha}(\lambda)} \leq D$
			for some constant $D$,
        \item \label{item:r2}
			$\sup_{\lambda \in [0, \infty)} \absval{ 1- \lambda g_{\alpha}(\lambda)} \leq \gamma_0$
			for some constant $\gamma_0$, and
        \item \label{item:r3}
			$\sup_{\lambda \in [0, \infty)} \absval{ g_{\alpha}(\lambda) } < B \alpha^{-1}$,
			for some constant $B$.
	\end{enumerate}
\end{definition}

It will be convenient later to have the shorthand notation
\begin{equation}
	\label{eq:residual_of_regularisation}
	r_{\alpha}(\lambda) \defeq 1 - \lambda g_{\alpha}(\lambda)
\end{equation}
for the \emph{residual} associated to the regularisation scheme $g_{\alpha}$, and both $g_{\alpha}$ and $r_{\alpha}$ will be applied to compact self-adjoint operators using the continuous functional calculus.
Additionally, we define the \emph{qualification}
of $g_{\alpha}$ as the maximal $q$ such that
\begin{equation}
    \label{eq:qualification}
    \sup_{\lambda \in [0, \infty)}
    \lambda^{q} \absval{ r_{\alpha}(\lambda)}
    \equiv
    \sup_{\lambda \in [0, \infty)}
    \lambda^{q} \absval{ 1- \lambda g_{\alpha}(\lambda)}
    \leq \gamma_q \alpha^q
\end{equation}
for some constant $\gamma_q$ which does not depend on $\alpha$.
If \eqref{eq:qualification} holds for every positive number $q$,
then we say the regularisation strategy has arbitrary qualification.
The above requirements for regularisation strategies are also
commonly found in the context of learning theory
(see e.g.\ \citealp{BauerEtAl2007}; \citealp{GerfoEtAl2008};
\citealp{DickerEtAl2017}; \citealp{BlanchardMuecke2018}).

We call $g_{\alpha}( \smash{A_{\CXX}} )$ the \emph{regularised inverse} of $A_{\CXX}$.
According to \cref{cor:functional_calculus},
\begin{equation*}
    g_{\alpha}(A_{\CXX})
    = A_{g_{\alpha}(\CXX)}
    = \sum_{\lambda \in \pSpectrum(\CXX)}
    g_{\alpha}(\lambda) P_{\cY \otimes \espace_{\lambda} (\CXX)},
\end{equation*}
i.e.\ the regularisation of \smash{$A_{\CXX}$} is exactly given by the precomposition operation associated with the correspondingly regularised covariance operator.
We may think of the regularised inverse $g_{\alpha}( \smash{A_{\CXX}} )$ as approximating the pseudoinverse $\smash{A_{\CXX}^{\dagger}}$ pointwise on its domain as $\alpha \to 0$.

\paragraph{Regularised population solution.}
The regularised inverse $g_{\alpha}(A_{\CXX})$ is bounded for every $\alpha$ and hence we may define the \emph{regularised population solution} to \eqref{eq:inverse_problem_hs} as
\begin{equation}
    \tag{REG}
    \label{eq:regularised_population_solution}
    \theta_{\alpha}
    \defeq
    g_{\alpha}( A_{\CXX} ) [ \CYX ]
    =
    \CYX \, g_{\alpha}( \CXX ).
\end{equation}
Note that as $\CYX$ is always trace class, so is $\theta_{\alpha}$.

\begin{remark}[Regularisation strategies]
	The solution arising from standard regularisation strategies leads to
    well-known statistical methodologies.
	Two classical examples are \emph{ridge regression}, which is obtained from
	the \emph{Tikhonov--Phillips regularisation}
    $g_{\alpha}(\lambda) \defeq (\alpha + \lambda )^{-1}$,
    and \emph{principal component regression}, which is
    obtained from the \emph{spectral truncation}
    $g_{\alpha} (\lambda) \defeq \lambda^{-1} \one [ \lambda > \alpha ]$.
    Tikhonov--Phillips regularisation satisfies
    our definition of a spectral regularisation strategy with
    $D = \gamma_0 = B = 1$ and has qualification $q=1$ with
    $\gamma_q = 1$. Principal component regression
    satisfies
    $D = \gamma_0 = B = 1$ as well but has arbritrary qualification with
    $\gamma_q = 1$.
	Both of these approaches are classically used in finite-dimensional regression
    as well as kernel regression
	(\citealp{DickerEtAl2017}; \citealp{BlanchardMuecke2018}).
	We show in \Cref{sec:kernel_regression} that kernel regression
    can be formulated in terms of our infinite-dimensional linear regression problem.
	Moreover, we emphasise that this perspective yields a
    formalism for gradient descent learning, as it is well-known that
    various iterative regularisation strategies lead to specific
    implementations of gradient descent (see e.g.\ \citealp{YaoEtAl2007}
    for the context of \emph{Landweber iteration}).
\end{remark}

\paragraph{Empirical solution.}
The random variables $X$ and $Y$ are in practice only accessible through sample pairs $(X_{i}, Y_{i}) \in \cX \times \cY$ for $i = 1, \dots, n$.
For simplicity, we assume that these sample pairs are obtained i.i.d.\ from the joint law $\law(X,Y)$.
However, ergodic sampling schemes for time series applications or Markov chain Monte Carlo methods
or even deterministic approaches in the context of quasi-Monte Carlo methods would be interesting to explore here.

Given the sample pairs above, we define the \emph{empirical covariance operators} by
\begin{align*}
	\hatCXX^{}  \defeq \frac{1}{n} \sum_{i = 1}^{n} X_{i} \otimes X_{i}
    \text{ and }
	\hatCYX^{}  \defeq \frac{1}{n} \sum_{i = 1}^{n} Y_{i} \otimes X_{i}.
\end{align*}
Note that $\hatCXX$ and $\widehat{C}_{XY}$ are $\prob$-a.s.\ of rank at most $n$.
We obtain the \emph{regularised empirical solution} straightforwardly as the empirical
analogue of the regularised population solution \eqref{eq:regularised_population_solution} for some regularisation parameter $\alpha$, i.e.
\begin{equation*}
    \tag{EMP}
    \label{eq:empirical_solution}
    \widehat{\theta}_{\alpha}
    \defeq
    g_{\alpha}( A_{\hatCXX} ) [ \hatCYX ]
    =
    \hatCYX \, g_{\alpha}( \hatCXX ).
\end{equation*}

\subsection{Risk and excess risk}

We introduce the shorthand notation
$R(\theta) \defeq \E [ \norm{ Y - \theta X }_{\cY}^{2} ]$
for the risk associated with some $\theta \in S_2(\cX, \cY)$.
The Pythagorean theorem in $L^2(\prob; \cY)$ implies that the risk decomposes as
\begin{align}
    R(\theta)
    &= \norm{ \Pi (Y - \theta X) }_{L^2(\prob; \cY)}^2
    + \norm{ \Pi^\perp (Y - \theta X) }_{L^2(\prob; \cY)}^2 \nonumber \\
    &= \norm{ \E [ Y | X ] - \theta X }_{L^2(\prob; \cY)}^2
    + \norm{ Y - \E [ Y | X ] }_{L^2(\prob; \cY)}^2 .
    \label{eq:risk_decomposition}
\end{align}
Here, just as in \cref{ex:linear_model}, $\Pi \colon L^2(\prob; \cY) \to L^2(\prob; \cY)$,
denotes the orthogonal projection operator onto
$L^2(\Omega, \sigalg^{X}, \prob; \cY)$ such that $\Pi Y \defeq E[Y | X]$,
and $\smash{ \Pi^\perp \defeq \idop_{L^2(\prob; \cY)} - \Pi}$.
The second summand in the risk decomposition \eqref{eq:risk_decomposition}
is an irreducible noise term associated with the
learning problem itself and is independent of the choice of $\theta$.

We introduce the \emph{excess risk} of
the estimator $\widehat{\theta}_\alpha$ as
\[
    \mathcal{E}(\widehat{\theta}_\alpha) \defeq
    R(\widehat{\theta}_\alpha)
    - \inf_{\theta \in S_2(\cX, \cY)} R(\theta) \geq 0.
\]

\begin{remark}[Excess risk and notation]
	\label{rmk:excess_risk_and_sample_notation}
	Note that as we interpret the estimator $\widehat{\theta}_\alpha$
    as a random variable depending on the sample pairs $(X_i, Y_i)$.
	Consequently, $R(\widehat{\theta}_\alpha)$ and $\mathcal{E}(\widehat{\theta}_\alpha)$
	are also interpreted as random variables depending on $(X_i, Y_i)$.
	However, note that the risk and the excess risk contain the
    expectation operator $\E$ with respect to $X$ and $Y$.
	To prevent ambiguity and confusion, we
    introduce the convention that we write expectations and probabilities
    with respect to $(X_i, Y_i)$ distributed according to
    the product law $\law(X, Y)^{\otimes n}$ on the measurable product space
    $((\cX \times \cY)^n, \Borel_{\cX \times \cY}^{\otimes n})$ as $\E^{\otimes n}$ and
    $\prob^{\otimes n}$, while $\E$ and $\prob$ will
    always be interpreted with respect to $(X,Y)$.
\end{remark}

\begin{lemma}[Excess risk]
	\label{lem:excess_risk}
	Assume that the minimiser
    $\theta_{\star} = A^\dagger_{\CXX} [\CYX] \in S_2(\cX, \cY)$ exists.
	Then
	\begin{equation*}
		\mathcal{E}( \widehat{\theta}_\alpha )
		\leq
		\bignorm{ \theta_{\star} X - \widehat{\theta}_\alpha X }^2_{L^2(\prob; \cY)}
		+
		2 M_\star
		\bignorm{ \theta_{\star} X - \widehat{\theta}_\alpha X }_{L^2(\prob; \cY)}
        \quad \prob^{\otimes n}\text{-a.s.},
	\end{equation*}
	where $M_\star \defeq \norm{ \E [ Y | X ] - \theta_{\star} X }_{L^2(\prob; \cY)}$
    is the \emph{misspecification error} of the linear regression problem.
\end{lemma}

\begin{remark}[Well-specified case]
	Note that $M_\star = 0$ if and only if the problem is well specified,
    i.e.\ the LCE representation $\E[Y | X ] = \theta_{\star} X$
    holds $\prob$-a.s.\ with
    $\theta_{\star} \in S_2(\cX, \cY)$ --- or, equivalently, if the linear
    model \eqref{eq:linear_model} holds with $\theta_{\star} \in S_2(\cX, \cY)$.
\end{remark}

\begin{proof}[Proof of \Cref{lem:excess_risk}]
    All inequalities in this proof are to be understood in the $\prob^{\otimes n}$-a.s.\ sense.
	Assuming that the minimiser
    $\theta_{\star} = A^\dagger_{\CXX} [\CYX] \in S_2(\cX, \cY)$ exists.
	we can insert the risk decomposition \eqref{eq:risk_decomposition}
    into the definition of the excess risk
    $\mathcal{E}( \widehat{\theta}_\alpha )$ and obtain
	\begin{align*}
		\mathcal{E}( \widehat{\theta}_\alpha )
		& =
		R(\widehat{\theta}_\alpha) - R(\theta_{\star}) \\
		& \leq
		\bignorm{ \E [ Y | X ] - \widehat{\theta}_\alpha X }^2_{L^2(\prob; \cY)}
		-
		\bignorm{ \E [ Y | X ] - \theta_{\star} X }^2_{L^2(\prob; \cY)} \\
		&\leq
		\left(
		\bignorm{ \E [ Y | X ] - \theta_{\star} X }_{L^2(\prob; \cY)}
		+
		\bignorm{ \theta_{\star} X - \widehat{\theta}_\alpha X }_{L^2(\prob; \cY)}
		\right)^2
		-
		\norm{ \E [ Y | X ] - \theta_{\star} X }^2_{L^2(\prob; \cY)} \\
		&\leq
		\bignorm{ \theta_{\star} X - \widehat{\theta}_\alpha X }^2_{L^2(\prob; \cY)}
		+
		2
		\underbrace{
		\norm{ \E [ Y | X ] - \theta_{\star} X }_{L^2(\prob; \cY)}
		}_{ \qefed M_\star } \,
		\bignorm{ \theta_{\star} X - \widehat{\theta}_\alpha X }_{L^2(\prob; \cY)} ,
	\end{align*}
	as claimed.
\end{proof}

Combining the bound in \Cref{lem:excess_risk} with the fact that
\begin{align*}
	\norm{ \theta_{\star} X - \widehat{\theta}_\alpha X }^2_{L^2(\prob; \cY)}
	& =
	\E [ \norm{ (\theta_{\star} - \widehat{ \theta }_\alpha ) X }_\cY^2 ] \\
	& =
	\trace \left(
	(\theta_{\star} - \widehat{ \theta }_\alpha)
	\CXX
	(\theta_{\star} - \widehat{ \theta }_\alpha)^{\ast}
	\right) \\
    & =
    \bignorm{ (\theta_{\star} - \widehat{\theta}_\alpha ) \CXX^{1/2} }^2_{S_2(\cX, \cY)}
    \quad \prob^{\otimes n}\text{-a.s.},
\end{align*}
we see that the performance of
$\widehat{\theta}_\alpha$ may be assessed by bounding the quantity
\begin{equation*}
    \bignorm{ (\theta_{\star} - \widehat{\theta}_\alpha ) \CXX^{s} }_{S_2(\cX, \cY)}
    =
    \bignorm{ A^s_{\CXX}[\theta_{\star} - \widehat{\theta}_\alpha ] }_{S_2(\cX, \cY)}
    \text{ for }
    0 \leq s \leq 1/2
\end{equation*}
either with high $\prob^{\otimes n}$-probability
or in terms of moment bounds with respect to $\E^{\otimes n}$.
The case $s = 0$ corresponds to the classical
\emph{reconstruction error}
$\norm{ \theta_{\star} - \widehat{\theta}_\alpha }_{S_2(\cX, \cY)}$.
A similar approach can be found in the literature
on kernel regression with scalar response, see e.g.\ \citet{BlanchardMuecke2018}.

Alternatively, it is also possible
to investigate the performance of the estimator
in terms of the weaker operator norm
$\norm{ \theta_{\star} - \widehat{\theta}_\alpha }_{\cX \to \cY}$, since
\begin{equation*}
    \bignorm{ (\theta_{\star} - \widehat{\theta}_\alpha) \CXX^{1/2} }^2_{S_2(\cX, \cY)}
	 \leq
	\tr ( \CXX ) \,
	\bignorm{ \theta_{\star} - \widehat{\theta}_\alpha }^2_{\cX \to \cY} \quad
    \quad \prob^{\otimes n}\text{-a.s.}
\end{equation*}

\subsection{Source conditions}

We assume that
$\theta_{\star} \defeq A^\dagger_{\CXX} [ \CYX ]$ exists,
i.e., by \Cref{prop:pseudoinverse_solution}, $\theta_{\star}$ is the unique operator of minimal
norm in $S_2(\cX, \cY)$ solving
\eqref{eq:regression_problem_hs}.
In general, the convergence of
\smash{$\norm{ (\theta_{\star} - \widehat{\theta}_\alpha ) \CXX^{s} }_{S_2(\cX, \cY)}$}
and
$\norm{ \theta_{\star} - \theta_{\alpha}}_{\cX \to \cY}$
as $\alpha \to 0$ can be arbitrarily slow.
In order to alleviate this problem, one usually imposes
\emph{source conditions} (e.g., \citealp{EnglHankeNeubauer1996}, Section~3.2)
for the inverse problem \eqref{eq:inverse_problem_hs}, which we are able to express in
terms of the functional calculus derived in \cref{cor:functional_calculus}.
Equivalent conditions are commonly imposed in order to
derive convergence rates in 
functional response regression \citep{BenatiaEtAl2017, KuttaDierickxDette2021}
and kernel regression with vector-valued response variables
\citep{LiEtAl2022, LiEtAl2024, Meunier2024}.

\begin{assumption}[H\"older source condition]
	\label{ass:Hoelder_source_condition}
    For $0 < \nu < \infty$ and $0 < R < \infty$,
    we define the \emph{source set}
    \begin{equation}
    	\label{eq:Hoelder_source_set}
        \Omega(\nu, R)
        \defeq
        \Set{
            A^\nu_{\CXX}[ \theta ] }{
            \theta \in S_2(\cX, \cY),
            \norm{ \theta }_{S_2(\cX, \cY)} \leq R
        }
        \subseteq S_2(\cX, \cY).
    \end{equation}
    We assume that the solution satisfies the \emph{source condition}
    $\theta_{\star} \in \Omega(\nu, R)$.
\end{assumption}

\begin{remark}[Source condition on $\CYX$]
	\label{rem:source_condition}
	Using \Cref{cor:functional_calculus},
	we may rewrite the above source condition as
	the constrained operator factorisation problem
	\begin{equation*}
		\theta_{\star} = \tilde{\theta} \CXX^{\nu}
		\text{ for some } \tilde{\theta} \in S_2(\cX, \cY)
		\text{ with }
		\norm{ \tilde{\theta} }_{S_2(\cX, \cY)} \leq R.
	\end{equation*}
	By assumption, as $\CYX = \theta_{\star} \CXX$, this factorisation
	is clearly equivalent to
	\begin{equation}
		\label{eq:factorisation_hoelder_rhs}
		\CYX = \tilde{\theta} \CXX^{\nu + 1}
		\text{ with }
		\norm{ \tilde{\theta} }_{S_2(\cX, \cY)} \leq R.
	\end{equation}
	This shows that the source condition
	$\theta_{\star} \in \Omega(\nu, R)$ is equivalent to
	\mbox{$\CYX \in \Omega(\nu+1,R)$}.
	In the classical context of inverse problems
	and kernel regression,
	source conditions are usually interpreted as smoothness
	assumptions about the underlying problem.
    In our setting,
	the interpretation is not immediately clear.
	However, we can straightforwardly apply the approach in
	\Cref{prop:existence_hs} to this factorisation problem
	\eqref{eq:factorisation_hoelder_rhs}
	by replacing $A_{\CXX}$ with $A_{\CXX^{\nu + 1}}$ and express
	its solubility in terms of an equivalent moment condition.
\end{remark}

\begin{corollary}[Source condition]
    \label{cor:source_condition}
    The source condition $\theta_{\star} \in \Omega(\nu, R)$ is satisfied if and
    only if
    \begin{equation}
        \label{eq:hoelder_moment}
        \sum_{ i \in I }
        \sup_{x \in \cX}
        \frac{ \absval{ \E[ \innerprod{x}{X}_{\cX} \innerprod{e_i}{Y}_{\cY} ] }^2 }%
        { \norm{ \CXX^{\nu + 1} x }_{\cX}^2 } \leq R^2
    \end{equation}
    for some (indeed, any) complete orthonormal system
    $\{ e_i \}_{i \in I}$ in $\cY$.
\end{corollary}

\begin{proof}
    Proceeding analogously to the proof of
    \Cref{prop:existence_hs}, we see that
    the factorisation \eqref{eq:factorisation_hoelder_rhs} admits a solution
    if and only if
    $\tilde{\theta} \defeq ((\CXX^{\nu + 1})^\dagger \CXY)^\ast$
    satisfies the condition $\norm{ \tilde{\theta} }_{S_2(\cX, \cY)} \leq R$,
    which is clearly equivalent to \eqref{eq:hoelder_moment}.
    Note that the left-hand side of \eqref{eq:hoelder_moment}
    is equivalent to $\norm{ \tilde{ \theta } }_{S_2(\cX, \cY) }^2$
    and hence independent of the choice of complete orthonormal system.
\end{proof}

We note that \cref{cor:source_condition} can be interpeted
as a special case of the so-called 
\emph{Picard criterion} for H\"older source
conditions \citep[Proposition~3.13]{EnglHankeNeubauer1996}.
Under the assumption of suitable source conditions, convergence rates
depending on $\alpha \to 0$ can be derived.
The details of this theory under more general assumptions
are not in the scope of this work, as there is a plethora
of results concerning \emph{general source conditions} in the literature
on inverse problems. However, our results serve as a starting point for the
investigation of source conditions in the context of infinite-dimensional
regression.

\subsection{Convergence analysis}

We demonstrate how our framework allows derivation of rates for Hilbert--Schmidt regression based on H\"older source conditions.
We begin by decomposing the error $\theta_{\star} - \widehat{\theta}_\alpha$ associated with the regularised empirical solution $\widehat{\theta}_\alpha$.
In particular, we are interested both in the Hilbert--Schmidt norm of this error and in the mean-square prediction error\footnote{Note well that the expectation here is with respect to $X$ and that both sides are random quantities, being functions of the sample data $(X_{i}, Y_{i})_{i = 1}^{n}$;
cf.\ \Cref{rmk:excess_risk_and_sample_notation}.}
\[
	\E \bigl[ \bignorm{ \bigl( \theta_{\star} - \widehat{\theta}_\alpha \bigr) X }_{\cY}^{2} \bigl]
	\equiv
	\bignorm{ \bigl( \theta_{\star} - \widehat{\theta}_\alpha \bigr) \CXX^{1/2} }_{S_{2}(\cX, \cY)}^2
\]
and in order to treat these in a unified way we will examine
\[
	\bignorm{ \bigl( \theta_{\star} - \widehat{\theta}_\alpha \bigr) \CXX^{s} }_{S_{2}(\cX, \cY)} \text{ for $0 \leq s \leq \tfrac{1}{2}$.}
\]

\paragraph{Error decomposition.}
Na\"ively, one would decompose this error norm using the triangle inequality as follows:
$\prob^{\otimes n}$-a.s.\ with respect to the samples $(X_{i}, Y_{i})_{i = 1}^{n}$,
\begin{equation}
	\label{eq:error_decomp_naive}
	\bignorm{
		\bigl( \theta_{\star} - \widehat{\theta}_\alpha \bigr) \CXX^{s}
	}_{S_2(\cX, \cY)}
	\leq
	\underbrace{%
	\bignorm{
		(\theta_{\star} - \theta_{\alpha}) \CXX^{s}
	}_{S_2(\cX, \cY)}
	}_{ = \text{approximation error}}
	+
	\underbrace{%
	\bignorm{
		\bigl( \theta_{\alpha} - \widehat{\theta}_\alpha \bigr) \CXX^{s}
	}_{S_2(\cX, \cY)}
	}_{ = \text{variance}}
	.
\end{equation}
However, this decomposition turns out to be less than ideal and we
shall consider the following alternative $\prob^{\otimes n}$-a.s.\ decomposition:
\begin{align*}
	\theta_{\star} - \widehat{\theta}_\alpha
	& = \theta_{\star} - \theta_{\star} \hatCXX g_{\alpha} \bigl( \hatCXX \bigr) + \theta_{\star} \hatCXX g_{\alpha} \bigl( \hatCXX \bigr) - \widehat{\theta}_\alpha \\
	& = \theta_{\star} r_{\alpha} \bigl( \hatCXX \bigr) + \theta_{\star} \hatCXX g_{\alpha} \bigl( \hatCXX \bigr) - \hatCYX g_{\alpha} \bigl( \hatCXX \bigr) \\
	& = \theta_{\star} r_{\alpha} \bigl( \hatCXX \bigr) + \bigl( \theta_{\star} \hatCXX - \hatCYX \bigr) g_{\alpha} \bigl( \hatCXX \bigr) ,
\end{align*}
where $r_{\alpha}$ is as in \eqref{eq:residual_of_regularisation}, and so
\begin{align}
	\notag
	\bignorm{
        \bigl( \theta_{\star} -
        \widehat{\theta}_\alpha \bigr) \CXX^{s}
        }_{S_2(\cX, \cY)}
	& \leq
	\bignorm{ \theta_{\star} r_{\alpha} \bigl( \hatCXX \bigr) \CXX^{s} }_{S_{2}(\cX, \cY)} \\
	\label{eq:error_decomp_smart}
	& \phantom{=} \quad +
	\bignorm{ \bigl( \theta_{\star} \hatCXX - \hatCYX \bigr) g_{\alpha} \bigl( \hatCXX \bigr) \CXX^{s} }_{S_{2}(\cX, \cY)} \quad \text{$\prob^{\otimes n}$-a.s.}
\end{align}
Again, we think of the two terms on the right-hand side of \eqref{eq:error_decomp_smart}
as an \emph{approximation error} and a \emph{variance term}.
Crucially, though, the approximation error in the decomposition
\eqref{eq:error_decomp_smart} is random --- as opposed to the deterministic
approximation term in \eqref{eq:error_decomp_naive} --- and both terms
in \eqref{eq:error_decomp_smart} will be amenable to analysis
using concentration-of-measure techniques.
Our approach combines error decomposition techniques
from kernel-based learning theory by \citet{BlanchardMuecke2018}
with concentration results for sub-exponential random operators
which derive based on recent results by \citet{MaurerPontil2021}.
Before bounding both these error terms, we
introduce sub-Gaussian and sub-exponential norms for Hilbertian
random variables.

\paragraph{Concentration bound for empirical covariance operators.}

We begin by defining notions of \emph{sub-exponentiality} and
\emph{sub-Gaussianity} of real-valued random variables
\citep[e.g.][]{BuldyginKozachenko2000}, which we generalise
to vector-valued random variables.
For a real-valued random variable $\xi$ defined
on $(\Omega, \sigalg, \prob)$,
we introduce the Banach spaces
$L_{\psi_1}(\Omega, \sigalg, \prob; \Reals) = L_{\psi_1}(\prob)$
and
$L_{\psi_2}(\Omega, \sigalg, \prob; \Reals) =   L_{\psi_2}(\prob)$
via the norms
\begin{equation*}
    \norm{ \xi }_{L_{\psi_1}(\prob)} \defeq
    \sup_{ 1 \leq p < \infty }
    \frac{ \norm{ \xi }_{L^p(\prob)} }{ p }
    \text{ and }
    \norm{ \xi }_{L_{\psi_2}(\prob)} \defeq
    \sup_{ 1 \leq p < \infty }
    \frac{ \norm{ \xi }_{L^p(\prob)} }{ p^{1/2} } ;
\end{equation*}
see \citet{MaurerPontil2021}. We extend this definition
to the case that $\xi$ takes values in a separable Hilbert space $\cH$
by defining
\begin{equation*}
    \norm{ \xi }_{L_{\psi_1}(\prob; \cH)} \defeq
    \norm{ \, \norm{ \xi }_\cH \, }_{L_{\psi_1}(\prob)} =
    \sup_{ 1 \leq p < \infty }
    \frac{ \norm{ \xi }_{L^p(\prob;\cH)} }{ p }
\end{equation*}
and analogously for
$\norm{ \xi }_{L_{\psi_2}(\prob; \cH)} \defeq
\norm{ \norm{ \xi }_\cH }_{L_{\psi_2}(\prob)} $.
For the real-valued case, these norms are equivalent
to the usual sub-exponential and sub-Gaussian norms
(see e.g.\ \citealp{Vershynin2018}, Propositions~2.5.2 and 2.7.1).
In the vector-valued case, the sub-Gaussian and sub-exponential norms
of $\xi$
are sometimes defined as the supremum of the real sub-exponential and
sub-Gaussian norms over all
one-dimensional projections $\innerprod{x}{ \xi }_\cH$ for $x \in \cH$ ---
the norms of $L_{\psi_1}(\prob; \cH)$ and $L_{\psi_2}(\prob; \cH)$
discussed here are stronger.
Note that $L_{\psi_2}(\prob; \cH) \subseteq L_{\psi_1}(\prob; \cH)$.

\begin{remark}[Bernstein condition]
Let $\xi$ take values in the separable Hilbert space $\cH$.
In statistical learning theory, the vector-valued \emph{Bernstein condition}
(\citealp{PinelisSakhanenko1985}) given by
\begin{equation}
    \label{eq:bernstein_condition}
    \E \big[ \norm{ \xi - \E[\xi] }_\cH^p \big] \leq \frac{1}{2} p! \sigma^2 L^{p-2}
    \text{ for all } p \geq 2
\end{equation}
with parameters $\sigma < \infty$ and $L < \infty$ is a standard assumption
in order to derive tail bounds for Hilbertian random variables.
Condition \eqref{eq:bernstein_condition} can be interpreted
as the classical real-valued Bernstein condition
applied to the random variable $\norm{\xi - \E[\xi]}_H$
\citep[see e.g.][Section~1.4]{BuldyginKozachenko2000}.
However, it is possible to prove that the
real-valued Bernstein condition is equivalent
to sub-exponentiality (see \Cref{app:subexponentiality}).
Hence, the vector-valued Bernstein condition \eqref{eq:bernstein_condition}
is equivalent to $\norm{ \xi }_{L_{\psi_1}(\prob; \cH)} < \infty$.
\end{remark}

Analogously to the fact that products of real-valued
sub-Gaussians are sub-exponential, we show that
such a property holds for tensor products of sub-Gaussians
in the vector-valued case.

\begin{lemma}[Tensor products of sub-Gaussians are sub-exponential]
    \label{lem:tensor_subexponential}
    We have
    \begin{equation*}
        \norm{ Y \otimes X }_{L_{\psi_1}(\prob; S_2(\cX, \cY))}
        \leq
        2
        \norm{ X }_{L_{\psi_2}(\prob;\cX)}
        \norm{ Y }_{L_{\psi_2}(\prob;\cY)}.
    \end{equation*}
\end{lemma}

\begin{proof}
    We have
    \begin{align*}
        \norm{ Y \otimes X }_{L_{\psi_1}(\prob; S_2(\cX, \cY))}^2
        &=
        \sup_{ 1 \leq p < \infty }
        \frac{
            \E[ \norm{ X }^p_{ \cX } \,
            \norm{ Y }^p_{ \cY }
            \big]^{2/p}
         }{ p^2 } && \text{(by \eqref{eq:norm_of_outer_product})} \\
        &\leq
        \sup_{ 1 \leq p < \infty }
        \frac{
            \E\big[ \norm{ X }^{2p}_{ \cX } \big]^{1/p} \,
            \E \big[\norm{ Y }^{2p}_{ \cY } \big]^{1/p}
         }{ p^2 } && \text{(Cauchy--Schwarz)} \\
        &=
        4
        \sup_{ 1 \leq p < \infty }
        \frac{ \norm{ X }^2_{L^{2p}(\prob; \cX)} \,
               \norm{ Y }^2_{L^{2p}(\prob; \cY)}
         }{  4 p^2 } \\
         &\leq
         4
         \norm{ X }^2_{L_{\psi_2}(\prob; \cX)} \,
         \norm{ Y }^2_{L_{\psi_2}(\prob; \cY)} .
        \end{align*}
        Taking square roots completes the proof.
\end{proof}

We now summarise the concentration-of-measure results that will be
used to control both the approximation error and the variance.
We call particular attention to the fact that, although these are
stated as five inequalities, some are deterministic consequences of
others (e.g.\ \eqref{eq:simul_conc_hatCXX} follows
from \eqref{eq:simul_conc_covariance} pointwise for all events
in a subset of the underlying sample space),
and this allows us to minimise the number
of overall appeals to concentration of measure and the union bound.

\begin{theorem}[Simultaneous concentration bounds]
	\label{thm:simul_conc}
	Suppose that $X \in L_{\psi_{2}}(\prob; \cX)$, $Y \in L_{\psi_{2}}(\prob; \cY)$,
    and that $(X_1, Y_1), \dots, (X_n, Y_n)$ are sampled i.i.d.\ from $\law(X, Y)$,
    i.e.\ with joint distribution $\prob^{\otimes n}$.
    Let $\alpha \in (0, 1)$, $\delta \in (0, \frac{1}{2}]$, $n \geq \log ( 1 / \delta )$,
    and $r \in [0, 1]$.
	Also let $T \in L(\cX, \cY)$.
	Then, with $\prob^{\otimes n}$-probability at least $1 - 2 \delta$,
    all the following bounds hold simultaneously (i.e.\ for the same subset
    of the underlying sample space):
	\begin{align}
		\label{eq:simul_conc_covariance}
		\bignorm{ \hatCXX - \CXX }_{S_{2}(\cX, \cY)}
		& \leq
        24 \sqrt{2} e \norm{ X }_{L_{\psi_{2}}(\prob; \cX)}^{2} \sqrt{\frac{\log(1/\delta)}{n}} , \\
		\label{eq:simul_conc_hatCXX}
		\bignorm{ \hatCXX }_{S_2(\cX)}
		& \leq \norm{ \CXX }_{S_2(\cX)} +
        24 \sqrt{2} e \norm{ X }_{L_{\psi_2}(\prob; \cX)}^2 \sqrt{\frac{ \log(1 / \delta)}{n}} , \\
		\label{eq:simul_conc_empirical_idop}
		2 & \geq \bignorm{ \bigl( \hatCXX + \alpha \idop_\cX \bigr)^{-1} ( \CXX + \alpha \idop_\cX ) }_{L(\cX)} , \\
		\label{eq:simul_conc_T_CXXs}
		\norm{ T \CXX^{r} }_{L(\cX, \cY)}
		& \leq 2^r \cdot
		\bignorm{ T \bigl( \hatCXX + \alpha \idop_\cX \bigr)^r }_{L(\cX, \cY)} ,
	\end{align}
	and,
	\begin{align}
		\label{eq:simul_conc_odd_residual}
		\bignorm{ \bigl( \theta_\star \hatCXX - \hatCYX \bigr)
		(\CXX + \alpha \idop_\cX)^{-1/2} }_{S_2(\cX, \cY)}
		& \leq \frac{16 \sqrt{2} e B_{\psi_2}}{\sqrt{\alpha}} \cdot \sqrt{\frac{\log(1/\delta)}{n}} ,
	\end{align}
	where
	\begin{equation}
		\label{eq:Bpsi2}
		B_{\psi_2} \defeq \norm{\theta_\star }_{L(\cX, \cY)} \norm{X}^2_{L_{\psi_2}(\prob; \cX)} + \norm{X}_{L_{\psi_2}(\prob; \cX)} \norm{Y}_{L_{\psi_2}(\prob; \cY)} ,
	\end{equation}
	and inequalities \eqref{eq:simul_conc_empirical_idop} and \eqref{eq:simul_conc_T_CXXs} require the additional assumption that
    \begin{equation}
		\label{eq:alpha_conc}
		n \geq \max
		\left\{ 1 , \frac{1152 \cdot e^2
			\norm{X}^4_{L_{\psi_2}(\prob;\cX)}}{\alpha^2}
		\right\}
		\cdot \log(1/\delta).
    \end{equation}
\end{theorem}

\begin{proof}
	Inequalities \eqref{eq:simul_conc_covariance} to \eqref{eq:simul_conc_odd_residual}
    are proven separately in the appendix as
    \hyperref[lem:concentration_empirical_covariance]{Lemmas} \ref{lem:concentration_empirical_covariance} to \ref{lem:concentration_of_odd_residual} respectively.
	Bearing in mind the dependency structure of these results, as illustrated
    in \Cref{fig:concentration_dependency_graph}, we see that there are only
    two statements that each hold with probability at least $1 - \delta$,
    namely \eqref{eq:simul_conc_covariance} and \eqref{eq:simul_conc_odd_residual},
    and the remainder are pointwise corollaries of those two.
	Thus, by the union bound, the complete set of inequalities
    holds simultaneously on a set of $\prob^{\otimes n}$-probability
    at least $1 - 2 \delta$.
\end{proof}

We remark that the condition \eqref{eq:alpha_conc}
is deliberately chosen such that it implies the bound
\begin{equation}
\label{eq:alpha_consequence}
24 \sqrt{2} e \norm{ X }_{L_{\psi_2}(\prob; \cX)}^2 \sqrt{\frac{ \log(1 / \delta)}{n}}
\leq \alpha.
\end{equation}
We make use of this implication multiple times.

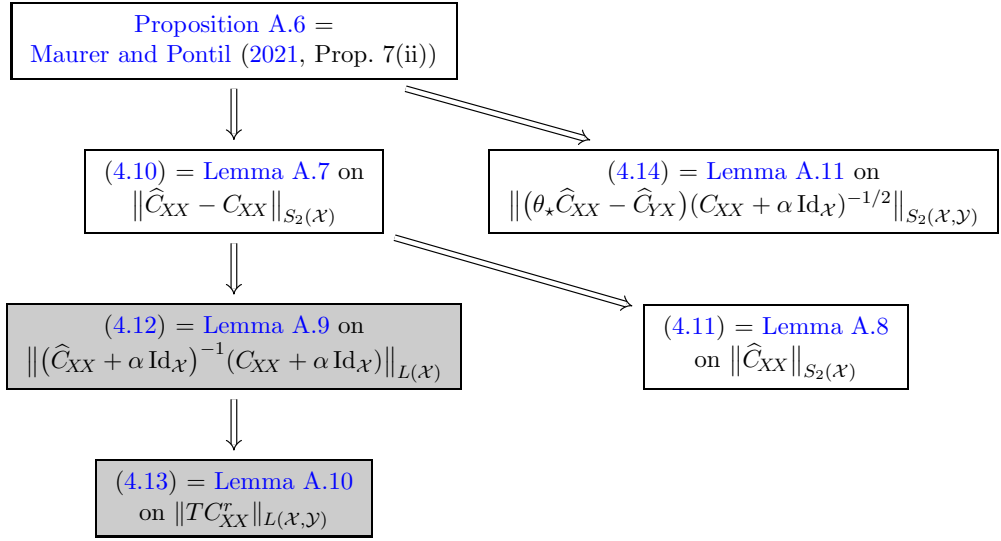
\begin{figure}[t]
	\centering
	\newcommand{\twoliner}[3]{\fcolorbox{black}{#1}{$\begin{array}{c} \text{#2} \\ \text{#3} \end{array}$}}
    \begin{tikzcd}[scale cd=.85]%
		\twoliner{white}{\Cref{prop:maurer_pontil} $=$}{\citet[Prop.~7(ii)]{MaurerPontil2021}} \ar[Rightarrow, d] \ar[Rightarrow, dr] & ~ \\
		\twoliner{white}{\eqref{eq:simul_conc_covariance} $=$ \Cref{lem:concentration_empirical_covariance} on}{$\bignorm{ \hatCXX - \CXX }_{S_{2}(\cX)}$} \ar[Rightarrow, dr] \ar[Rightarrow, d] & \hspace{-1cm}\twoliner{white}{\eqref{eq:simul_conc_odd_residual} $=$ \Cref{lem:concentration_of_odd_residual} on}{$\bignorm{ \bigl( \theta_{\star} \hatCXX - \hatCYX \bigr) ( \CXX + \alpha \idop_{\cX} )^{-1/2} }_{S_{2}(\cX, \cY)}$} \\
		\twoliner{black!20}{\eqref{eq:simul_conc_empirical_idop} $=$ \Cref{lem:concentration_empirical_idop} on}{$\bignorm{ \bigl( \hatCXX + \alpha \idop_{\cX} \bigr)^{-1} ( \CXX + \alpha \idop_{\cX} ) }_{L(\cX)}$} \ar[Rightarrow, d] & \twoliner{white}{\eqref{eq:simul_conc_hatCXX} $=$ \Cref{lem:concentration_norm}}{on $\bignorm{ \hatCXX }_{S_{2}(\cX)}$} \\
		\twoliner{black!20}{\eqref{eq:simul_conc_T_CXXs} $=$ \Cref{lem:weighted_norms}}{on $\norm{ T \CXX^{r} }_{L(\cX, \cY)}$} & ~ \\
	\end{tikzcd}
	\caption{The dependency structure of the concentration results used in this paper.
	The implications following on from \Cref{lem:concentration_empirical_covariance} hold pointwise in the sample space, and so two appeals to \Cref{prop:maurer_pontil}, each valid with $\prob^{\otimes n}$-probability at least $1 - \delta$, for $n \geq \log(1/\delta)$, suffice to establish \emph{all} the necessary bounds simultaneously with $\prob^{\otimes n}$-probability at least $1 - 2 \delta$ (by the union bound).
	The two shaded bounds require an additional lower bound on $n$, given in
    \eqref{eq:alpha_conc}.}
	\label{fig:concentration_dependency_graph}
\end{figure}

\paragraph{Bounding the approximation error.}
We now bound the approximation error $\theta_{\star} r_{\alpha} \bigl( \hatCXX \bigr)$ in the $\CXX^{s}$-weighted Hilbert--Schmidt norm using the properties of $g_\alpha$ defined in \Cref{def:spectral_regularisation_strategy}
and the qualification introduced in \eqref{eq:qualification};
we introduce the shorthand $\bar{\gamma} \defeq \max\{\gamma_0 , \gamma_q\}$.
In what follows, $\kappa_{a,b,c}$ denotes a positive, finite constant depending on some generic quantities $a,b,c$.

\begin{proposition}[High-probability bound on approximation error in \eqref{eq:error_decomp_smart}]
	\label{prop:approx-error}
	In addition to the assumptions of \Cref{thm:simul_conc},
	suppose the regularisation strategy $g_\alpha$ has
	qualification $q \geq \nu + s$.
	Suppose that
	$\theta_\star \in \Omega(\nu , R)$ and
	$s \in [0, \tfrac{1}{2}]$ and assume
	that $n$ satisfies \eqref{eq:alpha_conc}.
	There exists a constant
	$\kappa_{\nu, \CXX} < \infty$,
	exclusively depending on $\nu$ and $\norm{\CXX}_{L(\cX)}$, such that
	with $\prob^{\otimes n}$-probability at least $1 - \delta$,
	\begin{equation*}
	\norm{ \theta_\star r_\alpha \bigl( \hatCXX \bigr)  \CXX^{s} }_{S_2(\cX, \cY)}
	\leq \kappa_{\nu, \CXX} \bar{\gamma} R\alpha^s
	\left(
		\alpha^\nu +
		\norm{ X }_{L_{\psi_{2}}(\prob; \cX)}^{2} \sqrt{\frac{\log(1/\delta)}{n}}
	\right).
	\end{equation*}
\end{proposition}

\begin{proof}
	In view of \Cref{thm:simul_conc}, we may restrict attention
	to an event of $\prob^{\otimes n}$-probability
	at least $1 - \delta$ on which \eqref{eq:simul_conc_covariance} and its pointwise consequences
	\eqref{eq:simul_conc_hatCXX} and \eqref{eq:simul_conc_T_CXXs} all hold.
	In what follows, all statements are to be understood pointwise
	for outcomes in this event.

	Applying the assumption $\theta_\star \in \Omega(\nu, R)$
	and \eqref{eq:simul_conc_T_CXXs}, we have
	\begin{equation}
	\label{eq:approx-eq-1}
	\bignorm{
		\theta_\star r_\alpha \bigl( \hatCXX \bigr)
		\CXX^{s} }_{S_2(\cX, \cY)
	}
	\leq 2^s  R \cdot
		\bignorm{ \CXX^{\nu}
		r_\alpha \bigl( \hatCXX \bigr) \bigl( \hatCXX + \alpha \idop_{\cX} \bigr)^s }_{L(\cX)}
	\end{equation}
	We now bound the right-hand side of \eqref{eq:approx-eq-1}
	by treating the two cases $\nu \in (0,1)$ and $\nu \geq 1$ separately.

	\noindent\textbf{\textsf{Case I: $\nu \in (0,1)$.}}
	In this case we
	apply \eqref{eq:simul_conc_T_CXXs} and
	\Cref{lem:BM2018_511} and obtain
	\begin{align*}
	\bignorm{  \CXX^{\nu} r_\alpha \bigl( \hatCXX \bigr)
		\bigl( \hatCXX + \alpha \idop_{\cX} \bigr)^s }_{L(\cX)}
	&\leq 2^\nu
	\bignorm{ r_\alpha \bigl( \hatCXX \bigr)
		\bigl( \hatCXX + \alpha \idop_{\cX} \bigr)^{\nu +s}}_{L(\cX)} \\
	 &\leq 4 \bar{\gamma} \alpha^{\nu + s} \;,
	\end{align*}

	\noindent\textbf{\textsf{Case II: $\nu \geq 1$.}}
	In this case we follow \citet[Eq.~(5.10)]{BlanchardMuecke2018}:
	\begin{align*}
		\bignorm{ \CXX^{\nu} r_\alpha \bigl( \hatCXX \bigr)
		\bigl( \hatCXX + \alpha \idop_{\cX} \bigr)^s }_{L(\cX)}
		&\leq
		\bignorm{ \hatCXX^{\nu}
		r_\alpha \bigl( \hatCXX \bigr) \bigl( \hatCXX + \alpha \idop_{\cX} \bigr)^s }_{L(\cX)} \\
		& \phantom{=} \quad + \bignorm{ \bigl( \CXX^{\nu} - \hatCXX^{\nu} \bigr)
		r_\alpha \bigl( \hatCXX \bigr) \bigl( \hatCXX + \alpha \idop_{\cX} \bigr)^s }_{L(\cX)}
	\end{align*}
	and the terms on the right-hand side can be bounded individually.
	For the first term we use the qualification of $g_\alpha$
	and apply \Cref{lem:BM2018_511}:
	\begin{align*}
	\bignorm{ \hatCXX^{\nu}
		r_\alpha \bigl( \hatCXX \bigr)\bigl( \hatCXX + \alpha \idop_{\cX} \bigr)^s }_{L(\cX)}
	&\leq 2\bar{\gamma} \alpha^{\nu +s} .
	\end{align*}
	For the second term we obtain
	\begin{equation}
		\label{eq:approx-eq-2}
		\bignorm{( \CXX^{\nu}
		- \hatCXX^{\nu} )
		r_\alpha \bigl( \hatCXX \bigr) \bigl( \hatCXX + \alpha \idop_{\cX} \bigr)^s }_{L(\cX)}
		\leq 2 \bar{\gamma} \alpha^s \;
		\bignorm{ \CXX^{\nu} - \hatCXX^{\nu} }_{L(\cX)}
	\end{equation}
	again by the qualification of $g_\alpha$ and
	\Cref{lem:BM2018_511}.
	We now address bounding
	$\norm{ \CXX^{\nu} - \hatCXX^{\nu} }_{L(\cX)}$.
	The combination of \eqref{eq:simul_conc_hatCXX} and \eqref{eq:alpha_consequence}
	implies
	\begin{equation*}
		\bignorm{ \hatCXX }_{L(\cX)} \leq
		\norm{ \CXX }_{S_2(\cX)} + \alpha.
	\end{equation*}
	Together with the assumption $\alpha \leq 1$, we see that
	$\norm{ \hatCXX }_{L(\cX)}$ is bounded by an absolute constant
	depending on $\norm{ \CXX }_{L(\cX)}$,
	allowing us to apply \Cref{prop:perturbation_sum},
	yielding the existence of a constant $\kappa'_{\nu,\CXX} < \infty$ such that
	$\norm{ \CXX^{\nu} - \hatCXX^{\nu} }_{L(\cX)} \leq \kappa'_{\nu, \CXX}
	\norm{ \CXX  - \hatCXX }_{L(\cX)}$.
	We may finally proceed bounding \eqref{eq:approx-eq-2} as
	\begin{align*}
		2 \bar{\gamma} \alpha^s \;
		\bignorm{ \CXX^{\nu} - \hatCXX^{\nu} }_{L(\cX)}
		&\leq
		2\kappa'_{\nu,\CXX} \bar{\gamma} \alpha^s \;
		\bignorm{ \CXX - \hatCXX }_{L(\cX)}\\
		&\leq
		48 \sqrt{2} e \kappa'_{\nu,\CXX} \bar{\gamma} \alpha^s \;
		\norm{ X }_{L_{\psi_{2}}(\prob; \cX)}^{2} \sqrt{\frac{\log(1/\delta)}{n}},
	\end{align*}
	where the last step follows from \eqref{eq:simul_conc_covariance}.
	The final assertion follows from collecting all estimates.
\end{proof}

\paragraph{Bounding the variance term.}
We now bound the variance term $( \theta_{\star} \hatCXX - \hatCYX ) g_{\alpha} (\hatCXX)$ in the $\CXX^{s}$-weighted Hilbert--Schmidt norm, again using properties of the regularisation scheme and concentration of measure.

\begin{proposition}[High-probability bound on variance term in \eqref{eq:error_decomp_smart}]
    \label{prop:variance}
    Under the assumptions of \Cref{thm:simul_conc},
	let $\kappa_{D,B} \defeq (D + 1) B \geq 0$ with $D$ and $B$
    as in \Cref{def:spectral_regularisation_strategy}, and let $B_{\psi_{2}}$ be as in \eqref{eq:Bpsi2}.
    Then, with $\prob^{\otimes n}$-probability at least $1 - 2 \delta$
    and for $n$ satisfying \eqref{eq:alpha_conc} as well as $s \in [0,\tfrac{1}{2}]$,
	\[
		\bignorm{ ( \theta_{\star} \hatCXX - \hatCYX ) g_{\alpha} (\hatCXX) \CXX^{s} }_{S_{2}(\cX, \cY)}
		\leq
		64 \kappa_{D,B} e B_{\psi_2} \alpha^{s - 1} \sqrt{\frac{\log(1/\delta)}{n}} .
	\]
\end{proposition}

\begin{proof}
	Define (random) linear operators $T_{1} \colon \cX \to \cY$ and $T_{2}, T_{3} \colon \cX \to \cX$ by
	\begin{align*}
		T_{1} & \defeq ( \theta_{\star} \hatCXX - \hatCYX ) ( \CXX + \alpha \idop_{\cX} )^{-1/2} , \\
		T_{2} & \defeq ( \CXX + \alpha \idop_{\cX} )^{1/2} ( \hatCXX + \alpha \idop_{\cX} )^{-1/2} , \\
		T_{3} & \defeq ( \hatCXX + \alpha \idop_{\cX} )^{1/2} g_{\alpha} (\hatCXX) ( \hatCXX + \alpha \idop_{\cX} )^{s}
		= g_{\alpha} (\hatCXX) ( \hatCXX + \alpha \idop_{\cX} )^{s + 1/2} .
	\end{align*}
	In view of \Cref{thm:simul_conc}, we may restrict attention to an event of $\prob^{\otimes n}$-probability at least $1 - 2 \delta$ on which both \eqref{eq:simul_conc_covariance} (and hence both \eqref{eq:simul_conc_empirical_idop} and \eqref{eq:simul_conc_T_CXXs}) and \eqref{eq:simul_conc_odd_residual} hold.
	Then, pointwise	for outcomes in this event,
	\begin{align*}
		& \bignorm{ ( \theta_{\star} \hatCXX - \hatCYX ) g_{\alpha} (\hatCXX) \CXX^{s} }_{S_{2}(\cX, \cY)} \\
		& \quad \leq 2^{s} \bignorm{ ( \theta_{\star} \hatCXX - \hatCYX ) g_{\alpha} (\hatCXX) ( \hatCXX + \alpha \idop_{\cX})^{s} }_{S_{2}(\cX, \cY)} && \text{(by \eqref{eq:simul_conc_T_CXXs})} \\
		& \quad = 2^{s} \norm{ T_{1} T_{2} T_{3} }_{S_{2}(\cX, \cY)} \\
		& \quad \leq 2^{s}
		\norm{ T_{1} }_{S_{2}(\cX, \cY)}
		\norm{ T_{2} }_{L(\cX)}
		\norm{ T_{3} }_{L(\cX)} \\
		& \quad \leq 2^{s} \norm{ T_{1} }_{S_{2}(\cX, \cY)} \cdot 2 \cdot \kappa_{D,B} \alpha^{s - 1/2} && \text{(by \eqref{eq:simul_conc_empirical_idop} and \Cref{lem:BM514})} \\
		& \quad \leq 2 \cdot 2^{s} \kappa_{D,B} \alpha^{s - 1/2} \frac{16 \sqrt{2} e B_{\psi_2}}{\sqrt{\alpha}} \cdot \sqrt{\frac{\log(1/\delta)}{n}} && \text{(by \eqref{eq:simul_conc_odd_residual})} \\
		& \quad \leq 64 \kappa_{D,B} e B_{\psi_2} \alpha^{s - 1} \sqrt{\frac{\log(1/\delta)}{n}} .
	\end{align*}
	This completes the proof.
\end{proof}

\paragraph{Overall error bounds and rates.}
Combining the estimates for the approximation and variance in \eqref{eq:error_decomp_smart} yields the following:

\begin{corollary}[Convergence rates in probability]
    \label{cor:upper-rates}
    Suppose the regularisation strategy $g_\alpha$ has
    qualification $q \geq \nu + s$.
    Suppose that $Y \in L_{\psi_2}(\prob; \cY)$, $X \in L_{\psi_2}(\prob; \cX)$,
    $\theta_\star \in \Omega(\nu , R)$, and $0 < \alpha < 1$.
    Let $\delta \in (0, \frac{1}{e}]$ and $s \in [0, \tfrac{1}{2}]$.
    If the regularisation parameter $\alpha = \alpha_{n}$ is chosen to depend on the number $n$ of data points via
    \begin{equation*}
        \alpha_n \defeq
        \left(
        \frac{1}{\sqrt{n}}
        \right)^{\frac{1}{\nu+1}} ,
    \end{equation*}
    then, for
    \[
    	n \geq n_{0} \defeq \max \Bigl\{ \norm{ X }_{L_{\psi_2}(\prob; \cX)}^{4} , \bigl( 1152 e^{2} \norm{ X }_{L_{\psi_2}(\prob; \cX)}^{4} \log (1 / \delta) \bigr)^{\frac{1}{\nu}} \Bigr\}^{1 + \nu} ,
    \]
    with $\prob^{\otimes n}$-probability at least $1 - 2 \delta$,
    \begin{equation*}
	    \bignorm{
        \bigl( \theta_{\star} -
        \widehat{\theta}_{\alpha_n} \bigr) \CXX^{s}
        }_{S_2(\cX, \cY)}
        \leq
        3\bar{\kappa}
        \sqrt{ \log(1 /\delta) }
        \left(
        \frac{1}{\sqrt{n}}
        \right)^{ \frac{s + \nu}{ 1 + \nu}} ,
    \end{equation*}
    where
    $
    \bar{\kappa}
    \defeq \max \{
    \kappa_{\nu, \CXX} \bar{\gamma} R, \, 64 \kappa_{D, B} e B_{\psi_2}
     \} < \infty
    $
    is obtained from the constants appearing in
    \Cref{prop:approx-error,prop:variance}.
\end{corollary}

\begin{proof}
	For the chosen regularisation parameter schedule $\alpha_n = n^{- 1 / ( 2 + 2 \nu )}$, condition \eqref{eq:alpha_conc} reduces to
	\[
		n \geq 1152 e^{2} \norm{ X }_{L_{\psi_2}(\prob; \cX)}^{4} \log (1 / \delta) n^{\frac{1}{1 + \nu}} ,
	\]
	which is satisfied by those $n$ exceeding the stated value of $n_0$.
	Thus, we may appeal to \Cref{prop:approx-error,prop:variance}.

    Since both these results are derived from
    \Cref{thm:simul_conc}, their assertions hold simultaneously
    (i.e.\ for the same subset of the sample space) with
    $\prob^{\otimes n}$-probability at least $1 - 2 \delta$.
    Applying \Cref{prop:approx-error,prop:variance}
    to the error decomposition \eqref{eq:error_decomp_smart}
    yields
    \begin{equation}
        \label{eq:rates_proof_1}
        \bignorm{
        \bigl( \theta_{\star} -
        \widehat{\theta}_{\alpha_n} \bigr) \CXX^{s}
        }_{S_2(\cX, \cY)}
        \leq
        \bar \kappa \alpha_n^s
        \left(
            \alpha_n^\nu
            +
            \norm{ X }_{L_{\psi_{2}}(\prob; \cX)}^{2} \sqrt{ \frac{ \log(1 / \delta) }{n} }
            +
            \sqrt{ \frac{ \log(1 / \delta) }{\alpha_n^2 n} }
        \right)
    \end{equation}
    for all $n \geq n_0$ with $\bar{\kappa} = \max \{
    \kappa_{\nu, \CXX} \bar{\gamma} R, \, 64 \kappa_{D, B} e B_{\psi_2}
     \} $.
    A short calculation shows that
    \begin{equation*}
		\alpha_n^{s+\nu}
		= \frac{\alpha_n^s}{\sqrt{\alpha_n^2 n}}
		= \left( \frac{1}{\sqrt n} \right)^{\frac{s+\nu}{\nu +1}}.
    \end{equation*}
    As for the middle term on the right-hand side
    of \eqref{eq:rates_proof_1},
    \begin{equation*}
		\frac{\alpha_{n}^{s}}{\sqrt{n}} = \left( \frac{1}{\sqrt n} \right)^{\frac{s+\nu}{\nu +1} + \frac{1}{1 + \nu}} .
    \end{equation*}
    Thus,
    \[
    	\bignorm{
        \bigl( \theta_{\star} -
        \widehat{\theta}_{\alpha_n} \bigr) \CXX^{s}
        }_{S_2(\cX, \cY)}
        \leq \bar{\kappa} \left( \frac{1}{\sqrt n} \right)^{\frac{s+\nu}{\nu +1}} \left( 1 + \norm{ X }_{L_{\psi_{2}}(\prob; \cX)}^{2} \sqrt{ \frac{ \log(1 / \delta) }{n} } + \sqrt{\log (1/\delta)} \right) ,
    \]
    and the claim now follows from the hypotheses that $\delta < \tfrac{1}{e}$ and $n \geq n_{0}$.
\end{proof}

Two remarks concerning the above rates are in order.

\begin{remark}[Fast rates]
    Merely imposing source conditions is insufficient to characterise regression problems allowing for
    rates faster than $1/\sqrt{n}$ with high probability.
    We conjecture that, analogously to results in kernel regression,
    the use of a Bernstein-type inequality
    \citep[e.g.][Proposition~7(iii)]{MaurerPontil2021}
    combined with a sufficiently fast decay of eigenvalues
    of $\CXX$ will lead to the desired results.
    However, this requires linking the eigenvalue decay ---
    in terms of some notion of an \emph{effective dimension} (e.g.\ \citealp{CaponnettoDeVito2007}; \citealp{BlanchardMuecke2018}) ---
    to the variance proxy in the Bernstein bound
    in order to find the optimal schedule for the regularisation parameter.
    For kernel problems, this is commonly done based on the classical
    Bernstein bound by \citet{PinelisSakhanenko1985}
    under assumption of the vector-valued Bernstein condition
    \eqref{eq:bernstein_condition}.
    However, we deliberately choose to work with sub-Gaussian and
    sub-exponential norms instead, as these concepts
    allow a convenient analysis of unbounded random variables
    and their tensor products.
    This is a major difference compared to kernel regression, where
    the kernel and therefore the embedded random variables in the
    reproducing kernel Hilbert space are assumed to be bounded.
\end{remark}

\begin{remark}[Optimal rates and comparison to kernel setting]
    The rates in \Cref{cor:upper-rates} match those
    of kernel regression with scalar and finite-dimensional
    response variables under a H\"older source condition and with no
    additional assumptions on the eigenvalue decay of $\CXX$;
    see e.g.\ \citet{CaponnettoDeVito2007},
    \citet{BlanchardMuecke2018}, and \citet{LinEtAl2020}.
    Minimax optimality of these rates is only derived by
    \citet{CaponnettoDeVito2007} and \citet{BlanchardMuecke2018}
    under the additional assumption that the eigenvalues of $\CXX$ decay rapidly enough,
    which is an implicit assumption on the marginal distribution of $X$.
    To establish minimax optimality in our setting, we would have to repeat
    the standard arguments, e.g.\ apply a general reduction scheme
    in conjunction with Fano's method \citep{Tsybakov2009}.
    However, we show in \Cref{sec:kernel_regression} that the Hilbert--Schmidt
    regression problem \eqref{eq:regression_problem_hs}
    contains scalar response kernel regression
    as well as some settings of kernel regression with vector-valued response
    as special cases.
\end{remark}

\section{Related work, applications and examples}
\label{sec:applications}

In this section, we discuss some specific applications
of the infinite-dimensional regression setting.
We compare our approach to existing results in the literature
and comment on the insights obtained from our general framework
as well as potential directions for future research.

\subsection{Functional data analysis}

When the Hilbert spaces $\cX$ and $\cY$ are chosen to
be the space of square integrable functions over
some interval with respect to the Lebesgue measure $L^2([a,b])$,
the infinite-dimensional linear regression problem \eqref{eq:regression_problem}
becomes \emph{functional linear regression with functional response},
which is ubiquitous in functional data analysis.
However, it is important to note that virtually
all results for this problem derived in context of
functional data analysis assume the well-specified case
\eqref{eq:linear_model}, i.e.\ $Y = \theta_{\star} X + \epsilon$,
which is usually called the \emph{functional linear model}.
We refer the reader to the monographs by \citet{RamsaySilverman2005} 
and \citet{HorvathKokoszka2012} for comprehensive introductions to the
functional linear model. We emphasise that
the operator $\theta_{\star}$ is often assumed to
be Hilbert--Schmidt whenever it is expressed in terms of
a convolution with respect to some suitable integral kernel.
The functional linear model
with functional response is particularly important for
for \emph{autoregressive models} and the theory of
\emph{linear processes} in functional time series
(\citealp{MasPumo2010}, \citealp{HorvathKokoszka2012}).
We will address these topics separately in a more general setting
of Hilbertian time series in the next section.

In the existing literature on functional linear regression,
estimators are derived by assuming the well-specified case
based on the equation \smash{$\CYX = \theta_{\star} \CXX$}
(this is easily obtained from the functional linear model
as shown in \cref{ex:linear_model}). However,
instead of solving this identity for $\theta_{\star}$ in terms
of a more general inversion of a precomposition operation
(or equivalently by solving an operator factorisation),
estimators of $\theta_{\star}$ are obtained by
heuristically applying some
specific regularisation strategy to \smash{$\CXX$} directly
and performing an operator composition.
This approach has led to various estimators
based on the functional
linear model with functional response
(we refer the reader to
\citet{CrambesMas2013}, \citet{HoermannKidzinski2015},
\citet{BenatiaEtAl2017}, \citet{ImaizumiKato2018} and the references therein).
While most of these estimators do allow for a statistical analysis,
they bypass the application of standard consistency arguments
for classical inverse problems almost entirely
(see e.g.\ \citealp{EnglHankeNeubauer1996}).
This renders the analysis of the regularised estimators fairly
complicated, as the underlying spectral theory requires a manual
investigation of the composition of \smash{$\CYX$} and
the regularised version of \smash{$\CXX$}.
In fact, known approaches from inverse problem theory
are essentially implicitly reconstructed in these proofs
in the much more challenging context of operator composition.
The severity of this problem arising from
the singularity of \smash{$\CXX$}
and the involved composition 
is sometimes even mentioned explicitly by the aforementioned authors;
see also \citet[Section~3.2]{MasPumo2010} and
\citet[Section~8.2]{Bosq2000}.

Although our results in \Cref{sec:inverse_problem}
(and in particular \cref{cor:functional_calculus})
show that this approach is mathematically equivalent to
the inverse problem framework we propose here, we
argue below that our perspective has several advantages:
\begin{enumerate}[label=(\alph*)]
    \item It combines the existing estimators for functional linear
        regression into a unified framework which allows for
        a comparison and simplified investigation in the language
        of inverse problems.
        We emphasise that for kernel-based regression, an analogous approach
        has lead to a vast variety of important results
        and has now become a standard setup for the investigation of
        supervised learning (\citealp{CaponnettoDeVito2007};
        we refer the reader to \citealp{BlanchardMuecke2018}
        and the references therein for a more
        recent analysis based on this framework and an overview of existing results).
    \item It generalises the functional linear regression problem to a least-squares
        setting in which is not necessarily required to assume the functional linear
        model (i.e.\ the well-specified case)
        in order to construct estimators.
    \item It reduces the construction and the investigation
      of new solution algorithms for functional linear regression
      to the choice of new regularisation
      strategies in our framework.
\end{enumerate}

\subsection{Linear autoregression of Hilbertian time series}

We consider a time series $(Y_t)_{t \in \Naturals}$
taking values in the Hilbert space $\cY$.
For the sake of clarity in our presentation, we will assume that
$(Y_t)_{t \in \Naturals}$ is strictly stationary
and centered.
Our framework can be generalised
to the non-stationary case straightforwardly.
For some fixed time horizon $r \in \Naturals$,
we may seek to solve the \emph{linear autoregression problem}
\begin{equation}
    \label{eq:linear_autoregression_problem}
    \min_{\theta_1, \dots, \theta_r \in L(\cY)}
    \E \left[ \Norm{  Y_t - \sum_{i=1}^r \theta_{i} Y_{t-i} }^2_{\cY} \right].
\end{equation}
In the \emph{well-specified case}, there exist
operators $\theta_{\star 1}, \dots, \theta_{\star r} \in L(\cY)$ such that
\begin{equation*}
    \label{eq:autoregressive_model}
    \tag{ARH-$r$}
    Y_t = \sum_{i=1}^r \theta_{\star i} Y_{t-i} + \epsilon_t
    \text{ for all }t \in \Naturals
\end{equation*}
with a family of pairwise independent
$\cY$-valued noise variables $(\epsilon_t)_{t \in \Naturals}$
satisfying $\E[ \epsilon_t \mid Y_{t-i} ] = 0$ for all $1 \leq i \leq r$,
i.e.\ the process $(Y_t)_{t \in \Naturals}$ satisfies
the \emph{linear Hilbertian autoregressive model of order $r$}
(ARH-$r$, see \citealp{Bosq2000}).
In practice, depending on the scenario of application
(e.g.\ in the context of functional time series),
it may be reasonable to directly restrict the
linear autoregression problem \eqref{eq:linear_autoregression_problem}
to a minimisation over $p$-Schatten class operators
instead of bounded operators.

We can directly reformulate the
linear autoregression problem \eqref{eq:linear_autoregression_problem}
as a special case of the general
linear regression problem \eqref{eq:regression_problem}.
We define the Hilbert space $\cX \defeq \cY^r = \bigoplus_{i=1}^r \cY$
via the external direct sum of normed spaces.
We note that every operator
$\theta \in L(\cX, \cY)$
corresponds to a unique sequence
of operators $\theta_1, \dots, \theta_r \in L(\cY)$
such that
\begin{equation}
    \label{eq:isomorphisms_operators1}
    \theta \mathbf{x} = \sum_{i=1}^r \theta_i y_i
    \text{ for every }
    \mathbf{x} = (y_1, \dots, y_r) \in \cX
\end{equation}
and vice versa. Moreover, the same statement holds when
the classes of bounded operators are replaced with
$p$-Schatten operators.
In fact, is is straightforward to check that
the above correspondence \eqref{eq:isomorphisms_operators1}
defines the isomorphisms
\begin{equation}
    \label{eq:isomorphisms_operators2}
    L(\cX, \cY) \cong \bigoplus_{i=1}^r L(\cY)
    \text{ and }
    S_p(\cX, \cY) \cong \bigoplus_{i=1}^r S_p(\cY)
    \text{ for all } p \geq 1.
\end{equation}
As a consequence of \eqref{eq:isomorphisms_operators1}
and \eqref{eq:isomorphisms_operators2},
the linear autoregression problem \eqref{eq:linear_autoregression_problem}
is equivalent to the linear regression problem \eqref{eq:regression_problem} given by
\begin{equation}
    \label{eq:linear_autoregression_problem_simplified}
    \min_{ \theta \in L(\cX, \cY) }
    \E [ \norm{ Y_t - \theta \mathbf{X}_t }^2_{\cY} ]
    \text{ with }
    \mathbf{X}_t \defeq (Y_{t - 1}, \dots, Y_{t-r}) \in \cX
    \text{ and }
    \theta \mathbf{X}_t \defeq \sum_{i=1}^r \theta_i Y_{t-i}.
\end{equation}
for the family of operators $\theta_1, \dots, \theta_r \in L(\cY)$.
This equivalence still holds if the classes of
bounded operators are replaced with
$p$-Schatten operators in both descriptions of the
problem. The latter problem can be solved in terms
of the inverse problem approach proposed in the preceding sections.
In particular, we seek to solve the inverse problem
\begin{equation*}
 A_{ C_{ \mathbf{X}_{t} \mathbf{X}_{t} } } [\theta ]
 =
 C_{ \mathbf{X}_{t+1} \mathbf{X}_{t} },
\end{equation*}
where we may express the covariance operators in block format on
$\cX = \cY^r$ as
\begin{equation*}
C_{ \mathbf{X}_{t} \mathbf{X}_{t} }
=
\begin{bmatrix}
    C_{Y_{t-1} Y_{t-1}} & \cdots & C_{Y_{t-1} Y_{t-r}} \\
    \vdots & \ddots & \vdots \\
    C_{Y_{t-r} Y_{t-1}} & \cdots & C_{Y_{t-r} Y_{t-r}}
\end{bmatrix}
\text{ and }
C_{ \mathbf{X}_{t+1} \mathbf{X}_{t} }
=
\begin{bmatrix}
    C_{Y_{t} Y_{t-1}} & \cdots & C_{Y_{t} Y_{t-r}} \\
    \vdots & \ddots & \vdots \\
    C_{Y_{t-r+1} Y_{t-1}} & \cdots & C_{Y_{t-r+1} Y_{t-r}}
\end{bmatrix}.
\end{equation*}
A detailed discussion of the potential regularisation schemes arising
are out the scope of this paper,
as are the approaches for proving consistency under the assumption
of ergodicity and relevant concepts for measuring probabilistic
temporal dependence.

Applying our inverse problem framework to the reformulated autoregression problem \eqref{eq:linear_autoregression_problem_simplified} has some advantages for the theory of linear prediction for Hilbertian time series.
The inverse problem perspective may lead to directions of research with the potential of extending the available literature \citep{Bosq2000, MasPumo2010}:
\begin{enumerate}[label=(\alph*)]
	\item It allows us to express solutions of the forecasting problem in terms of the generic spectral regularisation approach, giving a theoretical foundation for heuristically designed estimators relying on the (approximate) inversion of $C_{ \mathbf{X}_{t} \mathbf{X}_{t} }$.

	\item It allows us to consider the more general least squares forecasting problem instead of assuming the well-specified ARH-$1$ model as commonly done in the literature.

	\item It allows the direct regularisation of the $r$\textsuperscript{th}-order forecasting problem, potentially leading to a broad class of solutions schemes designed for the inverse problem involving the block operators $C_{ \mathbf{X}_{t} \mathbf{X}_{t} }$ and $C_{ \mathbf{X}_{t+1} \mathbf{X}_{t} }$.
\end{enumerate}

\subsection{Vector-valued kernel regression}
\label{sec:kernel_regression}

We now show that our linear regression setting includes a
case of \emph{nonlinear} regression with
\emph{vector-valued reproducing kernels} \citep{CarmeliEtAl2006, CarmeliEtAl2010}
for infinite-dimensional response variables.
In fact, it is known that standard proof techniques
for obtaining convergence rates for this problem
\citep[e.g.][]{CaponnettoDeVito2007} may fail
for a commonly used type of kernel
in the infinite-dimensional case, as the compactness of
the involved inverse problem is explicitly assumed.
This has been remarked independently by several authors over the last years
(e.g.\ \citealp{GruenewaelderEtAl2012}; \citealp{KadriEtAl2016};
\citealp{ParkMuandet2020}; \citealp{Mollenhauer2022PhD}).
An important special case of the setting described here
is the \emph{conditional mean embedding}, for which
optimal rates (matching those of the scalar response kernel regression)
have recently been derived for the special case of
Tikhonov--Phillips regularisation; we refer the reader to
\citet{KlebanovEtAl2020}, \citet{ParkMuandet2020} and \citet{LiEtAl2022} and the references therein.
For the general infinite-dimensional response kernel regression,
our approach shows that
\emph{%
    the resulting non-compact inverse problem behaves as in the
real-valued learning case from a spectral regularisation perspective.
}

We elaborate on this statement in what follows.
We introduce the concept of a \emph{vector-valued reproducing kernel
Hilbert space} (vv-RKHS) in terms of an operator-based description
as formulated by \citet{Mollenhauer2022PhD}.
See \citet{CarmeliEtAl2006} and \citet{CarmeliEtAl2010} for a comprehensive and more general construction of vv-RKHSs from corresponding \emph{operator-valued reproducing kernels}, their topological properties and the theory of corresponding inclusion operators.

Let $\cE$  be a second-countable locally compact Hausdorff space
equipped with its Borel $\sigma$-algebra $\Borel_{\cE}$,
and let $\cX$ be an RKHS consisting of $\Reals$-valued
functions on $\cE$ with reproducing kernel
$k \colon \cE^{2} \to \Reals$ and canonical feature map $\varphi \colon \cE \to \cX$.
Assume further that $(\cE, \Borel_{\cE})$ is
equipped with a probability measure $\mu$, with a
\emph{compact embedding operator} $i \colon \cX \hookrightarrow L^{2} (\mu)$
\citep[e.g.][Section~4.3]{ChristmannSteinwart2008}.

Let $\cY$ be another separable real Hilbert space.
Consider $\mathcal{G} \defeq \set{ A \varphi(\quark) }{ A \in S_{2} (\cX, \cY) }$;
this is a vv-RKHS of $\cY$-valued functions with operator-valued reproducing kernel
\begin{align*}
	K \colon \cE^{2} & \to L(\cY) \\
	(x, x') & \mapsto k(x, x') \idop_{\cY}
\end{align*}
and we have a bounded linear embedding operator
\[
	I \defeq i \otimes \idop_{\cY} \colon \mathcal{G} \cong \cX \otimes \cY \hookrightarrow L^{2} (\mu) \otimes \cY \cong L^{2} (\mu; \cY) .
\]
As the embedding
$i \colon \cX \hookrightarrow L^{2} (\mu)$
is compact, the embedding
$I \defeq i \otimes \idop_{\cY}$ is compact precisely when $\dim \cY < \infty$.
The operator-valued kernel $K$ defined above plays a fundamental role for learning problems
with infinite-dimensional response variables.

We now consider an $\cE$-valued random
variable $\xi$ with law $\law(\xi) \qefed \mu$ on $(\cE, \Borel_{\cE})$
and a $\cY$- valued random variable $Y$, both defined on a common probability space.
We can now immediately see that our \emph{nonlinear kernel regression problem}
\begin{equation}
    \label{eq:rkhs-regression-nonlinear}
	\min_{F \in \mathcal{G}} \E [ \norm{ Y - F(\xi) }_{\cY}^{2} ]
\end{equation}
is equivalent to our
\emph{linear regression problem} \eqref{eq:regression_problem}
with $X \defeq \varphi(\xi)$,
as we have
\begin{equation}
    \label{eq:rkhs-regression-linear}
	\min_{\theta \in S_{2}(\cX, \cY)} \E [ \norm{ Y - \theta \varphi (\xi) }_{\cY}^{2} ] ,
\end{equation}
which we would then reformulate as the inverse problem \eqref{eq:inverse_problem}
in terms of cross-covariance and covariance operators.

\begin{remark}[Kernel trick]
    Replacing the nonlinear regression problem
    \eqref{eq:rkhs-regression-nonlinear}
    with the linearised problem
    \eqref{eq:rkhs-regression-linear}
    can be interpreted as the well-known
    \emph{kernel trick} for vector-valued regression.
\end{remark}

The well-specified kernel regression scenario given by the assumption
\[
	F_{\star} (\quark) \defeq \E [ Y | \xi = \quark ] \in I(\mathcal{G}) \subseteq L^{2} (\law(\xi); \cY)
\]
is equivalent to the linear model in the RKHS
(or equivalently the linear conditional
expectation property by \citealt{KlebanovEtAl2021})
where there exists $\theta_{\star} \in S_{2} (\cX, \cY)$ such that
\[
	F_{\star} (\quark) = \theta_{\star} \varphi (\quark) \quad \law(\xi)\text{-a.s.}
\]
In particular, this example motivates minimisation of the above objective function over the more restrictive space $S_2(\cX, \cY)$ instead of $L(\cX, \cY)$ in a theoretical discussion of convergence based on our framework.

\begin{remark}[Spectrum of non-compact inverse problem]
    In kernel learning theory
    (see e.g. \citealp{CaponnettoDeVito2007}),
    the regression problem \eqref{eq:rkhs-regression-nonlinear}
    is solved via the fundamentally important \emph{normal equation}
    \begin{equation*}
        T F = I^\ast F_\star, \quad F \in \mathcal{G}
    \end{equation*}
    based on the \emph{(generalised) kernel covariance operator}
    $T \defeq I^\ast I \colon \mathcal{G} \to \mathcal{G}$, which is
    sometimes also called \emph{frame operator}.
    Under the isomorphism $\mathcal{G} \cong \cX \otimes \cY$ we have
    \begin{equation*}
        T = I^{\ast} I  \cong ( i \otimes \idop_\cY )^{\ast} (i \otimes \idop_\cY) =
        i^{\ast} i \otimes \idop_{\cY} = \CXX \otimes \idop_{\cY},
    \end{equation*}
    where we apply \citet[Proposition~12.4.1]{Aubin2000} together
    with the well-known fact that $\CXX = i^{\ast}i$
    \citep[Theorem~4.26]{ChristmannSteinwart2008}.
    \citet[Chapter~12.4]{Aubin2000} also shows that
    the tensor product operator
    $\CXX \otimes \idop_{\cY} \colon \cX \otimes \cY \to \cX \otimes \cY$
    is again isomorphically equivalent to the precomposition
    $A_{\CXX} \colon S_2(\cX, \cY) \to S_2(\cX, \cY)$.
    Therefore, $T \cong A_{\CXX}$ (where we abuse notation and write $\cong$ for the isomorphism of operators arising from the isomorphisms of the domain and codomain spaces).
    Hence,
    \Cref{thm:spectral_properties_of_A_C}\ref{thm:spectral_properties_of_A_C_decomposition}
    shows that
    \begin{equation*}
        T \cong \sum_{ \lambda \in \pSpectrum(\CXX)}
        \lambda P_{ \cY \otimes \espace_\lambda(\CXX)}.
    \end{equation*}
    In particular, $T$ has the same spectrum as $\CXX$,
    which is the forward operator of the inverse problem
    associated with scalar response kernel regression.
    We revisit this fact in a more general setting in the next section.
\end{remark}

\begin{remark}[Convergence rates]
    The rates derived in
    \Cref{cor:upper-rates} appear to be the first convergence results
    obtained for general regression with operator-valued kernels
    which do not require the assumption of compactness of the underlying inverse problem,
    hence allowing for infinite-dimensional response variables.
    For the special case of the conditional mean embedding, they appear to
    be the first rates covering generic regularisation strategies.
\end{remark}

\section{General interpretation and implications of our work}
\label{sec:closing}

Let us briefly consider the well-known scalar response
regression problem with $\cY = \Reals$.
This setting allows the discussion of
generic regularisation schemes for
standard finite-dimensional linear least squares regression
(for the case $\cX = \Reals^d$) and various
nonlinear kernel regression settings
(in this case, $\cX$ is a reproducing kernel Hilbert space, e.g.\
\citealp{BauerEtAl2007}; \citealp{BlanchardMuecke2018}), including
kernel distribution regression \citep{SzaboEtAL2016} and
kernel regression over general Hilbert spaces \citep{LinEtAl2020}.
To emphasise that we are dealing with a scalar response learning
problem, we will use the notation $\mathbf{y} \defeq Y$ for
the real-valued response variable here.

We identify every operator $\theta \in L(\cX, \Reals)$
with some unique vector $\vartheta \in \cX$
via Riesz representation theorem as
$\theta x = \innerprod{ \vartheta }{ x }_{\cX}$
for all $x \in \cX$
(i.e.\  we consider the canonical isomorphism $L(\cX, \Reals) \cong \cX$) and rewrite
the regression problem \eqref{eq:regression_problem} in its
\emph{dual formulation} given by
\begin{equation*}
    \min_{\vartheta \in \cX}
    \E [\absval{ \mathbf{y} - \innerprod{ \vartheta }{X}_{\cX} }^2].
\end{equation*}
It is straightforward to check that the covariance operator
$\CYX = \E[ \mathbf{y} \otimes X ]  \in L(\cX, \Reals)$ satisfies
$\CYX x = \innerprod{ \E[ \mathbf{y} X ] }{ x }_{\cX} $ for all $x \in \cX$,
i.e.\ it is identified with the vector $\E[ \mathbf{y}X] \in \cX$
under the above isomorphism.
Furthermore, the dual analogue of the inverse problem \eqref{eq:inverse_problem}
clearly becomes
\begin{equation}
    \label{eq:inverse_problem_scalar}
    \CXX \vartheta = \E[ \mathbf{y} X ], \quad \vartheta \in \cX,
\end{equation}
which is well-known across all statistical disciplines
in a variety of formulations.

\emph{%
This shows that in the scalar response
setting, the forward operator $A_{\CXX}$ of the inverse problem \eqref{eq:inverse_problem}
is isomorphically equivalent to $\CXX$. }
This can also be seen directly by considering the spectral theorem from
\Cref{thm:spectral_properties_of_A_C}\ref{thm:spectral_properties_of_A_C_decomposition}
together with the fact that the eigenspaces of $A_{\CXX}$ and $\CXX$ are 
isomorphic, i.e.\ $\Reals \otimes \espace_\lambda(\CXX) \cong \espace_\lambda(\CXX)$
for all $\lambda \in \Spectrum(\CXX)$, which follows from the
definition of the Hilbert tensor product.

We gain several insights and potential directions
of future research for the general
infinite-dimensional learning setting
from considering the scalar response regression problem.
Most importantly, the discussion above shows that
the formalism derived in this article
can be interpreted as the natural generalisation
of the inverse problem associated with scalar response least squares regression
given by \eqref{eq:inverse_problem_scalar}.

\paragraph{Regularising infinite-dimensional response regression.}

\Cref{thm:spectral_properties_of_A_C} and \Cref{cor:functional_calculus}
show that the infinite-dimensional response regression problem \eqref{eq:regression_problem} with its inverse problem
\eqref{eq:inverse_problem} and the scalar response regression problem
with its inverse problem \eqref{eq:inverse_problem_scalar} are
equivalent in the sense that they share
the same spectrum and essentially require
``only'' regularising the operator $\CXX$, even though
the infinite dimensionality of $\cY$ introduces non-compactness
of \eqref{eq:inverse_problem} via infinite-dimensional eigenspaces
of the forward operator $A_{\CXX}$.
However, it is possible to
measure the estimation error in operator norm, which is independent of the dimensionality of the corresponding eigenspaces.
In particular,
$\norm{ A_{g_{\alpha}( \CXX) } }_{S_2(\cX, \cY) \to S_2(\cX, \cY) }
=
\norm{ g_{\alpha} (\CXX) }_{\cX \to \cX}$, which holds analogously
for operator norm distances between regularised precomposition
operators and the corresponding regularised covariance operators
(\Cref{lem:basic_precomposition} combined with \Cref{cor:functional_calculus}).
Consequently, the inverse problem
\eqref{eq:inverse_problem} and its convergence analysis
can be conveniently reduced to the analysis
of \eqref{eq:inverse_problem_scalar} to some extent.
In general, even the interpretation of the source conditions is
identical in both settings
according to \Cref{cor:source_condition}.
However, an infinite-dimensional response
of course requires square summability of the
source condition across all dimensions, analogously
to the existence criterion for Hilbert--Schmidt solutions in \Cref{prop:existence_hs}.

\paragraph{Existence of a solution, model misspecification, and covariance bounds.}

Although we are able to give basic existence criteria
of minimisers of the infinite-dimensional regression problem
in terms of alternative characterisations of range inclusions of bounded linear
operators based on ideas by \citet{Douglas1966},
the results formulated in \Cref{prop:existence}, \Cref{rem:existence} and
\Cref{prop:existence_hs}
seem a bit unspecific in our context of covariance operators
(we also refer to the recent discussion by \citealp{KlebanovEtAl2021}).
In fact, the importance of range inclusion properties
is already mentioned by \citet{Baker1973} in his seminal paper on covariance
operators.
However, to the best of our knowledge, no analysis of functional models for
$X$ and $Y$ satisfying this property are available in the literature.
Under the assumption that $Y = f(X) + \epsilon$ $\prob$-a.s.\
for some nonlinear transformation
$f\colon \cX \to \cY$ and noise variable $\epsilon$,
does there exist some finite positive constant $\beta_{X,Y,f,\epsilon}$ such that
\begin{equation*}
\CXY \CYX \leq \beta_{X,Y,f,\epsilon} \CXX^2,
\end{equation*}
where $\beta_{X,Y,f, \epsilon}$ may depend on reasonable assumptions about
smoothness and regularity properties of $f$ and
$X$, $Y$ and $\epsilon$ and their joint distributions?
Such a bound implies the range inclusion property
due to \Cref{prop:existence}\ref{prop:existence_douglas3}
and hence describes classes of nonlinear infinite-dimensional
statistical models which could still be reasonably empirically approximated
by performing linear regression in a misspecified setting.
For finite-dimensional $X$, this is always satisfied.

\paragraph{Inverse problem and empirical theory for bounded operators.}

We have derived our empirical framework in the context of an inverse problem in
the Hilbert space $S_2(\cX)$.
This allowed us to straightforwardly derive the form of
$\theta_\alpha$ and $\widehat{\theta}_\alpha$ in the context of
the regularised precomposition operator and its empirical analogue.
In our investigation, the precomposition operator
serves mainly as a tool to establish a connection
to the known theory of inverse problems and derive source conditions.
However, note that the regularised population solution
$\theta_\alpha = \CYX g_\alpha(\CXX)$
and its empirical analogue can still be defined
when we merely expect the true solution $\theta_\star = ( \CXX^\dagger \CXY )^*$
to be bounded, bypassing the need for the precomposition operator.
In this case, the prediction error
$
	\norm{ ( \theta_{\star} - \widehat{\theta}_\alpha) \CXX^{1/2} }_{S_{2}(\cX, \cY)}
$
is of course still guaranteed to be finite,
as $\CXX$ is trace class and hence $\CXX^{1/2}$
is Hilbert--Schmidt.
It seems reasonable that under this choice of (semi)-norm,
the bounded operator $\theta_\star$ can still be
approximated by the Hilbert--Schmidt (or even finite-rank) estimate
$\widehat{\theta}_\alpha$.
Under modified source conditions imposed on $\theta_\star$
allowing for mere boundedness, we expect that rates for the bounded regression
problem can be obtained. Interestingly, the
H\"older-type source condition
$\theta_\star = \tilde\theta \CXX^\nu$ for
some \emph{bounded} operator $\tilde\theta$ with
$\norm{ \tilde \theta }_{L(\cX, \cY)} \leq R$ already
implies that $\theta_\star$ is Hilbert--Schmidt whenever $\nu \geq 1/2$.
This indicates that the investigation of more
general source conditions is relevant for this problem.
The theory of learning bounded operators in the context of our framework
is a highly interesting field to explore in the future.

\section*{Acknowledgements}
\addcontentsline{toc}{section}{Acknowledgements}

MM is partially supported by the Deutsche Forschungsgemeinschaft (DFG) through grant EXC 2046 ``MATH+''; Project Number 390685689, Project IN-8 ``Infinite-Dimensional Supervised Least Squares Learning as a Noncompact Regularized Inverse Problem''.
This work has been partially supported by an award from the University of Warwick International Partnership Fund.
The authors wish to thank Ilja Klebanov for helpful comments.

\appendix

\section{Technical results}

We collect here supporting material for the results in the main text.
\Cref{app:operator_perturbation} contains deterministic results on the perturbation and spectral regularisation of operators.
\Cref{app:subexponentiality} sets out the relationship between sub-exponential Hilbertian random variables and the Bernstein condition.
\Cref{app:concentration_bounds} contains concentration-of-measure results for sub-exponential Hilbertian random variables, which were summarised in the main text as \Cref{thm:simul_conc}.

\subsection{Operator perturbation and spectral regularisation}
\label{app:operator_perturbation}

\begin{proposition}[Cf.\ {\citealp[Proposition~5.6]{BlanchardMuecke2018}}]
    \label{prop:perturbation_sum}
	Let $\cH$ be a Hilbert space, $T_1, T_2 \in L(\cH)$ positive semi-definite.
    Then, for all $0 < a < \infty $ and $1 \leq \nu < \infty$,
    there exists a constant $\kappa_{a, \nu} < \infty$
    such that, whenever $\max \{ \norm{ T_1 }_{L(\cX)} , \norm{ T_2 }_{L(\cX)} \} \leq a$,
    \begin{equation*}
        \norm{ T^\nu_1 - T^\nu_2 }_{L(\cX)}
        \leq \kappa_{a, \nu} \norm{ T_1 - T_2 }_{L(\cX)} .
    \end{equation*}
\end{proposition}

\begin{lemma}[Cf.\ {\citealp[Lemma~2.15]{BlanchardMuecke2018}}]
    \label{lem:BM2018lemma215}
    Let $g_{\alpha}$ be a
    spectral regularisation strategy (\Cref{def:spectral_regularisation_strategy})
    with qualification $q$.
    Then, for all $0 < \nu < q$ and $ 0 < \alpha < \infty$, we have
    \begin{equation*}
        \sup_{0 < \lambda < \infty}
            \absval{ r_\alpha(\lambda) } \lambda^\nu
        \leq \gamma_\nu \alpha^\nu,
    \end{equation*}
    with the constant
    $\gamma_\nu \defeq \gamma_0^{1- \tfrac{\nu}{q}} \gamma_q^{\tfrac{\nu}{q}}$.
\end{lemma}

\begin{proof}
    The result presented in Lemma~2.15 by \citet{BlanchardMuecke2018}
    is given for regularisation strategies $g_\alpha$ formally defined
    on $[0,1]$ for every $\alpha > 0$. However, it is obtained
    as a special case from a more general result
    given by \citet[Proposition~3]{MathePereverzev2003}, which also implies
    the assertion given here.
\end{proof}

\begin{lemma}[Cf.\ {\citealp[Eq.~(5.11)]{BlanchardMuecke2018}}]%
    \label{lem:BM2018_511}
	Let $\cH$ be a Hilbert space, $T \in L(\cH)$
    self-adjoint and positive semi-definite,
    and $g_{\alpha}$ a spectral regularisation
    strategy (\Cref{def:spectral_regularisation_strategy})
    with qualification $q \geq s + \nu $, where $s, \nu > 0$.
    Then, for $\bar{\gamma} \defeq \max\{\gamma_0, \gamma_q\}$,
    \begin{equation*}
        \norm{ ( T + \alpha \idop_\cH )^s r_\alpha( T ) }_{L(\cH)}
        \leq
        2 \bar{\gamma} \alpha^{s}
    \end{equation*}
    and
    \begin{equation*}
        \norm{ ( T + \alpha \idop_\cH )^s r_\alpha( T )
        T^\nu }_{L(\cH)}
        \leq
        2 \bar{\gamma} \alpha^{s + \nu} .
    \end{equation*}
\end{lemma}

\begin{proof}
    For the first bound, we have
    \begin{align*}
        \norm{ ( T + \alpha \idop_\cH )^s r_\alpha( T ) }_{L(\cH)}
		& =
        \sup \Set{ \absval{ \mu } }%
        { \mu \in \pSpectrum \bigl( ( T + \alpha \idop_\cH )^s r_\alpha( T ) \bigr) } \\
        &= \sup_{\lambda \in \pSpectrum(T) }
        \absval{ (\lambda + \alpha)^s r_\alpha (\lambda) } \\
        &\leq \sup_{\lambda \in \pSpectrum(T) }
        \absval{ r_\alpha (\lambda) } \lambda^s
        +
        \alpha^s \sup_{\lambda \in \pSpectrum(T) }
        \absval{ r_\alpha (\lambda) } \\
        &\leq 2 \bar{\gamma} \alpha^{s}
    \end{align*}
    by the qualification of $g_\alpha$,
    where in the last step
    \Cref{lem:BM2018lemma215} is applied to the first summand.
    We analogously obtain the second bound as
    \begin{align*}
        \norm{ ( T + \alpha \idop_\cH )^s r_\alpha( T ) T^\nu }_{L(\cH)}
		& =
        \sup \set{ \absval{ \mu } }%
        { \mu \in \pSpectrum( (T + \alpha \idop_\cH )^s r_\alpha( T ) T^\nu ) } \\
        &= \sup_{\lambda \in \pSpectrum(T) }
        \absval{ (\lambda + \alpha)^s r_\alpha (\lambda) \lambda^\nu } \\
        &\leq
        \sup_{\lambda \in \pSpectrum(T) }
            \absval{ r_\alpha(\lambda) }\lambda^{s + \nu}
        +
        \alpha^s \sup_{\lambda \in \pSpectrum(T) }
            \absval{ r_\alpha(\lambda) }\lambda^{\nu}  \\
        &\leq 2 \bar{\gamma} \alpha^{s + \nu},
    \end{align*}
    where in the last step \Cref{lem:BM2018lemma215} is applied
    to each summand.
\end{proof}

\begin{lemma}[Cf.\ {\citealp[Eq.~(5.14)]{BlanchardMuecke2018}}]%
	\label{lem:BM514}%
	Let $\cH$ be a Hilbert space, $T \in L(\cH)$ self-adjoint and positive semi-definite, and $g_{\alpha}$ a spectral regularisation strategy (\Cref{def:spectral_regularisation_strategy}).
    Then there exists a constant $\kappa_{D,B} \defeq (D + 1) B \geq 0$ such that,
    for all $\alpha > 0$ and $0 \leq s \leq \tfrac{1}{2}$,
	\[
		\norm{ g_{\alpha}(T)
        (T + \alpha \idop_{\cH})^{s + 1/2} }_{L(\cH)}
        \leq \kappa_{D,B} \alpha^{s - 1/2}.
	\]
\end{lemma}

\begin{proof}
	We calculate directly as follows:
	\begin{align*}
		& \norm{ g_{\alpha}(T) (T + \alpha \idop_{\cH})^{s + 1/2} }_{L(\cH)} \\
		& \quad = \sup \Set{ \absval{ \mu } }{ \mu \in \pSpectrum \bigl( g_{\alpha}(T) (T + \alpha \idop_{\cH})^{s + 1/2} \bigr) } \\
		& \quad = \sup_{\lambda \in \pSpectrum(T)} \absval{ g_{\alpha}(\lambda) (\lambda + \alpha)^{s + 1/2} } \\
		& \quad \leq \sup_{\lambda \in \pSpectrum(T)} \absval{ g_{\alpha}(\lambda) \lambda^{s + 1/2} } + \sup_{\lambda \in \pSpectrum(T)} \absval{ g_{\alpha}(\lambda) \alpha^{s + 1/2} } \\
		& \quad \leq \sup_{\lambda \in \pSpectrum(T)} \absval{ g_{\alpha}(\lambda) \lambda }^{s + 1/2} \cdot \sup_{\lambda \in \pSpectrum(T)} \absval{ g_{\alpha}(\lambda) }^{1/2 - s} + B \alpha^{s - 1/2} && \text{(by (R3))} \\
		& \quad \leq D^{s + 1/2} \cdot B \alpha^{s - 1/2} + B \alpha^{s - 1/2} && \text{(by (R1) and (R3))}
	\end{align*}
    which proves the claim with $\kappa_{D,B} \defeq (D + 1) B$.
\end{proof}

\subsection{Sub-exponentiality}
\label{app:subexponentiality}

An integrable real-valued random variable $\xi$ is said to be
\emph{sub-exponential} \citep[Section~1.3]{BuldyginKozachenko2000}, if it satisfies
\begin{equation*}
    \E[ e^{( t (\xi- \E[\xi] )} ] \leq e^{a^2 t^2/2}
    \quad
    \textnormal{for all } \absval{t} \leq \tau
\end{equation*}
for some $\tau >0$ and $a >0$.
It is said to satisfy the
\emph{Bernstein condition} \citep[Section~1.4]{BuldyginKozachenko2000}
with parameters $\sigma < \infty$ and $L < \infty$ if
\begin{equation}
    \label{eq:bernstein_condition_real}
    \bigabsval{ \E[ (\xi - \E[\xi])^p ] } \leq \frac{1}{2} p! \sigma^2 L^{p-2}
    \text{ for all } p \geq 2.
\end{equation}
Note that in its classical form, the real-valued Bernstein condition
is stated with parameters $\sigma$ and $L$ such that it
bounds the absolute value of the $p$-th centred moment, while
the vector-valued analogue \eqref{eq:bernstein_condition} usually
bounds the expectation of the $p$-power of the (not necessarily centred) norm
interpreted as a random variable. Hence, it may be important
to keep track of the parameters when moving from
a centered bound to an uncentered one.
The fact that the real-valued Bernstein condition implies sub-exponentiality
is well-known. We have not found a converse implication
explicitly stated in the literature, and so for completeness we
briefly deduce it from a known property of sub-exponentials.

\begin{lemma}[Bernstein condition and sub-exponentiality]
    \label{lem:bernstein_subexponential}
    Let $\xi$ be an integrable real-valued random variable.
    Then $\xi$ satisfies the Bernstein condition
    with some parameters $\sigma < \infty$ and $L < \infty$
    if and only if $\xi$ is sub-exponential.
\end{lemma}

\begin{proof}
    As previously mentioned, the fact that \eqref{eq:bernstein_condition_real} implies
    sub-exponentiality of $\xi$ can be found in the literature
    (e.g.\ \citealp{BuldyginKozachenko2000}, Theorem~4.2).
    We consider the converse implication. Under the assumption
    of sub-exponentiality, it is known that $\xi$ satisfies
    \begin{equation*}
        \sup_{p \geq 1}
        \left(
            \frac{ \E \big[ \absval{ \xi - \E[\xi] }^p \big] }{ p! }
        \right)^{ \frac{1}{p} }
        \qefed L < \infty
    \end{equation*} see e.g.\
   \citet[Theorem~3.2]{BuldyginKozachenko2000}.
   From this, we directly obtain
   \begin{align*}
       \bigabsval{ \E [ ( \xi - \E[\xi] )^p ] } \leq
       \E \big[ \absval{ \xi - \E[\xi] }^p \big] \leq p! L^p
       \quad \text{for all } p \geq 2
   \end{align*}
   which clearly implies \eqref{eq:bernstein_condition_real}
   with $\sigma^2 \defeq 2L^2$.
\end{proof}

\subsection{Concentration bounds}
\label{app:concentration_bounds}

A central ingredient for our analysis is the following Hoeffding-type bound for sub-exponential
Hilbertian random variables.

\begin{proposition}[\citealt{MaurerPontil2021}, Proposition~7(ii)]
    \label{prop:maurer_pontil}
    Let $\xi, \xi_1, \dots, \xi_n$ be i.i.d.\ random variables with joint law $\prob^{\otimes n}$ taking values in a separable Hilbert space $\cH$ such that
    $\E[ \xi ] = 0$ and $\norm{ \xi }_{ L_{\psi_1}(\prob; \cH)} < \infty$.
    Then, for all $\delta \in (0, \tfrac{1}{2}]$ and $n \geq \log(1/\delta)$, with $\prob^{\otimes n}$-probability at least $1 - \delta$,
    \begin{equation*}
    	\label{eq:maurer_pontil}
        \Norm{
        \frac{1}{n} \sum_{i = 1}^n \xi_i
        }_{\cH}
        \leq
        8 \sqrt{2} e \norm{ \xi}_{ L_{\psi_1}(\prob; \cH)}
        \sqrt{
            \frac{\log(1/\delta)}{n}
        } .
    \end{equation*}
\end{proposition}

We now use \Cref{prop:maurer_pontil} to derive several supporting results for our covariance operators.

\begin{lemma}[Sampling error of empirical covariance operators]
	\label{lem:concentration_empirical_covariance}
	Assume that $X \in L_{\psi_{2}}(\prob; \cX)$, $Y \in L_{\psi_{2}}(\prob; \cY)$, and that the data pairs $(X_{i}, Y_{i})_{i=1}^{n} \in \cX \times \cY$ are sampled i.i.d.\ from the joint law $\law(X,Y)$.
	Then, for all $\delta \in (0, \tfrac{1}{2}]$ and $n \geq \log(1/\delta)$, the bound
	\begin{equation*}
		\bignorm{ \hatCYX - \CYX }_{S_{2}(\cX, \cY)}
		\leq
		24 \sqrt{2} e \norm{ Y }_{L_{\psi_{2}}(\prob; \cY)} \norm{ X }_{L_{\psi_{2}}(\prob; \cX)} \sqrt{\frac{\log(1/\delta)}{n}}
	\end{equation*}
	holds with $\prob^{\otimes n}$-probability at least $1 - \delta$.
	An analogous bound holds for the sampling error of $\hatCXX$ when $\norm{ Y }_{L_{\psi_2}(\prob; \cY)}$ is replaced with	$\norm{ X }_{L_{\psi_2}(\prob; \cX)}$ in the right-hand side of the inequality.
\end{lemma}

\begin{proof}
	First observe that
	\begin{align*}
		& \norm{ Y \otimes X - \E [ Y \otimes X ] }_{L_{\psi_{1}}(\prob; \cY \otimes \cX)} \\
		& \quad \leq \norm{ Y \otimes X }_{L_{\psi_{1}}(\prob; \cY \otimes \cX)} + \norm{ \E [ Y \otimes X ] }_{L_{\psi_{1}}(\prob; \cY \otimes \cX)} \\
		& \quad \leq 2 \norm{ Y }_{L_{\psi_{2}}(\prob; \cY)} \norm{ X }_{L_{\psi_{2}}(\prob; \cX)} + \norm{ \E [ Y \otimes X ] }_{\cY \otimes \cX} && \text{(\Cref{lem:tensor_subexponential})} \\
		& \quad \leq 2 \norm{ Y }_{L_{\psi_{2}}(\prob; \cY)} \norm{ X }_{L_{\psi_{2}}(\prob; \cX)} + \norm{ Y }_{L^{1}(\prob; \cY)} \norm{ X }_{L^{1}(\prob; \cX)} && \text{(Cauchy--Schwarz)} \\
		& \quad \leq 3 \norm{ Y }_{L_{\psi_{2}}(\prob; \cY)} \norm{ X }_{L_{\psi_{2}}(\prob; \cX)} ,
	\end{align*}
	since the sub-Gaussian norm dominates the $L^{1}$ Bochner norm.
	Therefore, applying \citet[Proposition~7(ii)]{MaurerPontil2021}, with $\prob^{\otimes n}$-probability at least $1 - \delta$,
	\begin{align*}
		\bignorm{ \hatCYX - \CYX }_{S_{2}(\cX, \cY)}
		& = \Norm{ \frac{1}{n} \sum_{i = 1}^{n} ( Y_{i} \otimes X_{i} - \E [ Y \otimes X ] ) }_{\cY \otimes \cX} \\
		& \leq 8 \sqrt{2} e \norm{ Y \otimes X - \E [ Y \otimes X ] }_{L_{\psi_{1}}(\prob; \cY \otimes \cX)} \sqrt{\frac{\log(1/\delta)}{n}} \\
		& \leq 24 \sqrt{2} e \norm{ Y }_{L_{\psi_{2}}(\prob; \cY)} \norm{ X }_{L_{\psi_{2}}(\prob; \cX)} \sqrt{\frac{\log(1/\delta)}{n}} ,
	\end{align*}
	and this establishes the claim.
\end{proof}

The bound below can be obtained straightforwardly by applying the
reverse triangle inequality to the assertion of
\Cref{lem:concentration_empirical_covariance}.

\begin{lemma}[High-probability boundedness of empirical covariance operator]
    \label{lem:concentration_norm}
    Assume that we have $X \in L_{\psi_2}(\prob; \cX)$
    and that $X_1, \dots, X_n$ are sampled i.i.d.\ from $\law(X)$.
    Then, for all $\delta \in (0, \tfrac{1}{2}]$
    and $n \geq \log(1/\delta)$,
    \begin{equation*}
        \norm{ \hatCXX }_{S_2(\cX)}
        \leq
        \norm{ \CXX }_{S_2(\cX)} +
        24 \sqrt{2} e \norm{ X }^2_{L_{\psi_2}(\prob; \cX)}
        \sqrt{
        \frac{ \log(1 / \delta)}{n}
        }
    \end{equation*}
    with $\prob^{\otimes n}$-probability at least $1-\delta$.
    In particular, if $n$ satisfies \eqref{eq:alpha}
    for some $\alpha >0$, we obtain
    \begin{equation*}
        \norm{ \hatCXX }_{S_2(\cX)}
        \leq
        \norm{ \CXX }_{S_2(\cX)} + \alpha
    \end{equation*}
    with $\prob^{\otimes n}$-probability at least $1-\delta$
    under the above assumptions.
\end{lemma}

\begin{proof}
    The first bound follows directly from
    \Cref{lem:concentration_empirical_covariance},
    the second bound follows since the constraint
    \eqref{eq:alpha} for $n$ was specifically chosen in the proof of
    \Cref{lem:concentration_empirical_idop} such that it implies
    \begin{equation}
        \label{eq:alpha_choice}
    	24 \sqrt{2} e
        \norm{ X }_{L_{\psi_{2}}(\prob; \cX)}^{2}
        \sqrt{\frac{\log(1/\delta)}{n}} \leq \alpha.
    \end{equation}
\end{proof}

\begin{lemma}[High-probability boundedness of empirical identity operators]
    \label{lem:concentration_empirical_idop}
    Suppose that we have $X \in L_{\psi_2}(\prob; \cX)$, $\alpha > 0$
    and the $X_1, \dots, X_n$ are sampled i.i.d.\ from $\law(X)$.
    Then
	\begin{equation*}
		\bignorm{ ( \hatCXX + \alpha \idop_\cX )^{-1} ( \CXX + \alpha \idop_\cX ) }_{L(\cX)} \leq 2
	\end{equation*}
    with $\prob^{\otimes n}$-probability at least $1 - \delta$, for all $n \in \Naturals$ and $\delta \in (0, \tfrac{1}{2}]$ satisfying \eqref{eq:alpha_conc}, i.e.
    \begin{equation}
		\label{eq:alpha}
		n \geq \
		\max\left\{ 1 ,
			\frac{1152 e^{2} \norm{X}_{L_{\psi_2(\prob; \cX)}}^{4}}{\alpha^2}
		\right\} \log(1/\delta) .
    \end{equation}
    The same also holds for $\bignorm{ ( \CXX + \alpha \idop_\cX )
    ( \hatCXX + \alpha \idop_\cX )^{-1} }_{L(\cX)}$.
\end{lemma}

\begin{proof}
	We use the decomposition $A^{-1} B = A^{-1} (B- A) + \idop_{\cX}$ with $A \defeq \hatCXX + \alpha \idop_\cX$ and $B \defeq \CXX + \alpha \idop_{\cX}$.
	Since
    \begin{equation*}
        \bignorm{ ( \hatCXX + \alpha \idop_{\cX} )^{-1} }_{L(\cX)} \leq \frac{1}{\alpha}
    \end{equation*}
	and since the operator norm is dominated by the Hilbert--Schmidt norm, we may write,
    appealing to \Cref{lem:concentration_empirical_covariance} for the last inequality,
    \begin{align*}
        \bignorm {(\hatCXX + \alpha \idop_\cX )^{-1}
            ( \CXX + \alpha \idop_\cX) }_{L(\cX)}
        &\leq
        \bignorm{ (\hatCXX + \alpha \idop_\cX )^{-1}
            ( \CXX - \hatCXX ) }_{L(\cX)} + 1 \\
        &\leq \frac{1}{\alpha} \bignorm{ \CXX - \hatCXX }_{L(\cX)} + 1 \\
        &\leq \frac{1}{\alpha} \bignorm{ \CXX - \hatCXX }_{S_2(\cX)} + 1 \\
        &\leq 24 \sqrt{2} \frac{e}{\alpha} \norm{ X }_{L_{\psi_{2}}(\prob; \cX)}^{2} \sqrt{\frac{\log(1/\delta)}{n}} + 1,
    \end{align*}
    holding with $\prob^{\otimes n}$-probability at least $1-\delta$, provided that $n \geq \log(1/\delta)$.
    A short calculation shows that, if $n \in \Naturals$ satisfies \eqref{eq:alpha}, then
    \begin{equation*}
    	24 \sqrt{2} \frac{e}{\alpha} \norm{ X }_{L_{\psi_{2}}(\prob; \cX)}^{2} \sqrt{\frac{\log(1/\delta)}{n}} \leq 1 .
    \end{equation*}

	The bound for $\norm{ ( \CXX + \alpha \idop_{\cX} ) ( \hatCXX + \alpha \idop_{\cX} )^{-1} }_{L(\cX)}$ follows from the previous case by taking adjoints and using the fact that $\norm{ T }_{L(\cX)} = \norm{T^{\ast}}_{L(\cX)}$.
\end{proof}

\Cref{lem:concentration_empirical_idop} yields the following result.

\begin{lemma}[Upper bounds for $\CXX^{s}$-weighted norms]
    \label{lem:weighted_norms}
    Suppose that $X_1, \dots, X_n$ are sampled i.i.d.\ from $\law(X)$.
	Let $T \in L(\cX, \cY)$, $X \in L_{\psi_2}(\prob; \cX)$, and
    $r \in [0, 1]$.
	For all $n \in \Naturals$, $\delta \in (0, 1]$ and $\alpha >0$
	satisfying \eqref{eq:alpha}, with $\prob^{\otimes n}$-probability at least $1 - \delta$,
	\begin{equation*}
		\norm{ T \CXX^{r} }_{L(\cX, \cY)}
		\leq 2^r \cdot
		\norm{ T (\hatCXX + \alpha \idop_\cX )^r }_{L(\cX, \cY)} .
	\end{equation*}
	If $T \in S_2(\cX, \cY)$, then,
	under the above conditions, with $\prob^{\otimes n}$-probability at least $1 - \delta$,
	\begin{equation*}
		\norm{ T \CXX^{r} }_{S_2(\cX, \cY)}
		\leq 2^r \cdot
		\norm{ T (\hatCXX + \alpha \idop_\cX )^r }_{S_2(\cX, \cY)} .
	\end{equation*}
\end{lemma}

\begin{proof}
	Let $T \in L(\cX)$. We write
	\begin{align*}
	T \CXX^{r}
		&= T
        (\hatCXX + \alpha \idop_\cX)^r (\hatCXX + \alpha \idop_\cX)^{-r}
		(\CXX + \alpha \idop_\cX )^{r}(\CXX + \alpha \idop_\cX)^{-r}\CXX^{r} .
	\end{align*}
	Using the Cordes inequality $\norm{T_1^r \cdot T_2^r} \leq \norm{T_1\cdot T_2}^r$, valid for self adjoint, positive semi-definite operators $T_1, T_2 \in L(\cX)$ and all $r \in [0,1]$ \citep[e.g.][]{Furuta1989}, we have
	\begin{equation*}
		\norm{(\CXX + \alpha )^{-r} \CXX^{r}}_{L(\cX)}
		\leq \norm{( \CXX + \alpha )^{-1} \CXX }_{L(\cX)}^r \leq 1,
	\end{equation*}
	Thus, we obtain from \Cref{lem:concentration_empirical_idop} that,
	with $\prob^{\otimes n}$-probability at least $1-\delta$,
	\begin{align*}
		\norm{ T \CXX^{r} }_{L(\cX, \cY)}
		&\leq
		\norm{ T (\hatCXX + \alpha \idop_\cX)^r}_{L(\cX, \cY)}
		\cdot
		\norm{ (\hatCXX + \alpha \idop_\cX )^{-1}
		( \CXX + \alpha \idop_\cX)}_{L(\cX)}^r \\
		&\leq 2^r \cdot \norm{ T (\hatCXX + \alpha \idop_\cX )^r}_{L(\cX, \cY)},
	\end{align*}
	provided that \eqref{eq:alpha} holds. This proves the first bound.
	The second bound for $T \in S_2(\cX, \cY)$
	follows analogously from considering that
	\begin{align*}
		\norm{ T \CXX^{r} }_{S_2(\cX, \cY)}
		\leq
		\norm{ T ( \hatCXX + \alpha \idop_\cX)^r}_{S_2(\cX, \cY)}
		\cdot
		\norm{ ( \hatCXX + \alpha \idop_\cX)^{-1}
		(\CXX + \alpha \idop_\cX)}_{L(\cX)}^r
	\end{align*}
	holds $\prob^{\otimes n}$-a.s.\ and
	applying \Cref{lem:concentration_empirical_idop}.
	Note in particular that the involved Hilbert--Schmidt norms
	are finite in this case, as $T \in S_2(\cX, \cY)$
	implies that $T B \in S_2(\cX,  \cY)$ for any $B \in L(\cX)$.
\end{proof}

An additional consequence of \Cref{prop:maurer_pontil} is the following result, which bounds the Tikhonov--Phillips-regularised precision norm of the residual operator that arises when the exact solution $\theta_\star$ to \eqref{eq:regression_problem} is inserted into the unregularised empirical problem $\hatCYX = \theta \hatCXX$.

\begin{lemma}[Weighted residual operator]
	\label{lem:concentration_of_odd_residual}
	Suppose that $X \in L_{\psi_{2}}(\prob; \cX)$ and $Y \in L_{\psi_{2}}(\prob; \cY)$
    and $(X_1, Y_1), \dots, (X_n, Y_n)$ are
    sampled i.i.d.\ from $\law(X, Y)$.
	Let $\theta_\star$ be a solution of \eqref{eq:regression_problem} and set
	\begin{equation*}
		B_{\psi_2} \defeq \norm{\theta_\star }_{L(\cX, \cY)}
        \norm{X}^2_{L_{\psi_2}(\prob; \cX)} +
		\norm{X}_{L_{\psi_2}(\prob; \cX)} \norm{Y}_{L_{\psi_2}(\prob; \cY)}.
	\end{equation*}
	Then, for any $\alpha > 0$, $\delta \in (0, \tfrac{1}{2}]$, and $n \geq \log(1/\delta)$, with $\prob^{\otimes n}$-probability at least $1-\delta$,
	\begin{align*}
		\bignorm{ ( \theta_\star \hatCXX - \hatCYX )
        (\CXX + \alpha \idop_\cX)^{-1/2} }_{S_2(\cX, \cY)}
		\leq
        \frac{16 \sqrt{2} e B_{\psi_2}}
        {\sqrt{\alpha}} \cdot \sqrt{\frac{\log(1/\delta)}{n}} .
	\end{align*}
\end{lemma}

\begin{proof}
	For $j = 1, \dots, n$ let
	\begin{equation*}
		\xi_j \defeq
		( \theta_\star (X_j \otimes X_j)
		- Y_j \otimes X_j )(\CXX + \alpha \idop_\cX)^{-1/2}.
	\end{equation*}
	Then $\E[\xi_j] =0$, since $\theta_\star$ as a solution of
	\eqref{eq:regression_problem} satisfies $\CYX = \theta_\star \CXX$.
	Moreover,
	\begin{equation*}
		\frac{1}{n}\sum_{j=1}^n \xi_j =
		( \theta_\star \hatCXX - \hatCYX )(\CXX + \alpha \idop_\cX)^{-1/2}.
	\end{equation*}
	We now show that the $\xi_j$ are sub-Gaussian.
	Since $\norm{(\CXX + \alpha \idop_\cX)^{-1/2} }_{L(\cX)} \leq \alpha^{-1/2}$,
	we have
	\begin{align*}
		\norm{ \xi_{j} }_{S_2(\cX,\cY)}
		& \leq \frac{1}{\sqrt \alpha} \norm{ \theta_\star ( X_j \otimes X_j)
		    - Y_j \otimes X_j  }_{S_2(\cX,\cY)} \\
		& \leq \frac{1}{\sqrt \alpha}
            \left( \norm{\theta_\star}_{L(\cX, \cY)}
            \norm{ X_j \otimes X_j}_{S_2(\cX)}
		    + \norm{Y_j \otimes X_j}_{S_2(\cX,\cY)} \right).
	\end{align*}
	Hence, applying \Cref{lem:tensor_subexponential} then gives
	\begin{align*}
		\norm{ \xi_j }_{L_{\psi_1}(\prob ; S_2(\cX,\cY ))}
		&\leq \frac{1}{\sqrt \alpha}
		    \left( \norm{\theta_\star }_{L(\cX, \cY)}
            \norm{ X_j \otimes X_j}_{L_{\psi_1}(\prob; S_2(\cX ))}
		    + \norm{Y_j \otimes X_j }_{L_{\psi_1}(\prob ; S_2(\cX,\cY ))}  \right) \\
		&\leq \frac{2}{\sqrt \alpha}
		    \left( \norm{\theta_\star }_{L(\cX, \cY)} \norm{X_j}^2_{L_{\psi_2}(\prob; \cX)}
		    +  \norm{X_j}_{L_{\psi_2}(\prob; \cX)} \norm{Y_j}_{L_{\psi_2}(\prob; \cY)}
		\right).
	\end{align*}
	Since the $\xi_j$ are i.i.d., the claim now follows from \Cref{prop:maurer_pontil}.
\end{proof}

\bibliographystyle{abbrvnat}
\bibliography{references}

\begin{thebibliography}{52}
\providecommand{\natexlab}[1]{#1}
\providecommand{\url}[1]{\texttt{#1}}
\expandafter\ifx\csname urlstyle\endcsname\relax
  \providecommand{\doi}[1]{doi: #1}\else
  \providecommand{\doi}{doi: \begingroup \urlstyle{rm}\Url}\fi

\bibitem[Arias et~al.(2008)Arias, Corach, and Gonzalez]{AriasEtAl2008}
M.~L. Arias, G.~Corach, and M.~C. Gonzalez.
\newblock Generalized inverses and {Douglas} equations.
\newblock \emph{Proc. Amer. Math. Soc.}, 136\penalty0 (9):\penalty0 3177--3183,
  2008.
\newblock \doi{10.1090/S0002-9939-08-09298-8}.

\bibitem[Aubin(2000)]{Aubin2000}
J.-P. Aubin.
\newblock \emph{Applied {Functional} {Analysis}}.
\newblock Pure and Applied Mathematics (New York). Wiley-Interscience, New
  York, second edition, 2000.
\newblock \doi{10.1002/9781118032725}.
\newblock With exercises by B.\ Cornet and J.-M.\ Lasry. Translated from the
  French by C.\ Labrousse.

\bibitem[Baker(1973)]{Baker1973}
C.~R. Baker.
\newblock Joint measures and cross-covariance operators.
\newblock \emph{Trans. Amer. Math. Soc.}, 186:\penalty0 273--289, 1973.
\newblock \doi{10.2307/1996566}.

\bibitem[Bauer et~al.(2007)Bauer, Pereverzev, and Rosasco]{BauerEtAl2007}
F.~Bauer, S.~Pereverzev, and L.~Rosasco.
\newblock On regularization algorithms in learning theory.
\newblock \emph{J. Complexity}, 23\penalty0 (1):\penalty0 52--72, 2007.
\newblock \doi{10.1016/j.jco.2006.07.001}.

\bibitem[Benatia et~al.(2017)Benatia, Carrasco, and Florens]{BenatiaEtAl2017}
D.~Benatia, M.~Carrasco, and J.-P. Florens.
\newblock Functional linear regression with functional response.
\newblock \emph{J. Econometrics}, 201\penalty0 (2):\penalty0 269--291, 2017.
\newblock \doi{10.1016/j.jeconom.2017.08.008}.

\bibitem[Blanchard and M\"{u}cke(2018)]{BlanchardMuecke2018}
G.~Blanchard and N.~M\"{u}cke.
\newblock Optimal rates for regularization of statistical inverse learning
  problems.
\newblock \emph{Found. Comput. Math.}, 18:\penalty0 971--1013, 2018.
\newblock \doi{10.1007/s10208-017-9359-7}.

\bibitem[Bosq(2000)]{Bosq2000}
D.~Bosq.
\newblock \emph{Linear {Processes} in {Function} {Spaces}: {Theory} and
  {Applications}}, volume 149 of \emph{Lecture Notes in Statistics}.
\newblock Springer-Verlag, New York, 2000.
\newblock \doi{10.1007/978-1-4612-1154-9}.

\bibitem[Buldygin and Kozachenko(2000)]{BuldyginKozachenko2000}
V.~V. Buldygin and {\relax Yu}.~V. Kozachenko.
\newblock \emph{Metric {Characterization} of {Random} {Variables} and {Random}
  {Processes}}.
\newblock Number 188 in Translations of Mathematical Monographs. American
  Mathematical Society, Providence, R.I., 2000.
\newblock \doi{10.1090/mmono/188}.

\bibitem[Caponnetto and De~Vito(2007)]{CaponnettoDeVito2007}
A.~Caponnetto and E.~De~Vito.
\newblock Optimal rates for the regularized least-squares algorithm.
\newblock \emph{Found. Comput. Math.}, 7\penalty0 (3):\penalty0 331--368, 2007.
\newblock \doi{10.1007/s10208-006-0196-8}.

\bibitem[Carmeli et~al.(2006)Carmeli, De~Vito, and Toigo]{CarmeliEtAl2006}
C.~Carmeli, E.~De~Vito, and A.~Toigo.
\newblock Vector valued reproducing kernel {Hilbert} spaces of integrable
  functions and {Mercer} theorem.
\newblock \emph{Anal. Appl. (Singap.)}, 4\penalty0 (4):\penalty0 377--408,
  2006.
\newblock \doi{10.1142/S0219530506000838}.

\bibitem[Carmeli et~al.(2010)Carmeli, de~Vito, Toigo, and
  Umanit\`a]{CarmeliEtAl2010}
C.~Carmeli, E.~de~Vito, A.~Toigo, and V.~Umanit\`a.
\newblock Vector valued reproducing kernel {Hilbert} spaces and universality.
\newblock \emph{Anal. Appl. (Singap.)}, 8\penalty0 (1):\penalty0 19--61, 2010.
\newblock \doi{10.1142/S0219530510001503}.

\bibitem[Christmann and Steinwart(2008)]{ChristmannSteinwart2008}
A.~Christmann and I.~Steinwart.
\newblock \emph{Support {Vector} {Machines}}.
\newblock Springer, New York, 2008.
\newblock \doi{10.1007/978-0-387-77242-4}.

\bibitem[Crambes and Mas(2013)]{CrambesMas2013}
C.~Crambes and A.~Mas.
\newblock Asymptotics of prediction in functional linear regression with
  functional outputs.
\newblock \emph{Bernoulli}, 19\penalty0 (5B):\penalty0 2627 -- 2651, 2013.
\newblock \doi{10.3150/12-BEJ469}.

\bibitem[de~Hoop et~al.(2023)de~Hoop, Kovachki, Nelsen, and
  Stuart]{deHoopEtAl2023}
M.~V. de~Hoop, N.~B. Kovachki, N.~H. Nelsen, and A.~M. Stuart.
\newblock Convergence rates for learning linear operators from noisy data.
\newblock \emph{SIAM/ASA J. Uncertain. Quantif.}, 11\penalty0 (2):\penalty0
  480--513, 2023.
\newblock \doi{10.1137/21M1442942}.

\bibitem[Dicker et~al.(2017)Dicker, Foster, and Hsu]{DickerEtAl2017}
L.~H. Dicker, D.~P. Foster, and D.~Hsu.
\newblock Kernel ridge vs.\ principal component regression: minimax bounds and
  the qualification of regularization operators.
\newblock \emph{Electron. J. Stat.}, 11\penalty0 (1):\penalty0 1022--1047,
  2017.
\newblock \doi{10.1214/17-EJS1258}.

\bibitem[Diestel and Uhl(1977)]{DiestelUhl1977}
J.~Diestel and J.~J. Uhl, Jr.
\newblock \emph{Vector {Measures}}.
\newblock Number~15 in Mathematical Surveys. American Mathematical Society,
  Providence, R.I., 1977.
\newblock \doi{10.1090/surv/015}.
\newblock With a foreword by B.~J. Pettis.

\bibitem[Douglas(1966)]{Douglas1966}
R.~G. Douglas.
\newblock On majorization, factorization, and range inclusion of operators on
  {Hilbert} space.
\newblock \emph{Proc. Amer. Math. Soc.}, 17:\penalty0 413--415, 1966.
\newblock \doi{10.2307/2035178}.

\bibitem[Engl et~al.(1996)Engl, Hanke, and Neubauer]{EnglHankeNeubauer1996}
H.~W. Engl, M.~Hanke, and A.~Neubauer.
\newblock \emph{Regularization of {Inverse} {Problems}}, volume 375 of
  \emph{Mathematics and its Applications}.
\newblock Kluwer Academic Publishers Group, Dordrecht, 1996.

\bibitem[Fillmore and Williams(1971)]{Fillmore1971}
P.~A. Fillmore and J.~P. Williams.
\newblock On operator ranges.
\newblock \emph{Advances in Math.}, 7:\penalty0 254--281, 1971.
\newblock \doi{10.1016/S0001-8708(71)80006-3}.

\bibitem[Furuta(1989)]{Furuta1989}
T.~Furuta.
\newblock Norm inequalities equivalent to {L\"owner}--{Heinz} theorem.
\newblock \emph{Rev. Math. Phys.}, 1:\penalty0 135--137, 1989.
\newblock \doi{10.1142/S0129055X89000079}.

\bibitem[Gerfo et~al.(2008)Gerfo, Rosasco, Odone, Vito, and
  Verri]{GerfoEtAl2008}
L.~L. Gerfo, L.~Rosasco, F.~Odone, E.~D. Vito, and A.~Verri.
\newblock Spectral algorithms for supervised learning.
\newblock \emph{Neural Comput.}, 20\penalty0 (7):\penalty0 1873--1897, 2008.
\newblock \doi{10.1162/neco.2008.05-07-517}.

\bibitem[Gr\"{u}new\"{a}lder et~al.(2012)Gr\"{u}new\"{a}lder, Lever,
  Baldassarre, Patterson, Gretton, and Pontil]{GruenewaelderEtAl2012}
S.~Gr\"{u}new\"{a}lder, G.~Lever, L.~Baldassarre, S.~Patterson, A.~Gretton, and
  M.~Pontil.
\newblock Conditional mean embeddings as regressors.
\newblock In \emph{Proceedings of the 29th International Conference on Machine
  Learning}, pages 1803--1810, 2012.
\newblock URL \url{https://icml.cc/2012/papers/898.pdf}.

\bibitem[H\"{o}rmann and Kidzi\'{n}ski(2015)]{HoermannKidzinski2015}
S.~H\"{o}rmann and {\L}.~Kidzi\'{n}ski.
\newblock A note on estimation in {Hilbertian} linear models.
\newblock \emph{Scand. J. Stat.}, 42\penalty0 (1):\penalty0 43--62, 2015.
\newblock \doi{10.1111/sjos.12094}.

\bibitem[Horv{\'a}th and Kokoszka(2012)]{HorvathKokoszka2012}
L.~Horv{\'a}th and P.~Kokoszka.
\newblock \emph{Inference for {Functional} {Data} with {Applications}}.
\newblock Springer Series in Statistics. Springer, New York, 2012.
\newblock \doi{10.1007/978-1-4614-3655-3}.

\bibitem[Hsing and Eubank(2015)]{HsingEubank2015}
T.~Hsing and R.~Eubank.
\newblock \emph{Theoretical {Foundations} of {Functional} {Data} {Analysis},
  with an {Introduction} to {Linear} {Operators}}.
\newblock Wiley Series in Probability and Statistics. John Wiley \& Sons, Ltd.,
  Chichester, 2015.
\newblock \doi{10.1002/9781118762547}.

\bibitem[Imaizumi and Kato(2018)]{ImaizumiKato2018}
M.~Imaizumi and K.~Kato.
\newblock {PCA}-based estimation for functional linear regression with
  functional responses.
\newblock \emph{J. Multivariate Anal.}, 163:\penalty0 15--36, 2018.
\newblock \doi{10.1016/j.jmva.2017.10.001}.

\bibitem[Jin et~al.(2023)Jin, Lu, Blanchet, and Ying]{JinEtAl2022}
J.~Jin, Y.~Lu, J.~Blanchet, and L.~Ying.
\newblock Minimax optimal kernel operator learning via multilevel training.
\newblock In \emph{Eleventh International Conference on Learning
  Representations}, 2023.
\newblock URL \url{https://openreview.net/pdf?id=zEn1BhaNYsC}.

\bibitem[Kadri et~al.(2016)Kadri, Duflos, Preux, Canu, Rakotomamonjy, and
  Audiffren]{KadriEtAl2016}
H.~Kadri, E.~Duflos, P.~Preux, S.~Canu, A.~Rakotomamonjy, and J.~Audiffren.
\newblock Operator-valued kernels for learning from functional response data.
\newblock \emph{J. Mach. Learn. Res.}, 17\penalty0 (20):\penalty0 1--54, 2016.
\newblock URL \url{http://jmlr.org/papers/v17/11-315.html}.

\bibitem[Kallenberg(2021)]{Kallenberg2021}
O.~Kallenberg.
\newblock \emph{Foundations of {Modern} {Probability}}, volume~99 of
  \emph{Probability Theory and Stochastic Modelling}.
\newblock Springer, Cham, third edition, 2021.
\newblock \doi{10.1007/978-3-030-61871-1}.

\bibitem[Klebanov et~al.(2020)Klebanov, Schuster, and
  Sullivan]{KlebanovEtAl2020}
I.~Klebanov, I.~Schuster, and T.~J. Sullivan.
\newblock A rigorous theory of conditional mean embeddings.
\newblock \emph{SIAM J. Math. Data Sci.}, 2\penalty0 (3):\penalty0 583--606,
  2020.
\newblock \doi{10.1137/19M1305069}.

\bibitem[Klebanov et~al.(2021)Klebanov, Sprungk, and
  Sullivan]{KlebanovEtAl2021}
I.~Klebanov, B.~Sprungk, and T.~J. Sullivan.
\newblock The linear conditional expectation in {Hilbert} space.
\newblock \emph{Bernoulli}, 27\penalty0 (4):\penalty0 2267--2299, 2021.
\newblock \doi{10.3150/20-BEJ1308}.

\bibitem[Kostic et~al.(2022)Kostic, Novelli, Maurer, Ciliberto, Rosasco, and
  Pontil]{KosticEtAl2022}
V.~Kostic, P.~Novelli, A.~Maurer, C.~Ciliberto, L.~Rosasco, and M.~Pontil.
\newblock Learning dynamical systems via {Koopman} operator regression in
  reproducing kernel {Hilbert} spaces.
\newblock In \emph{Advances in Neural Information Processing Systems},
  volume~36. Curran Associates, Inc., 2022.
\newblock URL
  \url{https://proceedings.neurips.cc/paper\_files/paper/2022/file/196c4e02b7464c554f0f5646af5d502e-Paper-Conference.pdf}.

\bibitem[Kutta et~al.(2022)Kutta, Dierickx, and Dette]{KuttaDierickxDette2021}
T.~Kutta, G.~Dierickx, and H.~Dette.
\newblock Statistical inference for the slope parameter in functional linear
  regression.
\newblock \emph{Electron. J. Stat.}, 16\penalty0 (2):\penalty0 5980--6042,
  2022.
\newblock \doi{10.1214/22-EJS2078}.

\bibitem[Li et~al.(2022)Li, Meunier, Mollenhauer, and Gretton]{LiEtAl2022}
Z.~Li, D.~Meunier, M.~Mollenhauer, and A.~Gretton.
\newblock Optimal rates for regularized conditional mean embedding learning.
\newblock In S.~Koyejo, S.~Mohamed, A.~Agarwal, D.~Belgrave, K.~Cho, and A.~Oh,
  editors, \emph{Advances in Neural Information Processing Systems}, volume~35,
  pages 4433--4445. Curran Associates, Inc., 2022.
\newblock URL
  \url{https://proceedings.neurips.cc/paper\_files/paper/2022/file/1c71cd4032da425409d8ada8727bad42-Paper-Conference.pdf}.

\bibitem[Li et~al.(2024)Li, Meunier, Mollenhauer, and Gretton]{LiEtAl2024}
Z.~Li, D.~Meunier, M.~Mollenhauer, and A.~Gretton.
\newblock Towards optimal {Sobolev} norm rates for the vector-valued
  regularized least-squares algorithm.
\newblock \emph{J. Mach. Learn. Res.}, 25\penalty0 (181):\penalty0 1--51, 2024.
\newblock URL \url{http://jmlr.org/papers/v25/23-1663.html}.

\bibitem[Lin et~al.(2020)Lin, Rudi, Rosasco, and Cevher]{LinEtAl2020}
J.~Lin, A.~Rudi, L.~Rosasco, and V.~Cevher.
\newblock Optimal rates for spectral algorithms with least-squares regression
  over {Hilbert} spaces.
\newblock \emph{Appl. Comput. Harmon. Anal.}, 48\penalty0 (3):\penalty0
  868--890, 2020.
\newblock \doi{10.1016/j.acha.2018.09.009}.

\bibitem[Mas and Pumo(2010)]{MasPumo2010}
A.~Mas and B.~Pumo.
\newblock Linear processes for functional data.
\newblock In \emph{The Oxford Handbook of Functional Data Analysis}. Oxford
  University Press, 2010.
\newblock \doi{10.1093/oxfordhb/9780199568444.013.3}.

\bibitem[Math\'{e} and Pereverzev(2003)]{MathePereverzev2003}
P.~Math\'{e} and S.~Pereverzev.
\newblock Geometry of linear ill-posed problems in variable {Hilbert} scales.
\newblock \emph{Inverse Probl.}, 19\penalty0 (3):\penalty0 789, 2003.
\newblock \doi{10.1088/0266-5611/19/3/319}.

\bibitem[Maurer and Pontil(2021)]{MaurerPontil2021}
A.~Maurer and M.~Pontil.
\newblock Concentration inequalities under sub-{Gaussian} and sub-exponential
  conditions.
\newblock In \emph{Advances in Neural Information Processing Systems},
  volume~34, pages 7588--7597. Curran Associates, Inc., 2021.
\newblock URL
  \url{https://proceedings.neurips.cc/paper/2021/file/3e33b970f21d2fc65096871ea0d2c6e4-Paper.pdf}.

\bibitem[Meunier et~al.(2024)Meunier, Shen, Mollenhauer, Gretton, and
  Li]{Meunier2024}
D.~Meunier, Z.~Shen, M.~Mollenhauer, A.~Gretton, and Z.~Li.
\newblock Optimal rates for vector-valued spectral regularization learning
  algorithms, 2024.
\newblock URL \url{https://arxiv.org/abs/2405.14778}.

\bibitem[Mollenhauer(2022)]{Mollenhauer2022PhD}
M.~Mollenhauer.
\newblock \emph{On the {Statistical} {Approximation} of {Conditional}
  {Expectation} {Operators}}.
\newblock PhD thesis, Freie Universit\"at Berlin, 2022.
\newblock \url{http://dx.doi.org/10.17169/refubium-34182}.

\bibitem[Mollenhauer and Koltai(2020)]{MollenhauerKoltai2020}
M.~Mollenhauer and P.~Koltai.
\newblock Nonparametric approximation of conditional expectation operators,
  2020.
\newblock \arXiv{2012.12917}.

\bibitem[Park and Muandet(2020)]{ParkMuandet2020}
J.~Park and K.~Muandet.
\newblock A measure-theoretic approach to kernel conditional mean embeddings.
\newblock In H.~Larochelle, M.~Ranzato, R.~Hadsell, M.~Balcan, and H.~Lin,
  editors, \emph{Advances in Neural Information Processing Systems}, volume~33,
  pages 21247--21259. Curran Associates, Inc., 2020.
\newblock URL
  \url{https://proceedings.neurips.cc/paper/2020/file/f340f1b1f65b6df5b5e3f94d95b11daf-Paper.pdf}.

\bibitem[Pinelis and Sakhanenko(1986)]{PinelisSakhanenko1985}
I.~F. Pinelis and A.~I. Sakhanenko.
\newblock Remarks on inequalities for probabilities of large deviations.
\newblock \emph{Theory Probab. Appl.}, 30\penalty0 (1):\penalty0 143--148,
  1986.
\newblock \doi{https://doi.org/10.1137/1130013}.
\newblock English translation of \textit{Teor. Veroyatnost. i Primenen.},
  30(1):127--131, 1985.

\bibitem[Ramsay and Silverman(2005)]{RamsaySilverman2005}
J.~O. Ramsay and B.~W. Silverman.
\newblock \emph{Functional {Data} {Analysis}}.
\newblock Springer Series in Statistics. Springer, New York, second edition,
  2005.
\newblock \doi{10.1007/b98888}.

\bibitem[Reed and Simon(1980)]{ReedSimon1980}
M.~Reed and B.~Simon.
\newblock \emph{Methods of {Modern} {Mathematical} {Physics}. {I}: {Functional}
  {Analysis}}.
\newblock Academic Press, Inc. [Harcourt Brace Jovanovich, Publishers], New
  York, second edition, 1980.

\bibitem[Schuster et~al.(2012)Schuster, Kaltenbacher, Hofmann, and
  Kazimierski]{SchusterEtAl2012}
T.~Schuster, B.~Kaltenbacher, B.~Hofmann, and K.~S. Kazimierski.
\newblock \emph{Regularization {Methods} in {Banach} {Spaces}}.
\newblock De Gruyter, Berlin, Boston, 2012.
\newblock \doi{10.1515/9783110255720}.

\bibitem[Shmul$'\!$yan(1967)]{Shmulyan1967}
{\relax Yu}.~L. Shmul$'\!$yan.
\newblock Two-sided division in a ring of operators.
\newblock \emph{Math. Notes Acad. Sci. USSR}, 1:\penalty0 400--403, 1967.
\newblock \doi{10.1007/BF01094080}.
\newblock English translation of \textit{Mat.\ Zametki}, 1(5):605--610, 1967.

\bibitem[Szab{{\'o}} et~al.(2016)Szab{{\'o}}, Sriperumbudur, P{{\'o}}czos, and
  Gretton]{SzaboEtAL2016}
Z.~Szab{{\'o}}, B.~K. Sriperumbudur, B.~P{{\'o}}czos, and A.~Gretton.
\newblock Learning theory for distribution regression.
\newblock \emph{J. Mach. Learn. Res.}, 17\penalty0 (152):\penalty0 1--40, 2016.
\newblock URL \url{http://jmlr.org/papers/v17/14-510.html}.

\bibitem[Tsybakov(2009)]{Tsybakov2009}
A.~B. Tsybakov.
\newblock \emph{Introduction to {Nonparametric} {Estimation}}.
\newblock Springer Series in Statistics. Springer, New York, 2009.
\newblock \doi{10.1007/b13794}.
\newblock Revised and extended from the 2004 French original, Translated by V.
  Zaiats.

\bibitem[Vershynin()]{Vershynin2018}
R.~Vershynin.
\newblock \emph{High-{Dimensional} {Probability}: {An} {Introduction} with
  {Applications} in {Data} {Science}}.
\newblock Cambridge University Press, Cambridge.

\bibitem[Yao et~al.(2007)Yao, Rosasco, and Caponnetto]{YaoEtAl2007}
Y.~Yao, L.~Rosasco, and A.~Caponnetto.
\newblock On early stopping in gradient descent learning.
\newblock \emph{Constr. Approx.}, 26:\penalty0 289--315, 2007.
\newblock \doi{10.1007/s00365-006-0663-2}.

\end{thebibliography}
\addcontentsline{toc}{section}{References}

\end{document}